\documentclass{article}

\usepackage[letterpaper,top=2cm,bottom=2cm,left=3cm,right=3cm,marginparwidth=1.75cm]{geometry}

\usepackage{amsmath,amssymb,color}
\usepackage{array}
\usepackage{amsthm}
\usepackage{graphicx}
\usepackage{float}
\usepackage[colorlinks=true, allcolors=blue]{hyperref}
\usepackage[shortlabels]{enumitem}
\usepackage{tikz}
\usepackage{multicol}
\usepackage{pgfplots}
\pgfplotsset{compat=1.5}
\usetikzlibrary{babel}
\usetikzlibrary{calc, arrows}
\usetikzlibrary{arrows.meta}
\usetikzlibrary{cd} 
\usepackage{float} 

\theoremstyle{plain}

\newcommand{\N}{\mathbb{N}}
\newcommand{\R}{\mathbb{R}}

\newcommand{\Z}{\mathbb{Z}}
\newcommand{\Q}{\mathbb{Q}}

\newcommand{\Ind}{\bold{1}}
\newcommand{\quand}{\quad\text{and}\quad}
\newcommand{\squand}{\ \text{and}\ }
\newcommand{\quif}{\quad\text{if}\quad}
\newcommand{\ssi}{\quad\text{if and only if}\quad}
\newcommand{\lf}{\lfloor}
\newcommand{\rf}{\rfloor}
\newcommand{\ent}[1]{\left\lfloor #1\right\rfloor}

\newtheorem{theorem}{Theorem}[section]
\newtheorem{corollary}[theorem]{Corollary}
\newtheorem{lemma}[theorem]{Lemma}

\newtheorem{proposition}[theorem]{Proposition}

\theoremstyle{definition}
\newtheorem{definition}[theorem]{Definition}

\newtheorem{remark}[theorem]{Remark}
\newtheorem{example}[theorem]{Example}

\title{Rotation number and dynamics of 3-interval piecewise $\lambda$-affine contractions.}

\author{P. Guiraud, M. Hern\'andez, A. Meyroneinc and A. Nogueira}

\date{}
\AtEndDocument{
  \medskip
  \textsc{Instituto de Ingenier\'{\i}a Matem\'atica and Centro de Investigaci\'on y Modelamiento de Fen\'omenos Aleatorios - Valpara\'{\i}so, Universidad de Valpara\'{\i}so,
General Cruz 222, Valpara\'{\i}so, Chile}\\
    \textit{E-mail address}: \texttt{pierre.guiraud@uv.cl, mario.hernandez@uv.cl, arnaud.meyroneinc@uv.cl}
    
    \medskip
    \textsc{Aix Marseille Universit\'e, CNRS, Institut de Math\'ematiques de Marseille, 163 avenue de Luminy, Case 907, 13288 Marseille, Cedex 9, France}\\
    \textit{E-mail address}: \texttt{arnaldo.nogueira@univ-amu.fr}
  }

\begin{document}
\maketitle


\begin{abstract}
We consider a  family of  piecewise contractions admitting a rotation number and 
defined for every $x\in[0,1)$ by $f(x)=\lambda x + \delta + d \theta_a(x) \pmod 1$, where $\lambda\in(0,1)$, $d\in(0,1-\lambda)$, $\delta\in[0,1]$, $a\in[0,1]$ and  $\theta_a(x)=1$ if $x\geq a$ and $\theta_a(x)=0$ otherwise. 
In the special case where $a=1$, the family reduces to the well studied ``contracted rotations" $x\mapsto \lambda x + \delta \pmod 1$, which  are  2-interval piecewise $\lambda$-affine  contractions when $\delta\in(1-\lambda,1)$.
Considering $a\in(0,1)$ allows maps with an additional discontinuity, that is, $3$-interval piecewise $\lambda$-affine  contractions.  
Supposing $\lambda$ and $d$ fixed, for any $\rho\in(0,1)$ and $\alpha\in[0,1]$, we provide the values of the parameters $\delta$ and $a$ for which the corresponding map has rotation number $\rho$, and a symbolic dynamics containing that of the rotation $R_\rho:[0,1)\to[0,1)$ of angle $\rho$ with respect to the partition given by the positions of $1-\rho$ and $\alpha$ in $[0,1)$.
This enables in particular to determine the maps that have a given number of periodic orbits of an arbitrary period, or a Cantor set attractor supporting a dynamics of a given complexity.
\end{abstract}



\section{Introduction}
An interval map $f:I\to I$ is a piecewise contraction if there exist $\lambda\in(0,1)$ and a partition of $I$ in  intervals $I_1,I_2,\dots,I_N$, such that  $|f(x)-f(y)|\leq\lambda |x-y|$ for every  $x,y\in I_i$ and $i\in\{1,\dots,N\}$. 
During the last decade, several results on the asymptotic dynamics of these maps have been obtained: we can mention the existence of a decomposition of their attractor into periodic and Cantor sets components \cite{CCG20}, sharp bounds on the number of their periodic orbits \cite{B06,CCG20,GN22,NP15, NPR18},  and the possible complexity functions of their symbolic dynamics \cite{CGM20,P19}.
These results give  the general characteristics  of the asymptotic dynamics of interval piecewise contractions.

In theory, given a piecewise contraction, the number of periodic and Cantor set components of its attractor, as well as the complexity of its symbolic dynamics, can be estimated by analyzing  the recurrence properties of the set of its discontinuities \cite{CCG20,CGM20}.
In general,  there is no explicit criterion relating directly  the parameters
of the map  (such as the positions of the discontinuities or the contraction rate  on each subinterval) to the exact structure and symbolic dynamics of its attractor. But when such a criterion can be established, it facilitates the detailed description of the dynamics as a function of the parameters, as well as the construction of examples whose attractor has some desired characteristics.
It is the case   for  the 2-interval piecewise $\lambda$-affine contraction, often called contracted rotation, defined for any  
$x\in[0,1)$ by 
\begin{equation}\label{2PCM}
f(x)=\left\{\begin{array}{lcl}
\lambda x + \delta &\text{if}& 0\leq x<\eta\\
\lambda x + \delta -1 &\text{if}& \eta\leq x<1
\end{array}
\right.,
\end{equation}
where $\lambda\in(0,1)$, $\delta\in(1-\lambda,1)$ and $\eta=(1-\delta)/\lambda$.
The following results, relating the parameters values of \eqref{2PCM}  to 
the properties of its attractor, have been proved in \cite{B93,BC99,Coutinho1999,JO19,LN18} by establishing explicit links between  the rotations  and \eqref{2PCM}.
The map $f$ admits a rotation number (see \eqref{DEFRN}), that  is equal to $\rho\in(0,1)$ if and only if 
\begin{equation}\label{DELTA2PCM}
\delta\in[\delta(\rho^-),\delta(\rho)],\quad\text{where}\quad
\delta(\rho)=(1-\lambda)\left(1+\frac{1-\lambda}{\lambda}\sum_{k=1}^{\infty}\lambda^k\ent{k\rho}\right)
\end{equation}
and $\delta(\rho^-)$ is the left-sided limit of $\delta(\cdot)$ at $\rho$. 
When $\delta$ satisfies \eqref{DELTA2PCM}  for some $\rho\in(0,1)$ the asymptotic dynamics of $f$ has the following properties.
If $\rho\notin\Q$, then the closure of the attractor $\Lambda:=\bigcap_{n\in\N}f^n([0,1))$ of $f$ is a Cantor set. If $\rho=p/q\in\Q$ and $\delta\neq\delta(\rho)$, then $\Lambda$ is a $q$-periodic orbit.
In both cases, the symbolic dynamics of $f$ on the set $\Lambda$, with respect to the natural partition $\{[0,\eta),[\eta,1)\}$, is the same as that of the rotation $R_\rho$ of angle $\rho$ with respect to the partition $\{[0,1-\rho),[1-\rho,1)\}$. 
Besides, for irrational rotation number, the map restricted to the closure of $\Lambda$ is semi-conjugate to the rotation. 

The map \eqref{2PCM} has at most one periodic orbit and if its attractor is a Cantor set the code of any orbit on this set has the Sturmian complexity $p(n)=n+1$ of the irrational rotation.
This illustrates the recent theoretical results mentioned earlier. 
On the one hand, injective piecewise contractions with $N-1$ discontinuities have at most $N$ periodic orbits \cite{NP15}. If the map is right continuous and piecewise increasing, such as \eqref{2PCM}, its number of periodic orbits is bounded by $N-l$, where $l$ is the number of discontinuities whose image is equal to 0 (see \cite{GN22}).
On the other hand, for injective piecewise contractions, the complexity function of the code of any orbit (with respect to the natural partition) is an eventually affine function of slope at most $N-1$ \cite{CGM20,P19}. 
It follows that in order to observe a richer dynamics than that of \eqref{2PCM} in the class of injective piecewise contractions, it is necessary to consider maps with more discontinuities.

It is known that the bounds on the number of periodic orbits and on the growth rate of the complexity functions are sharp. 
The sharpness of the first bound is proved by exhibiting a map  with $N-1$ discontinuities and $N$ fixed points.
The proof of the sharpness of the second one is harder, since maps with a non periodic
attractor are exotic \cite{G25,GP, GN22,JL25,JO19,NPR18, P20}. 
For instance, in \cite{CGM20} the map \eqref{2PCM} is used as a base case to prove by induction the existence for any $N$ of maps with $N-1$ discontinuities and maximal complexity.
Alternatively, in \cite{P19} the proof stands on showing that for any minimal interval exchange transformation  there exists an injective piecewise contraction that is semi-conjugate to it. 
The family of the contracted rotations provides examples showing that the bounds are 
sharp for $N=2$. But for $N\geq 3$ arbitrary, examples with the maximal number of periodic orbits of any period or with a Cantor set attractor of maximal complexity remain difficult to construct. More generally, it is  still missing a parametric family for which one can observe and control the features of the asymptotic dynamics arising when the number of discontinuities is greater than one.

The  criterion \eqref{DELTA2PCM} allows on the one hand  a detailed description of the attractor of the map \eqref{2PCM} as a function of the parameter $\delta$. On the other hand, by choosing $\rho$ adequately, \eqref{DELTA2PCM} enables  to construct examples having a stable periodic orbit of an arbitrary period, or a Cantor set attractor, in both cases with the symbolic dynamics of the corresponding rotation. 
In this paper, we show that it is possible to generalize methods developed   for \eqref{2PCM} in \cite{Coutinho1999} to a larger family of maps having up to two discontinuities.
For this family, we establish a criterion like \eqref{DELTA2PCM}, which relates directly the values of the parameters to some specific properties of the attractor.
In particular, this enables  to generate a wide set of examples with $N-1=2$ discontinuities for which the theoretical bounds on the number of periodic orbits and complexity are reached.

The family we consider is constructed as follows. We can verify that the map \eqref{2PCM} 
admits a lift $F_1:\R\to\R$ defined  by 
\begin{equation}\label{L2PCM}
F_1(x)=\lambda x + \delta + (1-\lambda)\lf x \rf \quad\forall x\in\R,
\end{equation}
where $\ent{x}$ and $\{x\}$ are the integer part and the fractional part of $x$. In order to give rise to  3-interval generalizations of \eqref{2PCM}, 
we let $F:\R\to\R$ be defined  by 
\begin{equation}\label{LIFT}
F(x)=\lambda x + \delta + (1-\lambda)\lf x \rf + d\theta_a(\{x\})\quad\forall x\in\R,
\end{equation}
where $\lambda\in(0,1)$, $d\in(0,1-\lambda)$, $\delta,a\in[0,1]$ and $\theta_a$ is the step function defined by 
\begin{equation}\label{THETAA}
\theta_a(z)=\left\{\begin{array}{lcr}
0 &\text{if}& z<a\\
1 &\text{if}& z\geq a
\end{array}
\right.
\quad\forall z\in\R.
\end{equation}
We are interested in  the map defined by
\begin{equation}\label{PROJ}
f(x)=F(x)\pmod 1\quad\forall x\in[0,1).
\end{equation}
In the particular case where $a=1$ (or $d=0$), the map $f$ is given by \eqref{2PCM}.  For $a\in(0,1)$ the map $F$ has 2 discontinuities in $[0,1)$. But depending on the value of $\delta\in(1-\lambda-d,1)$, as an interval map $f$ can have 1 or 2 discontinuities and therefore be a 2-interval or a 3-interval piecewise contraction (see Section \ref{PMR} for an exhaustive study).
We can verify that $F(x+1)=F(x)+1$ for all $x\in\R$ and that $F$ is a lift for $f$. Also, $f$ admits a rotation number 
\begin{equation}\label{DEFRN}
\rho_f:=\lim_{n\to\infty}\frac{F^n(x)}{n} \pmod 1,
\end{equation}
since $F$ is monotonically increasing for $d\in(0,1-\lambda)$, which implies that \eqref{DEFRN} exists and does not depend on $x$, see \cite{Brette,RT86}.

In this paper, for each $\lambda$ and $d$ fixed, we give  the expression of  subsets $\mathcal{P}_{\rho,\alpha}$ of $[0,1]^2$, whose  union over $\alpha\in[0,1]$ provides all the values of 
$(\delta,a)$ for which  the rotation number of $f=f_{\delta,a}$ is $\rho\in(0,1)$.
The number $\alpha$ allows to control additional features of $f_{\delta,a}$ such as its symbolic dynamics, its number of periodic orbits, and the relative position of $a$ with respect to the other discontinuity. 
Indeed, 
we show that 
when $(\delta,a)$ belongs to $\mathcal{P}_{\rho,\alpha}$, the relation  $f_{\delta,a} \circ \phi_{\delta,\rho,\alpha} = \phi_{\delta,\rho,\alpha} \circ R_\rho$
holds for an explicit function $\phi_{\delta,\rho,\alpha}:[0,1)\to\Lambda=\bigcap_{n\in\N}f_{\delta,a}^n([0,1))$ which is
right continuous and non-decreasing.
It allows to prove that the symbolic dynamics of $f_{\delta,a}$ contains that of the rotation $R_\rho$, with respect to the partition in three (or two) intervals given by $1-\rho$ and $\alpha$. 
If $\{\alpha\}$ does not belong to the  orbit (backward and forward) of $1-\rho$ by $R_\rho$, then $\Lambda$ is the image of $\phi_{\delta,\rho,\alpha}$ and either it  is 
the union of 2 periodic orbits ($\rho$ rational) or its closure is a Cantor set ($\rho$ irrational) supporting a dynamics of complexity $p(n)=2n+1$. Otherwise, $\Lambda$ supports a dynamics of any eventually affine complexity (of slope 0 or 1) determined by $\rho$ and the position of $\{\alpha\}$ in the orbit of $1-\rho$.

As a consequence, we provide the values of the parameters $\delta$ and $a$ 
for which the maps \eqref{PROJ} have a given number of periodic orbits of an arbitrary period or a Cantor set attractor with a given affine  complexity:
to obtain a map $f_{\delta,a}$ whose attractor has some desired properties, it is enough to choose $\rho$ and $\alpha$ accordingly and to pick $(\delta,a)$ in the set $\mathcal{P}_{\rho,\alpha}$. The rotation number $\rho$ alone determines only the Cantor set or periodic nature of the attractor. Together, $\rho$ and $\alpha$  set the symbolic dynamics, that in turn determines the number of periodic orbits and the complexity.

We trust that our results can be generalized  to $N$-interval piecewise $\lambda$-affine contractions, provided they admit an increasing lift. 
It would give the parameters of the maps  
having any prescribed rotation number, number of periodic orbits and complexity, within the admissible values of these quantities in such a family. 
Also, every injective piecewise contraction of $N$ intervals is topologically conjugate to a piecewise affine contraction of $N$ intervals 
whose slopes are  1/2 in absolute value \cite{NP15}.
Therefore, an extension  of our work to $N$-interval piecewise $1/2$-affine contractions would also be of interest for the study of general order-preserving injective piecewise contractions as initiated in  \cite{B06}.
The approach used  in the present paper consists in establishing explicit relations between some interval exchange transformations, namely the rotations,
and the piecewise contractions \eqref{PROJ}. The study of injective maps with some decreasing branches, may require to consider more general interval exchange transformations, such as those with flips, as it is the case in \cite{FP20}.

The results on \eqref{2PCM} have been a starting point for further analysis concerning the relation between the algebraic properties of the parameters $\lambda$ and $\delta$ 
and the rational nature of the rotation number of the map \cite{LN18}. 
Also,  the transcendental nature of the elements of the Cantor attractors of the maps \eqref{2PCM} having as slope  the reciprocal of an integer has been established \cite{BuKLN}.
On the other hand, the map \eqref{2PCM}  appears in the study of geodesics of homothety surfaces \cite{BS20}, period adding bifurcations \cite{GAK17}, voting methods \cite{JO19}, biological systems  \cite{CFLM, NS72}, electronics \cite{FC91} and finance \cite{MSG}.
This is why we think that the present work, that extends the framework and results on \eqref{2PCM},  can  serve as a basis for further theoretical studies and will also be of interest for applied  fields.

The paper is organized as follows. In Section \ref{PMR}, we give the possible forms of the map \eqref{PROJ} and present the main results of the paper. In Section \ref{CGPP}, we show how to construct the function $\phi_{\delta,\rho,\alpha}$ which relates a rotation to maps of the family \eqref{PROJ} and we give its principal properties. In Sections \ref{PTH1}, \ref{PP}, \ref{PTH2} and \ref{PTH3} we provide the proofs of the main results. Finally, 
in Appendix \ref{APPA} we gather some auxiliary and complementary materials.

\section{Preliminaries and main results}\label{PMR}

\subsection{The different projections of $F$}

In this paper, we always assume that $\lambda\in(0,1)$ and $d\in(0,1-\lambda)$ are fixed. Now, according to the values of the parameters $\delta$ and $a\in[0,1]$, the map \eqref{PROJ} has different forms that we explicit in the following.

We start by noticing that for $\delta\in[0,1]\setminus(1-\lambda-d,1)$, the map $f$ has one or two fixed points for any value of $a\in[0,1]$. This implies that $\rho_f=0$ and that the $\omega$-limit set of any point is one of these fixed points. For $\delta\in(1-\lambda-d,1)$ the map can also have a fixed point, but for values of $a$ that now depend on $\delta$, see \eqref{F1} and \eqref{F2} bellow. Hence, we will suppose $\delta\in(1-\lambda-d,1)$ and let $a\in[0,1]$. 

So, let $\delta\in(1-\lambda-d,1)$ and consider the quantities
\begin{equation}\label{DETA1ETA2}
\eta_1:=\frac{1-\delta-d}{\lambda}\quand \eta_2:=\frac{1-\delta}{\lambda},
\end{equation}
which are the only two possible solutions of $F(x)=1$. Then, we obtain three different types for $f$  by letting $a\in[0,\eta_1)$, $a\in[\eta_1,\eta_2]$ or $a\in(\eta_2,1]$. The first type exists  if $\eta_1>0$ ($\delta<1-d$) and the third one  if $\eta_2<1$ ($\delta>1-\lambda$). 
As $\eta_1< 1$ and $0<\eta_2$, there is always some $a\in[0,1]\cap[\eta_1,\eta_2]$ and the second type of map exists for any $\delta\in(1-\lambda-d,1)$. 
Now, we can verify  that the three sets $M_1,M_2$ and $M_3$ defined by 
\begin{align*}
M_1:=&\{(\delta,a): \delta\in(1-\lambda-d,1-d)\squand a\in[0,\eta_1)\},\\
M_2:=&\{(\delta,a): \delta\in(1-\lambda-d,1)\squand a\in[0,1]\cap[\eta_1,\eta_2]\},\\
M_3:=&\{(\delta,a): \delta\in(1-\lambda,1)\squand a\in(\eta_2,1]\},
\end{align*}
form a partition of the set $M:=(1-\lambda-d,1)\times[0,1]$ of the values of $(\delta,a)$ considered in this paper, see Figure \ref{FPARAM}.
\begin{figure}[h]
\includegraphics[scale=0.7]{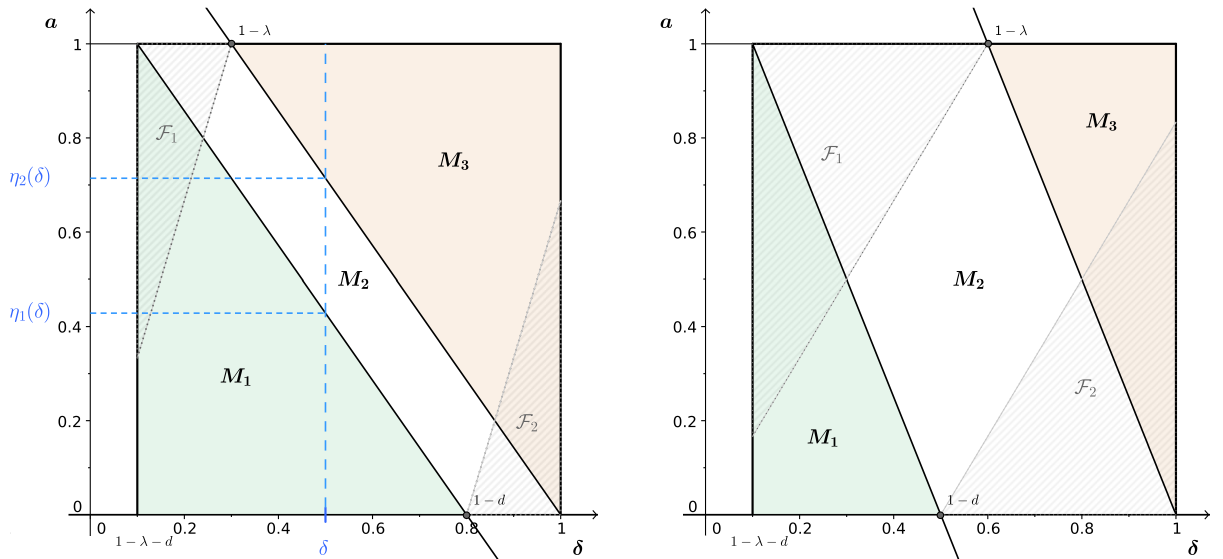} 
\centering
\caption{
 Partition of the set of the parameters $M$ for $\lambda=0.7$ and $d=0.2$ (left), and for $\lambda=0.4$ and $d=0.5$ (right). The dashed regions are the values of parameters $\mathcal{F}_1$
 and $\mathcal{F}_2$ for which the map has a fixed point, see \eqref{F1} and \eqref{F2}. 
}\label{FPARAM}
\end{figure}

The types $M_1$ and $M_3$ are of particular interest since they correspond to 
$3$-interval piecewise contractions. Indeed, for any $(\delta,a)\in M_1$ the map \eqref{PROJ} is given by
\begin{equation}\label{MAP2}
f(x)=\left\{\begin{array}{lcl}
\lambda x + \delta &\text{if}& 0\leq x<a\\
\lambda x + \delta +d &\text{if}& a\leq x<\eta_1\\
\lambda x + \delta + d -1 &\text{if}& \eta_1\leq x<1
\end{array}
\right..
\end{equation}
Also, for any $(\delta,a)\in M_3$ the map \eqref{PROJ} is given by 
\begin{equation}\label{MAP}
f(x)=\left\{\begin{array}{lcl}
\lambda x + \delta &\text{if}& 0\leq x<\eta_2\\
\lambda x + \delta -1 &\text{if}& \eta_2\leq x<a\\
\lambda x + \delta - 1 + d  &\text{if}& a\leq x<1
\end{array}
\right..
\end{equation}
We note that for $a=1$ the map \eqref{MAP} is  $x\mapsto\lambda x +\delta\pmod 1$, that is, the contracted rotation \eqref{2PCM}. In the same way, for $a=0$ the map \eqref{MAP2} is the 2-interval piecewise contraction  defined by $x\mapsto\lambda x +\delta+d\pmod 1$.  
Other 2-interval piecewise contractions  are obtained by letting $(\delta,a)\in M_2$. In this case \eqref{PROJ} writes as 
\begin{equation}\label{MAP3}
f(x)=\left\{\begin{array}{lcr}
\lambda x + \delta &\text{if}& 0\leq x<a\\
\lambda x + \delta + d -1 &\text{if}& a\leq x<1
\end{array}
\right..
\end{equation}
The graphs of the different projections of \eqref{LIFT}  are given in Figure \ref{FGRAPH}.
\begin{figure}[h]
\includegraphics[width=0.98\textwidth]{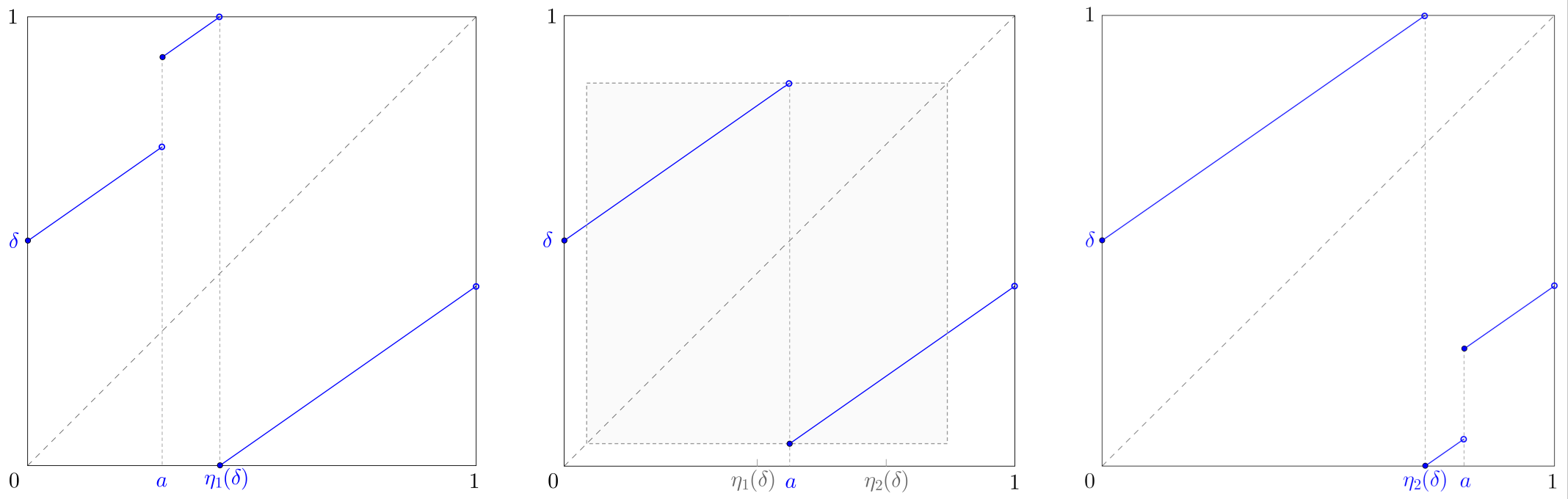}
\centering
\caption{
From left to right, an example of map of the type \eqref{MAP2}, \eqref{MAP3} and \eqref{MAP}. 
All the examples have in common the parameter values $\lambda=0.7$, $d=0.2$ and $\delta=0.5$. They correspond to three values of $(\delta, a)$ taken on the vertical dashed line  
$\delta=0.5$ of  Figure \ref{FPARAM} at $a=0.3$ for the leftmost, $a=0.5$ for the middle, and $a=0.8$ for the rightmost example.}
\label{FGRAPH}
\end{figure}

As mentioned earlier, some maps given by the set $M$ can have a (ghost) fixed point $x_*$. They are the maps such that $(\delta,a)\in\mathcal{F}_1\subset (M_1\cup M_2)$, where 
\begin{equation}\label{F1}
\mathcal{F}_1:=\left\{(\delta,a) : \delta\in(1-\lambda-d,1-\lambda]\quand a\in \left[\frac{\delta}{1-\lambda},1\right]\right\}\quad\text{with}\quad x_*=\frac{\delta}{1-\lambda},
\end{equation}
and such that $(\delta,a)\in\mathcal{F}_2\subset (M_2\cup M_3)$, where 
\begin{equation}\label{F2}
\mathcal{F}_2:=\left\{(\delta,a) : \delta\in[1-d,1)\quand a\in \left[0,\frac{\delta+d-1}{1-\lambda}\right]\right\}\quad\text{with}\quad x_*=\frac{\delta+d-1}{1-\lambda}.
\end{equation}
The sets $\mathcal{F}_1$ and $\mathcal{F}_2$ are shown  in Figure \ref{FPARAM}.
We say that $x_*$ is a ghost fixed point if $x_*=f(x^-_*)$ but  $x_*\neq f(x_*)$ or $x_*\notin f([0,1))$. Such a point exists only for the maps of  $\mathcal{F}_1$ for which  $x_*\in\{a,1\}$. In our case, a ghost fixed point is not in the image of the map, but the orbits accumulate on it and the map has zero rotation number. So, the sets $\mathcal{F}_1$ and $\mathcal{F}_2$ gather all the maps of $M$ with zero rotation number. 

\begin{remark} Note that, unless $a\in\{0,1\}$,  when seen in the circle the maps defined by the set $M$ have two discontinuities,  one at $x=0$ and the other one at $x=a$. However, as shown in Figure \ref{FGRAPH}, the map \eqref{MAP3} has an attracting invariant interval, namely $[f(a),f(a^-))$. Its restriction to this interval can be renormalized with an affine transformation to obtain the map  \eqref{2PCM}, which has one discontinuity less. 
The dynamical properties of this map can be deduced from those of \eqref{2PCM}. On the opposite, the maps \eqref{MAP2} and \eqref{MAP} can be seen as renormalizations of some maps with a discontinuity jump smaller than 1, and that are not necessarily orientation preserving, see Figure \ref{RENORM}. So, \eqref{MAP2} and \eqref{MAP} are also of interest to study some 3-interval piecewise contractions, with a discontinuity jump smaller than 1 and three discontinuities when seen in the circle.
\end{remark}

\begin{figure}[h]
\includegraphics[width=0.7\textwidth]{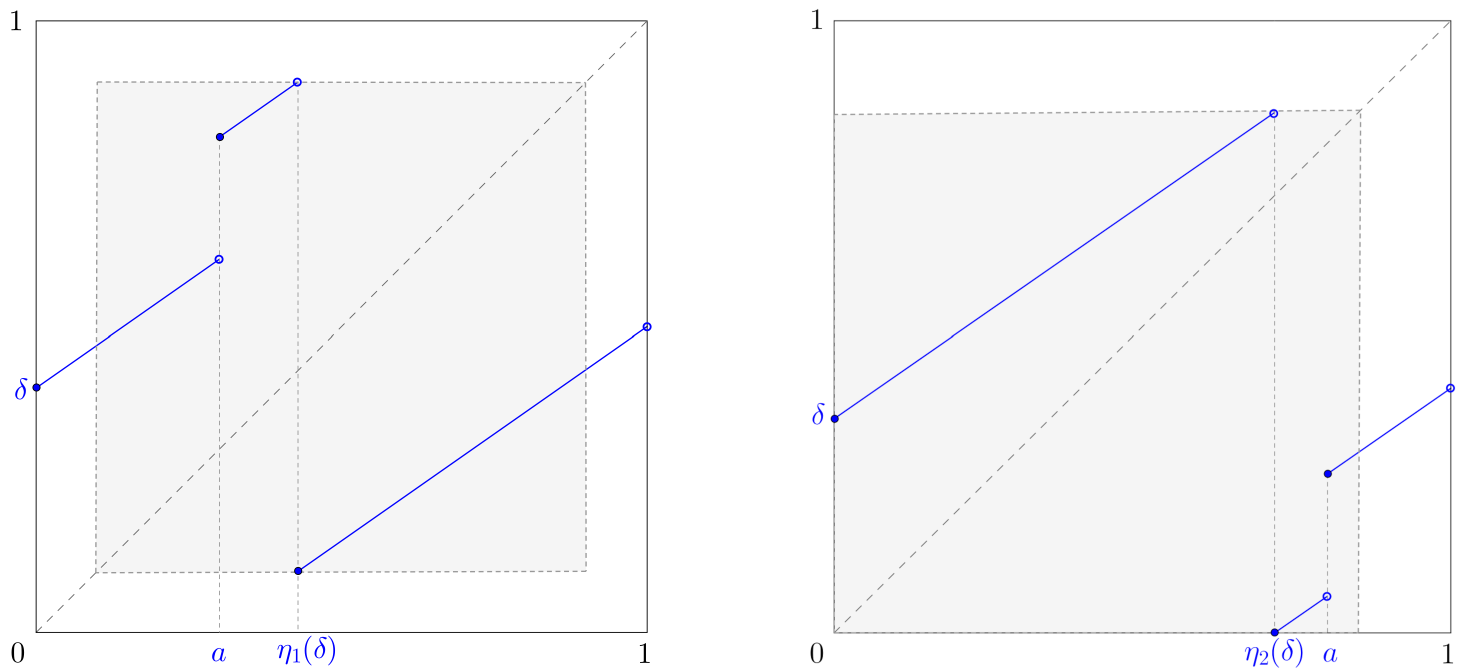} 
\centering
\caption{Examples of circle maps with three discontinuities which have a renormalization belonging to the model $M_1$ (left) and $M_3$ (right). 
}\label{RENORM}
\end{figure}

\subsection{Main results}\label{MAIN}

Our results  rely on the study of the two following functions.
\begin{definition} Let $\lambda\in(0,1)$ and $d\in(0,1-\lambda)$.
For any $\rho\in[0,1]$ and $\alpha\in[0,1]$, let 
$\delta(\rho,\alpha)$ and $a(\delta,\rho,\alpha)$ be defined by
\[
\delta(\rho,\alpha)=1-\lambda-d+\frac{1-\lambda}{\lambda}\sum_{k=1}^{\infty}\lambda^k\left((1-\lambda)\lf k\rho\rf+d\theta_{1-\alpha}(\{k\rho\})\right)
\]
and
\[
a(\delta,\rho,\alpha)=\frac{\delta}{1-\lambda}+\frac{1}{\lambda}
\sum_{k=1}^{\infty}\lambda^k\left((1-\lambda)\ent{\alpha-k\rho}+d\theta_{\alpha}(\{\alpha-k\rho\})\right)\quad\forall\delta\in[0,1].
\]
\end{definition}
Both series are uniformly convergent in any bounded set of values of $(\rho,\alpha)$. In particular, they admit a left-sided and right-sided limit at any value of $\rho\in(0,1)$, that we denote by
\[
\delta(\rho^-,\alpha):=\lim_{r\to\rho^-}\delta(r,\alpha)\quand a(\delta,\rho^+,\alpha):=\lim_{r\to\rho^+}a(\delta,r,\alpha).
\]
This notation will also be used for other functions appearing in the paper. We denote by $\Ind_{A}:\R\to\{0,1\}$  the indicator function of a set $A\subset\R$, defined for every $x\in\R$ by $\Ind_{A}(x)=1$ if $x\in A$ and $\Ind_{A}(x)=0$, otherwise.

We begin our study with the following Theorem \ref{THROT}, which allows us to find values of $\delta$ and $a$  for which the map \eqref{PROJ} has some prescribed rotation number.

\begin{theorem}\label{THROT} Suppose $\lambda\in(0,1)$ and $d\in(0,1-\lambda)$. Let  $\rho\in(0,1)$ and $\alpha\in[0,1]$. Then, for every  $\delta\in(1-\lambda-d,1)$ and $a\in[0,1]$ such that
\begin{equation}\label{CONDTH}
\delta\in[\delta(\rho^-,\alpha),\delta(\rho,\alpha)]
\quand
a\in[a(\delta,\rho^+,\alpha),a(\delta,\rho,\alpha)] 
\end{equation}
the rotation number of the map $f$ defined in \eqref{PROJ}  is $\rho$. Moreover,  
\begin{enumerate}
\item We have $[\delta(\rho^-,\alpha),\delta(\rho,\alpha)]\subset(1-\lambda-d,1)$, and if $\delta$ satisfies \eqref{CONDTH}, then $[a(\delta,\rho^+,\alpha),a(\delta,\rho,\alpha)]\cap[0,1]\neq\emptyset$ with $[a(\delta,\rho^+,\alpha),a(\delta,\rho,\alpha)]\subset [0,1]$ if $\alpha\in(0,1)$ or $\rho\notin \Q$.

\item If $\rho\in(0,1)\setminus\Q$, then 
\[
\delta(\rho^-,\alpha)=\delta(\rho,\alpha) \quand a(\delta,\rho,\alpha)=a(\delta,\rho^+,\alpha)\quad\forall\delta\in[0,1]
\]
if $k\rho\notin\Z+\alpha$ for every $k\in\Z^*$ and otherwise 
either $\delta(\rho^{-},\alpha)<\delta(\rho,\alpha)$ or $a(\delta,\rho^+,\alpha)<a(\delta,\rho,\alpha)$.

\item If $\rho=p/q$, where $0<p<q$ are co-prime, then
\begin{equation}\label{GAPSDRAT}
\delta(\rho,\alpha)-\delta(\rho^-,\alpha)\geq\frac{\lambda^{q-1}}{1-\lambda^q}(1-\lambda)(1-\lambda-d+d\Ind_{\{0,1\}}(\alpha)),
\end{equation}
and
\begin{equation}\label{GAPSARAT}
a(\delta,\rho,\alpha)-a(\delta,\rho^+,\alpha)\geq\frac{\lambda^{q-1}}{1-\lambda^q}((1-\lambda-d)\Ind_{\{0,1\}}(\alpha)+d)\quad\forall\delta\in[0,1].
\end{equation}
Moreover,  \eqref{GAPSDRAT} and \eqref{GAPSARAT} are equalities if and only if $rp/q\notin\Z+\alpha$ for every $r\in\{1,\dots,q-1\}$.
\end{enumerate}
\end{theorem}

Theorem \ref{THROT} can be read as follows. Let $\rho\in(0,1)$. If $\alpha$ is any number in $[0,1]$ and $\delta$ is chosen in $[\delta(\rho^-,\alpha),\delta(\rho,\alpha)]$, then, to obtain a map \eqref{PROJ} with rotation number $\rho$ it is enough to select any $a\in[a(\delta,\rho^+,\alpha),a(\delta,\rho,\alpha)]\cap[0,1]$. 
Item 1)  ensures  that the region 
\begin{equation}\label{PRA}
\mathcal{P}_{\rho,\alpha}:=\{(\delta,a)\in\R^2: \delta\in[\delta(\rho^-,\alpha),\delta(\rho,\alpha)]
\quand
a\in[a(\delta,\rho^+,\alpha),a(\delta,\rho,\alpha)]\}
\end{equation}
 defined by \eqref{CONDTH}  intersects our set of interest $M=(1-\lambda-d,1)\times[0,1]$ for every $\rho\in(0,1)$ and $\alpha\in[0,1]$, and  is actually included in $M$ unless  $\rho\in \Q$  and $\alpha\in\{0,1\}$.

As shown by items 2) and 3), the size of the intervals \eqref{CONDTH} depends on the rational or irrational nature of $\rho$, but also on the relation between $\rho$ and $\alpha$. 
For $\rho$ irrational,  the intervals \eqref{CONDTH} are both a singleton for almost all values of $\alpha$. These intervals are not singletons only if $\alpha$ is such that $\{k\rho\}=\alpha$ for some $k=k_\alpha\in\Z^*$, that is, if $\alpha$ belongs to the orbit  of $\rho$ by the irrational rotation of the same angle. For such $\alpha$,  one of the two intervals (but not both) is not a singleton. Its exact length depends on $k_\alpha$
and is given in Lemma \ref{LINTIR}. For $\rho=p/q$ rational, the intervals \eqref{CONDTH} have always positive Lebesgue measure. The inequalities \eqref{GAPSDRAT} and \eqref{GAPSARAT} become equalities if $\alpha$ does not belong to the $q$-periodic orbit of $\rho$ by the rotation of the same angle. 
In particular, if $p$ and $q$ are coprime, the size of both intervals is a function of $\alpha$ which is constant in $[0,1]\setminus\{0,1/q,2/q,\dots,1\}$. Now, for $\alpha\in\{0,1/q,2/q,\dots,1\}$ the intervals are bigger and their exact length, which depends on $\alpha$, is given in Lemma \ref{LINTRAT}. 

If $\rho=p/q$ is rational, the series defining $\delta(\rho,\alpha)$ and $a(\delta,\rho,\alpha)$ are finite sums given in Lemma \ref{DELT_ARAT}, which can be numerically computed with a high precision.   Also, $\delta(\rho^-,\alpha)$ and $a(\delta,\rho^+,\alpha)$ are deduced from item 3), or from Lemma \ref{LINTRAT} if $rp/q\in\Z+\alpha$ for some $r\in\{1,\dots,q-1\}$. As an example, we carry out the numerical computation of $\mathcal{P}_{\rho,\alpha}$ for $\rho=1/3$ when $\alpha=1/2$ and plot it in the left panel of Figure \ref{RHO13}.
Since $a(\delta,\rho,\alpha)$ is an affine function of $\delta$ and  $a(\delta,\rho,\alpha)-a(\delta,\rho^+,\alpha)$ does not depend on $\delta$ (see item 3) or Lemma \ref{LINTRAT}), the region obtained is a parallelogram. Now, as $\delta(p/q,\alpha)$ and $a(\delta,p/q,\alpha)$ are piecewise constant functions of $\alpha$, see Lemma \ref{DELT_ARAT},
if $\alpha$ takes all the values in $[0,1]$, then \eqref{CONDTH} gives only a finite number of different regions $\mathcal{P}_{\rho,\alpha}$ for which the rotation number is $p/q$. If  $p$ and $q$ are coprime,  this number is equal to $2q+1$ and the regions are obtained by letting $\alpha=l/q$ for every  
$l\in\{0,1,\dots,q\}$ and by choosing any $\alpha\in(l/q,(l+1)/q)$ for every $l\in\{0,1,\dots,q-1\}$. In the right panel of Figure \ref{RHO13}, the $2q+1=7$ regions $\mathcal{P}_{\rho,\alpha}$ corresponding to $\rho=1/3$ are shown.  We note  that $\mathcal{P}_{\rho,l/q}$ and $\mathcal{P}_{\rho,(l+1)/q}$ have a non-empty intersection which is equal to the region $\mathcal{P}_{\rho,\alpha}$ obtained for any $\alpha\in(l/q,(l+1)/q)$.

\begin{figure}[H]
\includegraphics[scale=0.50]{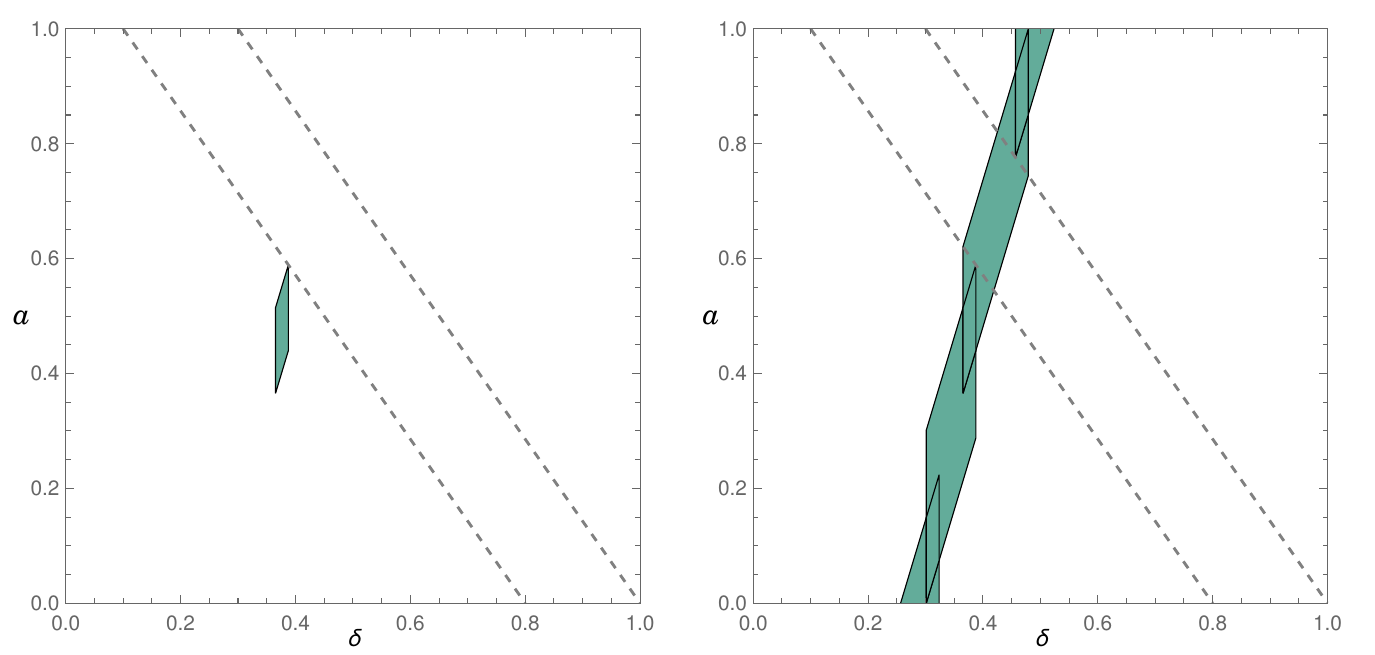} 
\centering
\caption{In green, the set $\mathcal{P}_{\rho,\alpha}$  for $\lambda=0.7, d=0.2$, $\rho=p/q=1/3$ and different values of $\alpha$. Left panel: the parallelogram $\mathcal{P}_{1/3,1/2}$. Right panel: the union of $\mathcal{P}_{1/3,\alpha}$ over 
all $\alpha\in[0,1]$. The resulting set is made of $2q+1=7$ different regions $\mathcal{P}_{1/3,\alpha}$. We note  that  $\mathcal{P}_{1/3,1/2}$ is also given by the intersection $\mathcal{P}_{1/3,1/3}\cap\mathcal{P}_{1/3,2/3}$.}
\label{RHO13}
\end{figure}

The following Proposition \ref{RANGE} shows  that for any $\rho$ we can choose 
$\alpha$ in such a way that the set of values of $\delta$ and $a$ 
given by \eqref{CONDTH} contains elements  of the desired type $M_1$, $M_2$ or $M_3$. Thus, Proposition \ref{RANGE} establishes that Theorem \ref{THROT} allows us to obtain maps of any type $M_1, M_2$ or $M_3$ with any desired rotation number.

\begin{proposition}\label{RANGE} 
Let $\lambda\in(0,1)$ and $d\in(0,1-\lambda)$. Let  $\rho\in(0,1)$, $\alpha\in[0,1]$ and choose $\delta\in[\delta(\rho^-,\alpha),\delta(\rho,\alpha)]$. 
If $\rho\notin\Q$, then for every $a\in[a(\delta,\rho^+,\alpha),a(\delta,\rho,\alpha)]$ the couple  $(\delta,a)$ belongs to $M_1$ if $\alpha\in[0,1-\rho)$, to $M_2$ if $\alpha=1-\rho$ and  to $M_3$  if $\alpha\in(1-\rho,1]$. If $\rho\in\Q$, then there exists  $a\in[a(\delta,\rho^+,\alpha),a(\delta,\rho,\alpha)]$ such that $(\delta,a)$ belongs to $M_1$ if $\alpha\in[0,1-\rho)$, to $M_2$ if $\alpha=1-\rho$ and  to $M_3$  if $\alpha\in(1-\rho,1]$. Precisely, 
\begin{enumerate}
\item If $\alpha\in[0,1-\rho)$, then $\delta\in (1-\lambda-d,1-d)$. Moreover, 
\begin{enumerate}
\item For $\alpha\neq 0$  
\begin{equation}
[a(\delta,\rho^+,\alpha),a(\delta,\rho,\alpha)]\subset[0,\eta_1], 
\end{equation}
with $a(\delta,\rho^+,\alpha)=0$  only if $\rho\in\Q$ and $\delta=\delta(\rho^-,\alpha)$, and 
with $a(\delta,\rho,\alpha)=\eta_1$ only if $\rho\in\Q$ and  $\delta=\delta(\rho,\alpha)$. 

\item For $\alpha=0$, 
\[
a(\delta,\rho^+,\alpha)=a(\delta,\rho,\alpha)=0 \quif \rho\notin\Q \quand a(\delta,\rho^+,\alpha)\leq 0\leq a(\delta,\rho,\alpha)\leq\eta_1\quif \rho\in\Q,
\]
where the last inequality is strict if $\delta\neq\delta(\rho,\alpha)$.
\end{enumerate}
\item If $\alpha\in(1-\rho,1]$, then $\delta\in (1-\lambda,1)$. Moreover, 
\begin{enumerate}
\item For $\alpha\neq 1$  
\begin{equation}
[a(\delta,\rho^+,\alpha),a(\delta,\rho,\alpha)]\subset[\eta_2,1], 
\end{equation}
with $a(\delta,\rho^+,\alpha)=\eta_2$ only if $\rho\in\Q$ and $\delta=\delta(\rho^-,\alpha)$, and with $a(\delta,\rho,\alpha)=1$ only if $\rho\in\Q$ and  $\delta=\delta(\rho,\alpha)$. 
\item For $\alpha=1$, 
\[
a(\delta,\rho^+,\alpha)=a(\delta,\rho,\alpha)=1 \quif \rho\notin\Q \quand \eta_2\leq a(\delta,\rho^+,\alpha)\leq 1\leq a(\delta,\rho,\alpha)\quif \rho\in\Q,
\]
where the first inequality is strict if $\delta\neq\delta(\rho^-,\alpha)$ and the second one if $\delta\neq\delta(\rho,\alpha)$.
\end{enumerate}
\item If $\alpha=1-\rho$, then $\delta\in (1-\lambda-d,1)$. Moreover,
\[
a(\delta,\rho^+,\alpha)=a(\delta,\rho,\alpha)\in[\eta_1,\eta_2] \quif \rho\notin\Q
\quand
[a(\delta,\rho^+,\alpha),a(\delta,\rho,\alpha)]\cap[\eta_1,\eta_2]\neq\emptyset\quif \rho\in\Q,
\]
with $a(\delta,\rho,\alpha)= \eta_1$  if and only if $\delta=\delta(\rho^-,\alpha)$ and $a(\delta,\rho^+,\alpha)=\eta_2$  if and only if $\delta=\delta(\rho,\alpha)$.
\end{enumerate}
\end{proposition}

The proposition ensures the following properties to the regions $\mathcal{P}_{\rho,\alpha}$.
Suppose $\alpha\in[0,1-\rho)$. If $\rho\notin\Q$, the region $\mathcal{P}_{\rho,\alpha}$ contains exclusively values of $(\delta,a)$ corresponding to a map of the type  \eqref{MAP2}. 
If $\rho\in\Q$ and $\alpha\neq 0$, the same remains true, excepted possibly for the value of $(\delta,a)$ given by the (right) extremity of each interval \eqref{CONDTH}. 
This value of $(\delta,a)$ can belong to $M_2$ ($a\in[\eta_1,\eta_2]$) instead of $M_1$ ($a<\eta_1$) and in such a case the corresponding map is of the type \eqref{MAP3}. 
If $\alpha=0$, the region $\mathcal{P}_{\rho,\alpha}$ is not necessarily contained in $M$, but its intersection with $M$ is nonempty and has the same properties as the regions with $\alpha\neq 0$.
The same holds for $\alpha\in(1-\rho,1]$, replacing \eqref{MAP2} by \eqref{MAP}, 
$\alpha\neq 0$ by $\alpha\neq 1$, (right) by (left), $M_1$ ($a<\eta_1$) by $M_3$ ($a>\eta_2$) and $\alpha= 0$ by $\alpha= 1$. Finally, for $\alpha=1-\rho$ and $\rho\notin\Q$ the region $\mathcal{P}_{\rho,\alpha}$ contains exclusively values of $(\delta,a)$ corresponding to a map of the type \eqref{MAP3}. If  $\rho\in\Q$, then $\mathcal{P}_{\rho,\alpha}$ still contains such values of $(\delta,a)$, but also intersects a large domain of $M_1$ and/or $M_3$. All these features can be observed in Figure \ref{FIGPROP}.

\begin{figure}[H]
\includegraphics[scale=0.55]{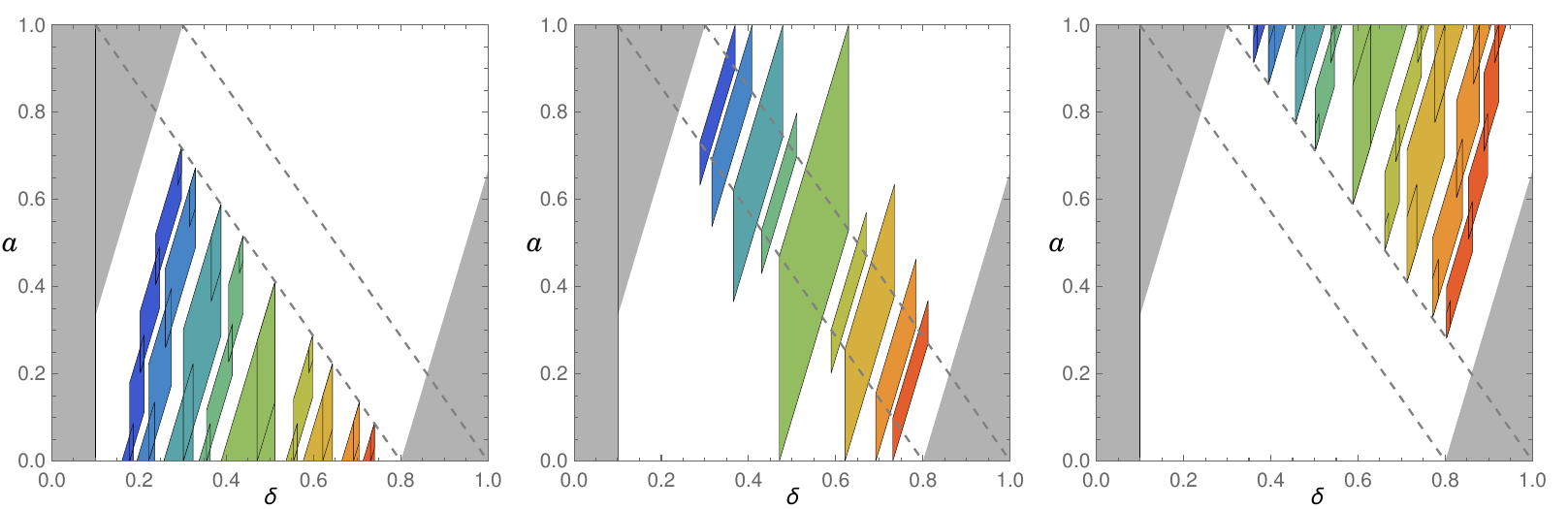} 
\centering
\caption{Plot of the $\mathcal{P}_{\rho,\alpha}$ for $\lambda=0.7$, $d=0.2$ and $\rho=1/5, 1/4, 1/3, 2/5$, $1/2, 3/5, 2/3, 3/4, 4/5$ (from left in blue to right in red), with $\alpha \in [0,1-\rho)$ (left panel)
$\alpha = 1-\rho$ (middle panel) and $\alpha \in (1-\rho,1]$ (right panel).
}\label{FIGPROP}
\end{figure}

\begin{remark}\label{RTHROTA1} Items $1(b)$ and $2(b)$ of Proposition \ref{RANGE}  allow to obtain the simplified formulation of the Theorem \ref{THROT} for the 2-interval piecewise contractions corresponding to $a=0$ and $a=1$. For instance, let us suppose that $a=1$, that is, $f$ is the contracted rotation  \eqref{2PCM}. By Item $2 (b)$,  for any $\rho\in(0,1)$ we have $a=1\in[a(\delta,\rho^+,1),a(\delta,\rho,1)]$, provided $\delta\in[\delta(\rho^-,1),\delta(\rho,1)]$. 
In other words, it is enough that $\delta\in[\delta(\rho^-,1),\delta(\rho,1)]$ for some $\rho\in(0,1)$, to ensure that $\delta$ and $a$ satisfy \eqref{CONDTH} for this value of $\rho$ and $\alpha=1$.
As a consequence, for the map \eqref{2PCM}, Theorem \ref{THROT} reduces to: if $\delta\in[\delta(\rho^-,1),\delta(\rho,1)]$, then the rotation number of \eqref{2PCM} is $\rho$. This agrees with \eqref{DELTA2PCM} since $\delta(\cdot,1)$ is equal to the function $\delta(\cdot)$ defined in \eqref{DELTA2PCM}.  
\end{remark}

\begin{figure}[h!]
\includegraphics[scale=0.65]{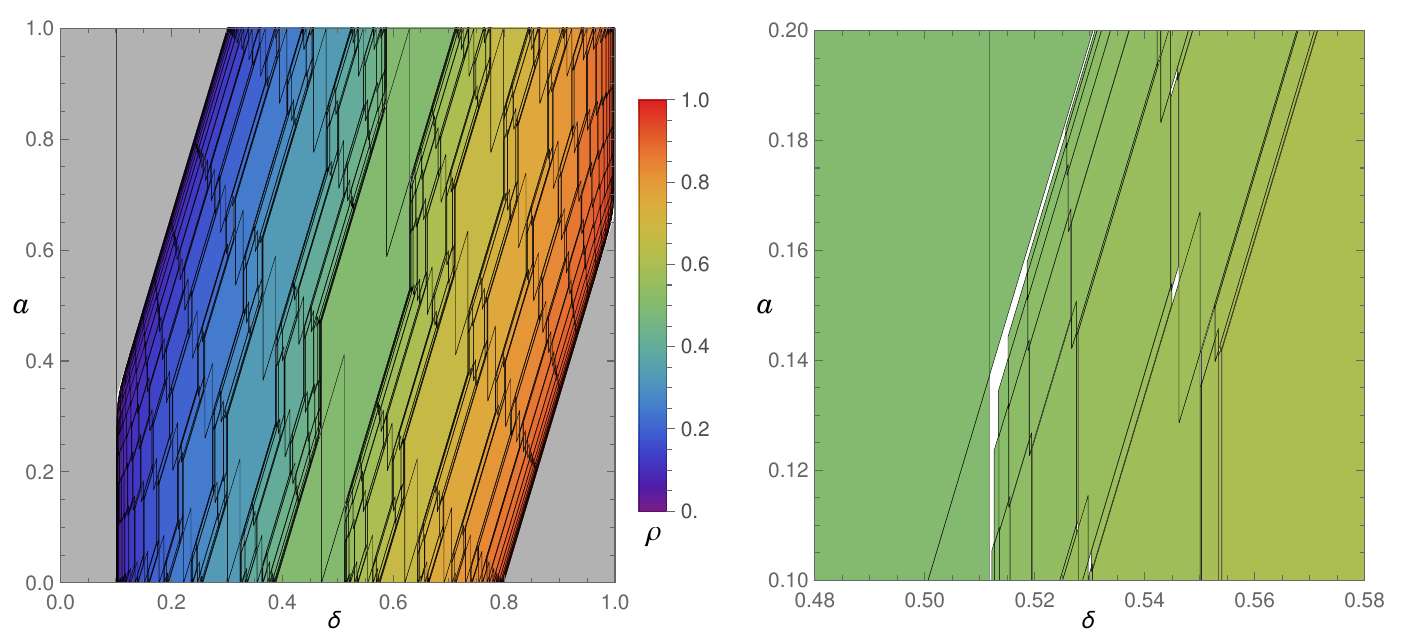} 
\centering
\caption{Left panel: Plot of the regions \eqref{CONDTH} for  every $\rho$ in the Farey sequence $F_{20}$  and every $\alpha\in [0,1]$. Right panel: zoom on a part of the left
panel showing white zones that are not covered by any of the plotted regions.}\label{FIGRHODELTAA}
\end{figure}

In the left panel of Figure \ref{FIGRHODELTAA}, we plot the union over $\alpha\in[0,1]$ of $\mathcal{P}_{\rho,\alpha}$ for a set of $127$ values of $\rho$. 
If the set $M$ could be covered using a finite set of values of $\rho$, then some values of rotation number would not be reached in the family \eqref{PROJ}, which contradicts Theorem \ref{THROT}. This is why in Figure \ref{FIGRHODELTAA} we observe  values of $(\delta,a)$ that do not belong to any of the plotted regions, see right panel. 
Therefore, the question we ask now is {\it if for any $(\delta,a)\in M$, there exist some  $\rho\in(0,1)$ and $\alpha\in[0,1]$ such that $\delta$ and $a$ satisfy \eqref{CONDTH}}. 
If the answer is in the positive, then by Theorem \ref{THROT} and unicity of the rotation number, any map $f$ of the form \eqref{PROJ} satisfies  \eqref{CONDTH} for $\rho=\rho_f$ and some $\alpha\in[0,1]$. In other words, the reciprocal of Theorem \ref{THROT} would be true.

We can already answer our question if we suppose that  $a=1$  and $\delta\in(1-\lambda,1)$. As observed in Remark \ref{RTHROTA1}, we only have to show that $\delta\in[\delta(\rho^-,1),\delta(\rho,1)]$ for some $\rho\in(0,1)$, since it will  imply that $a$ also satisfies \eqref{CONDTH} for the same $\rho$ and $\alpha=1$. We can show (see Lemma  \ref{RANGEDELTA}) that $\rho\mapsto\delta(\rho,1)$ is an increasing function such that 
\[
\lim_{\rho\to 0}\delta(\rho,1)=1-\lambda\quand \lim_{\rho\to 1^-}\delta(\rho,1)=1.
\] 
As $\delta\in(1-\lambda,1)$, this implies that $\delta$ belongs either to the image of $\rho\mapsto\delta(\rho,1)$ or to a discontinuity gap. In both cases $\delta\in[\delta(\rho^-,1),\delta(\rho,1)]$ for some $\rho\in(0,1)$, as desired.  We conclude that for any $\delta\in(1-\lambda,1)$ and $a=1$, there exists $\rho\in(0,1)$ and $\alpha=1$ such that $\delta$ and $a$ satisfy \eqref{CONDTH}. 
As a consequence, the reciprocal of Theorem \ref{THROT} is true for the subclass \eqref{2PCM}. So, we have re-proved that the rotation number of \eqref{2PCM} is $\rho$ if and only if $\delta\in[\delta(\rho^-,1),\delta(\rho,1)]$ (see Remark \ref{RTHROTA1}).

We have been able to answer our question for the maps \eqref{2PCM} because for $a=1$ the suitable  value $\alpha=1$ is suggested by Proposition \ref{RANGE}. The same can be done for the maps with  $a=0$, but it is more involved if $a\in(0,1)$. 
For any given $\delta\in(1-\lambda-d,1)$ and any $\alpha\in(0,1)$, there is some $\rho\in(0,1)$ such that $\delta\in[\delta(\rho^-,\alpha),\delta(\rho,\alpha)]$, see Lemma \ref{RANGEDELTA}. But, within all these couples of $\rho$ and $\alpha$, there 
is a priori no guaranty that $a\in[a(\delta,\rho^+,\alpha),a(\delta,\rho,\alpha)]$ for one of them.
The following Theorem \ref{RECTHROT} gives such a guaranty and states also  that the reciprocal of Theorem \ref{THROT} is true.

\begin{theorem}\label{RECTHROT} Let $\lambda\in(0,1)$ and $d\in(0,1-\lambda)$. For any $\delta\in(1-\lambda-d,1)$ and $a\in[0,1]$ such that $(\delta,a)\notin\mathcal{F}_1\cup\mathcal{F}_2$, there exist
$\rho\in(0,1)$ and $\alpha\in[0,1]$ such that $\delta$ and $a$ satisfy \eqref{CONDTH}. Consequently, if $\rho_f\in(0,1)$ is the rotation number of \eqref{PROJ}, then there exists $\alpha\in[0,1]$, such that 
\begin{equation}\label{INTREC}
\delta\in[\delta(\rho_f^-,\alpha),\delta(\rho_f,\alpha)]
\quand
a\in[a(\delta,\rho^+_f,\alpha),a(\delta,\rho_f,\alpha)]. 
\end{equation}
Moreover, $\alpha\in[0,1-\rho_f]$ if $(\delta,a)\in M_1$, $\alpha\in[1-\rho_f,1]$ if $(\delta,a)\in M_3$ and $\alpha=1-\rho_f$ if $(\delta,a)\in M_2$ and $\rho_f\notin\Q$ or $a\notin\{\eta_1,\eta_2\}$. 
\end{theorem}

A corollary of Theorem \ref{THROT} and its reciprocal Theorem \ref{RECTHROT} is the following.

\begin{corollary}\label{CORTH13}
Let $\lambda\in (0,1)$ and $d\in (1-\lambda,1)$. For any $\delta\in (1-\lambda-d,1)$ and $a\in [0,1]$ let $f_{\delta,a}$ be the map \eqref{PROJ}. 
\begin{enumerate}
\item For any $\rho\in(0,1)$, the rotation number of $f_{\delta,a}$ is $\rho$ if and only if $(\delta,a)\in\bigcup_{\alpha\in[0,1]}\mathcal{P}_{\rho,\alpha}$. 
\item The set $\mathcal{I}:=\{(\delta,a)\in (1-\lambda-d,1)\times [0,1]: \text{the rotation number of } f_{\delta,a} \text{ is irrational}\,\}$ is a Lebesgue measurable set.
\end{enumerate}
\end{corollary}
\begin{proof} Item 1) follows directly from \eqref{INTREC} and Theorem \ref{THROT}. Now, let $\mathcal{P}_{0,\alpha}:=\mathcal{F}_1\cup\mathcal{F}_2$ for every $\alpha\in[0,1]$.
Then, from item 1) we deduce that
\[
M\setminus \mathcal{I}=\bigcup_{\rho\in\Q\cap[0,1)}\bigcup_{\alpha\in[0,1]}(\mathcal{P}_{\rho,\alpha}\cap M).
\]
For any $\rho\in\Q\cap[0,1)$ and $\alpha\in[0,1]$ the set $\mathcal{P}_{\rho,\alpha}$ is
Lebesgue measurable. On the other hand, for any $\rho\in\Q\cap[0,1)$ the set $\{\mathcal{P}_{\rho,\alpha}:\alpha\in[0,1]\}$ is finite, since $\delta(\rho,\alpha)$ and $a(\delta,\rho,\alpha)$ are piecewise constant functions of $\alpha$, see Lemma \ref{DELT_ARAT}. It follows that $\mathcal{I}$ is the complement in $M$ of a countable union of Lebesgue measurable sets, which proves item 2).  
\end{proof}

\begin{remark} Item 2) of Theorem \ref{THROT} shows that for $\rho$ irrational and $\alpha$ in the orbit of $\rho$ by the rotation of angle $\rho$, the region $\mathcal{P_{\rho,\alpha}}$ is a segment. It follows that the set $\mathcal I$ of the values of $(\delta,a)$ for which the rotation number of $f_{\delta,a}$ is irrational  has Hausdorff dimension at least equal to 1. On the other hand, item 2) of Corollary \ref{CORTH13} together with the results of \cite{NPR18} allow to prove that this set has zero Lebesgue measure, as shown in the following.
Since we have proved that $\mathcal{I}$ is a Lebesgue measurable set, we can apply Fubini's theorem to obtain that
\begin{equation}\label{FUB}
\ell_2(\mathcal I)=\int_{\R^2} \Ind_{\mathcal I}(\delta,a) \ d\ell_2(\delta,a)=\int_{\R}\left(\int_{\R} \Ind_{\mathcal I_a}(\delta) \ d \ell_1(\delta)\right) \ d\ell_1(a),
\end{equation}
where $\ell_n$ denotes the Lebesgue measure in $\R^n$ and 
\[
\mathcal I_a:=\{\delta\in (1-\lambda-d,1): (\delta,a)\in \mathcal I\}\qquad\forall a\in [0,1].
\]
According to Theorem 1.1 of \cite{NPR18}, for any $a\in [0,1]$ fixed, the set of the values of $\delta\in (1-\lambda-d,1)$ such that $f_{\delta,a}$ is not asymptotically periodic has zero  Lebesgue measure. As $f_{\delta,a}$ has no periodic orbit if its rotation number is irrational, Theorem 1.1 of \cite{NPR18} implies that $\ell_1(\mathcal I_a)=0$ for any $a\in [0,1]$ and we deduce from \eqref{FUB} that $\ell_2(\mathcal I)=0$.
\end{remark}

Now, we establish explicit relations between the piecewise contractions of the form \eqref{PROJ} and the rigid rotations. To that aim, let us introduce the 
function $\phi_{\delta, \rho,\alpha}:\R\to\R$ defined as follows.

\begin{definition} Let $\lambda\in(0,1)$ and $d\in(0,1-\lambda)$. For any $\alpha\in[0,1]$, let $\psi_{\alpha}:\R \to\R$ be defined by
\begin{equation}\label{PSI}
\psi_{\alpha}(z)=(1-\lambda)\lf z\rf+d\theta_{\alpha}(\{z\})\qquad\forall z\in\R.
\end{equation}
For any $\delta,\rho$ and $\alpha\in[0,1]$, let $\phi_{\delta, \rho,\alpha}:\R\to\R$ be defined by
\begin{equation}\label{DEFPHI}
\phi_{\delta, \rho,\alpha}(y)=\frac{\delta}{1-\lambda}
+\frac{1}{\lambda}\sum_{k=1}^{\infty}\lambda^k\psi_{\alpha}(y-k\rho) \quad\forall y\in\R.
\end{equation}
\end{definition}

\noindent The monotonicity and continuity properties of $\phi_{\delta,\rho,\alpha}$ are given in Proposition \ref{PHISCR}. It is right continuous and non-decreasing. Besides, if $\rho\in(0,1)\setminus\Q$, then it is increasing and  admits a continuous generalized inverse.

\begin{theorem}\label{CODDING2} Let $\lambda\in(0,1)$, $d\in(0,1-\lambda)$,  $\delta\in(1-\lambda-d,1)$, $a\in[0,1]$ and denote $f_{\delta, a}$ the map \eqref{PROJ}.
Let $\rho\in(0,1)$ and $R_\rho:[0,1)\to [0,1)$ be a rotation of angle $\rho$, that is, the map defined by $R_\rho(y)=\{y+\rho\}$ for every $y\in [0,1)$. Let $\alpha\in[0,1]$, then, the following assertions are equivalent:
\begin{enumerate}
\item The parameters $\delta$ and $a$ satisfy
\begin{equation}\label{CONDTHIR}
\delta\in[\delta(\rho^-,\alpha),\delta(\rho,\alpha)]
\quand
a\in[a(\delta,\rho^+,\alpha),a(\delta,\rho,\alpha)], 
\end{equation}
with $\delta\neq \delta(\rho,\alpha)$ and $a\neq a(\delta,\rho^+,\alpha)$
whenever $\rho\in\Q$. 

\item  $\phi_{\delta, \rho,\alpha}([0,1))\subset[0,1)$ and  
\begin{equation}\label{CONJUGf}
f_{\delta,a}\circ\phi_{\delta, \rho,\alpha}=\phi_{\delta, \rho,\alpha}\circ R_\rho.
\end{equation}

\item For every $y\in[0,1)$ we have
\belowdisplayshortskip=10pt
\begin{equation}\label{CODES}
\theta_a(f^n_{\delta,a}(\phi_{\delta, \rho,\alpha}(y)))=\theta_{\alpha}(R^n_\rho(y))
\quand
\theta_\eta(f^n_{\delta,a}(\phi_{\delta, \rho,\alpha}(y)))=\theta_{1-\rho}( R^n_\rho(y))
\quad\forall n\in\N,
\end{equation}
with $\alpha\in[0,1-\rho]$ and $\eta=\eta_1$ if $(\delta,a)\in M_1$, 
$\alpha=1-\rho$ and $\eta=a$ if $(\delta,a)\in M_2$ and with 
$\alpha\in[1-\rho,1]$ and $\eta=\eta_2$ if $(\delta,a)\in M_3$. 
\end{enumerate}
Also, if one of the above assertions holds, then
\begin{equation}\label{THATT}
\phi_{\delta, \rho,\alpha}([0,1))\subset\Lambda:=\bigcap_{n\in\N}f_{\delta,a}^n([0,1)).
\end{equation}
Moreover, if $\rho=p/q$, with $0<p<q$ co-prime, then $\phi_{\delta, \rho,\alpha}([0,1))$ is a $q$-periodic orbit of $f_{\delta,a}$ if $\alpha\in\{0,\tfrac{1}{q},\tfrac{2}{q},\dots,\tfrac{q-1}{q},1\}$ and otherwise $\phi_{\delta, \rho,\alpha}([0,1))=\Lambda$ and is the union of the two $q$-periodic orbits of $f_{\delta,a}$. If $\rho\in(0,1)\setminus\Q$, then $\phi_{\delta, \rho,\alpha}([0,1))=\Lambda$ and $\bar{\Lambda}$ is a Cantor set.  
\end{theorem}

Theorem \ref{RECTHROT} ensures that any map $f=f_{\delta,a}$ of the type \eqref{MAP2}, \eqref{MAP} or \eqref{MAP3} with no (ghost) fixed point satisfies \eqref{INTREC}. 
Therefore, if for instance $f_{\delta,a}$ is of the type $M_1$, then  it satisfies \eqref{CONDTHIR}   for $\rho=\rho_f$ and an $\alpha\in[0,1-\rho_f]$, unless $\rho_f\in\Q$ and $(\delta,a)$  is such that $\delta=\delta(\rho_f,\alpha)$ or $a=a(\delta,\rho_f^+,\alpha)$. 
We believe that the maps such that $\rho_f\in\Q$ and $\delta=\delta(\rho_f,\alpha)$ or $a=a(\delta,\rho_f^+,\alpha)$ have  a ghost periodic orbit and an empty attractor $\Lambda$, as it is the case for the contracted rotation \eqref{2PCM} when $\rho_f\in\Q$ and $\delta=\delta(\rho_f,1)$ (see Theorem 7.2 in \cite{JO19}). 
Let us suppose that $f_{\delta,a}$ satisfies the item 1) of Theorem \ref{CODDING2} for some
$\rho\in(0,1)$ and some $\alpha\in[0,1]$.
Then, the map $f_{\delta,a}$ has rotation number $\rho$ (by Theorem \ref{THROT}) and the  following additional properties:

By item 2) any orbit of the rotation $R_\rho$ of angle $\rho$ is sent by $\phi_{\delta, \rho,\alpha}$ to an orbit of $f_{\delta,a}$ contained in the image of $[0,1)$ by $\phi_{\delta, \rho,\alpha}$. The relation \eqref{THATT}  ensures that this image is a subset of the attractor ${\Lambda}$ of $f_{\delta,a}$. Actually, if $\rho\notin\Q$, or $\rho\in\Q$ and
$\{\alpha\}$ is not in the orbit of $\rho$ by the rotation $R_\rho$, then  $\phi_{\delta, \rho,\alpha}([0,1))$  is exactly the attractor of $f_{\delta,a}$. In the first case its closure is a Cantor set and in the second one it is the union of two periodic orbits with the same period. For $\rho\notin\Q$ the function $\phi_{\delta,\rho,\alpha}$ admits a continuous inverse and  \eqref{CONJUGf} implies that $f_{\delta,a}$ restricted to $\Lambda$ and $\overline{\Lambda}$ is semi-conjugate to $R_\rho$, see Appendix \ref{SEMICONJ}.

Now we give the interpretation of item 3) and of the parameter $\alpha$ with an example. Let us suppose  that $(\delta,a)\in M_1$ (and therefore $\alpha\in[0,1-\rho]$ and $\eta=\eta_1(\delta)$). Then, consider the natural partition  $P_{a,\eta}$  of the phase space of $f_{\delta,a}$ defined by $I_0=[0,a)$, $I_1=[a,\eta)$ and $I_2=[\eta,1)$, and on the other hand the partition  $P'_{\alpha,1-\rho}$  of the phase space  of $R_\rho$ defined by $Y_0=[0,\alpha)$, $Y_1=[\alpha,1-\rho)$ and $Y_2=[1-\rho,1)$ (for $\alpha\in\{0,1-\rho\}$, the interval $Y_0$ or $Y_1$ is empty). Now, let $y\in[0,1)$ and define the code of the orbit of $y$ by $R_\rho$ with respect to 
$P'_{\alpha,1-\rho}$, as the sequence $i=\{i_n\}_{n\in\N}$ such that $R^n_\rho(y)\in Y_{i_n}$ for every $n\in\N$.  Then,  \eqref{CODES} states that $ f_{\delta,a}^n(\phi_{\delta,\rho,\alpha}(y))\in I_{i_n}$ for every $n\in\N$, that is, the sequence $i$ is also the code of an orbit of $f_{\delta,a}$ with respect to  $P_{a,\eta}$. 
So,  the set of the codes of  $f_{\delta, a}$ restricted to the image of $\phi_{\delta,\rho,\alpha}$ is equal to the set of the codes of $R_{\rho}$ with respect to $P'_{\alpha,1-\rho}$. 
The same is also true for $(\delta,a)\in M_2$ but with $\eta=a$  and $\alpha=1-\rho$, as well as  for $(\delta,a)\in M_3$ but with $\eta=\eta_2$ and the partitions $P_{\eta,a}$ and $P'_{1-\rho,\alpha}$. 

\begin{remark} As shown in the right panel of Figure \ref{RHO13}, for $\rho=p/q$ rational the values of $\delta$ and $a$ satisfying \eqref{CONDTHIR} for $\alpha\in(l/q,(l+1)/q)$, where $l\in\{0,\dots,q-1\}$, also satisfy \eqref{CONDTHIR} for $\alpha'=l/q$. It follows that \eqref{THATT} holds with $\phi_{\delta,\rho,\alpha}$ and $\phi_{\delta,\rho,\alpha'}$. As
$\phi_{\delta,\rho,\alpha}([0,1))$ is the union of two periodic orbits and $\phi_{\delta,\rho,\alpha'}([0,1))$ is a unique
periodic orbit, we deduce that the inclusion \eqref{THATT} can indeed be strict.
\end{remark}

\begin{example} As a first application of Theorem \ref{CODDING2}, we show how to construct a map of the type \eqref{MAP2} with 2 periodic orbits of period $4$. Choose any $\lambda\in(0,1)$, $d\in(0,1-\lambda)$ and let $\rho=1/4$. Now, let 
$\alpha=5/16<1-\rho$ and choose $\delta$ and $a$ in the corresponding intervals \eqref{CONDTHIR}. Then $f_{\delta,a}$ is  of the type \eqref{MAP2}, has rotation number $\rho=1/4$, an attractor composed of two periodic orbits, since $\alpha\not\in\{0,1/4,1/2,3/4,1\}$, whose codes with respect to the partition $\{[0,a),[a,\eta_1),[\eta_1,1)\}$  are  those of the orbits of the rotation  with respect to the partition $\{[0,\alpha),[\alpha,1-\rho),[1-\rho,1)\}$. A numerical example is given in Figure \ref{CONSTRUCT}.
\end{example}

\begin{figure}[h]
\includegraphics[width=0.98\textwidth]{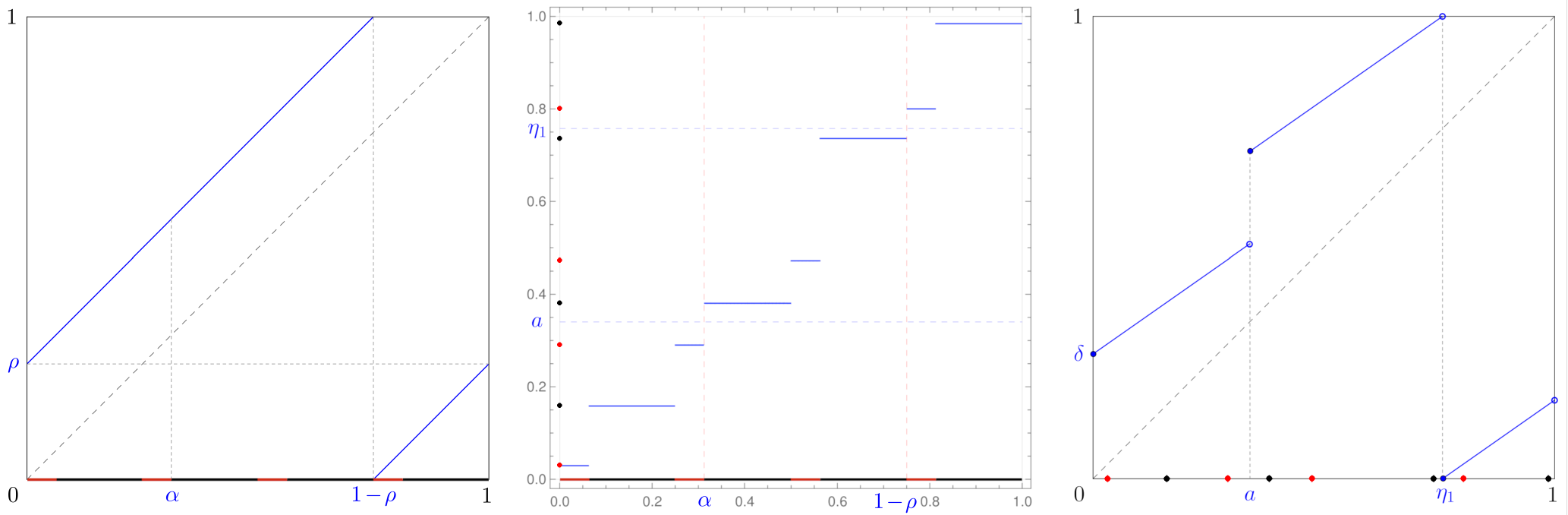} 
\centering
\caption{
Left: Rotation $R_\rho$ of angle $\rho=1/4$ with the partition $\{[0,\alpha), [\alpha,1-\rho),[1-\rho,1)\}$, where $\alpha=5/16$.
Right: The map $f_{\delta,a}$ for $\lambda=7/10$, $d=1/5$, $\delta=0.27\in[\delta(\rho^-,\alpha), \delta(\rho,\alpha))$ and $a=0.34\in(a(\delta,\rho^+,\alpha),a(\delta,\rho,\alpha)]$. The two periodic orbits composing its attractor are shown in red and black.
Center: The function $\phi_{\delta,\rho,\alpha}$.
}\label{CONSTRUCT}
\end{figure}

For any symbolic sequence $i\in\{0,1,2\}^\N$, the complexity function $p_i:\N\setminus\{0\}\to\N$ of $i$ is defined by $p_i(n)=\#\{i_ti_{t+1}\dots i_{t+n-1}:t\geq 0\}$ for every $n\geq 1$. That is, $p_i(n)$ is the number of distinct blocks of size $n$ of consecutive symbols appearing in $i$. The complexity of a periodic symbolic sequence is eventually equal to its period. The complexity of the codes of irrational rotations with respect to any given partition  can be computed thanks to  Theorem 10 of \cite{AB}. Applying this theorem, we obtain the following corollary of Theorem \ref{CODDING2}: 

\begin{corollary}\label{CORCOMP} Let $\lambda\in(0,1)$, $d\in(0,1-\lambda)$, $\rho\in(0,1)\setminus\Q$ and  $\alpha\in[0,1]$. Let $\delta\in(1-\lambda-d,1)$ and $a\in[0,1]$ satisfying \eqref{CONDTHIR} and denote $f_{\delta, a}$ the map \eqref{PROJ}. 
Let $\eta=\eta_1$ if $\alpha\in[0,1-\rho)$, $\eta=a$ if $\alpha=1-\rho$ and $\eta=\eta_2$ if $\alpha\in(1-\rho,1]$ and let $P$ be the partition of $[0,1)$ in two or three left closed intervals defined by the positions of $a$ and $\eta$ in $[0,1)$.
Let $x\in\Lambda$ and $p:\N\setminus\{0\}\to\N$ be the complexity function of the code of the orbit of $x$ with respect to $P$. Then,
\begin{enumerate}
\item $P$ is the partition given by the positions of the discontinuities of $f_{\delta,a}$ in $[0,1)$.
\item If $\{\alpha\}\neq R_\rho^k(0)$ for every $k\in\Z$, we have 
\[
p(n)=2n+1\qquad\forall n\geq 1.
\]
\item If $\{\alpha\}= R_\rho^k(0)$ for some $k\in\Z$, we have 
\[
p(n)=n+b \qquad\forall n\geq\max\{1,b-1\},  
\]
where $b=|k|$ if $k\leq -2$, $b=1$ if $k\in\{-1,0\}$ and $b=k+1$ if $k\geq 1$.
\end{enumerate}
\end{corollary}

\begin{proof} Suppose that $\delta\in(1-\lambda-d,1)$ and $a\in[0,1]$ satisfy \eqref{CONDTHIR}. As $\rho\in(0,1)\setminus\Q$, by Proposition \ref{RANGE}, the map $f_{\delta,a}$ 
is of the type \eqref{MAP2} if $\alpha\in[0,1-\rho)$, of the type \eqref{MAP3} if $\alpha=1-\rho$ and of the type \eqref{MAP} if $\alpha\in(1-\rho,1]$. Therefore, $P$ is the partition given by the positions of the discontinuities of $f_{\delta,a}$ in $[0,1)$.
Now, if $x\in\Lambda$, then $x= \phi_{\delta,\rho,\alpha}(y)$ for some $y\in[0,1)$, since 
$\phi_{\delta,\rho,\alpha}([0,1))=\Lambda$ for $\rho\in(0,1)\setminus\Q$. 
Then by item 3) of Theorem \ref{CODDING2}, the code of the orbit of $x$ by $f_{\delta,a}$ with respect to $P$ is the same as the code of the orbit of $y$ by $R_\rho$ with respect to the partition of $[0,1)$ in two or three left closed intervals defined by the positions of $\alpha$ and $1-\rho$ in $[0,1)$. 
Applying Theorem 10 of \cite{AB}, we obtain that the complexity function of such a rotation's code is given by item 2) or 3).
\end{proof}

\begin{example} Here we give an example of application of Corollary \ref{CORCOMP}. 
Choose any $\lambda\in(0,1)$, $d\in(0,1-\lambda)$ and let $\rho=\pi/4$. Now, by choosing 
a suitable $\alpha$, we can obtain maps whose symbolic dynamics has  a desired admissible complexity:   

\begin{itemize}
\item Let $\alpha=1/4\neq R^{k}_{\frac{\pi}{4}}(0)$ for every $k\in\Z$. Let $\delta=\delta(\pi/4,1/4)$ and $a=a(\delta,\pi/4,1/4)$. Then, $f_{\delta,a}$ is of the type \eqref{MAP} since $\alpha=1/4>1-\rho=(4-\pi)/4$. By items 1) and 2) of Corollary \ref{CORCOMP},  for any $x\in\Lambda$ the code of the orbit of $x$, with respect to the partition $P_{\eta_2,a}:=\{[0,\eta_2),[\eta_2,a),[a,1)\}$, has complexity  $p(n)=2n+1$ for all $n\geq 1$. 

\item Let $\alpha=R^{k}_{\frac{\pi}{4}}(0)$ for $k=100$, that is $\alpha=\{25\pi\}$. Let  $\delta=\delta(\pi/4,\{25\pi\})$ and 
\[
a\in[a(\delta,(\pi/4)^+,\{25\pi\}),a(\delta,\pi/4,\{25\pi\})]=[a(\delta,\pi/4,\{25\pi\})-\lambda^{99}(1-\lambda-d),a(\delta,\pi/4,\{25\pi\})].
\]
Then, $f_{\delta,a}$ is of the type \eqref{MAP} since $\alpha=\{25\pi\}>1-\rho=(4-\pi)/4$. By items 1) and 3) of Corollary \ref{CORCOMP}, for any $x\in\Lambda$ the code of the orbit of $x$, with respect to the partition $P_{\eta_2,a}$, has complexity  $p(n)=n+101$ for all $n\geq 100$.
\end{itemize}
\end{example}

\begin{remark} The function $\phi_{\delta,\rho,\alpha}$ allows to parametrize a large variety of Cantor sets. Indeed, choose any $\rho\in(0,1)\setminus\Q$ any $\alpha\in[0,1]$. Then the closure of $\phi_{\delta,\rho,\alpha}([0,1))$ is a Cantor set, whose extremities are $0$ and $1$ if $\delta=\delta(\rho,\alpha)$ and $\{\alpha+k\rho\}\neq 0$ for all $k\geq 1$, see Lemma \ref{PHICANT}. In such a case $\phi_{\delta,\rho,\alpha}(y)$ can be written as
\begin{align}\label{SUPERCANT}
\phi_{\delta, \rho,\alpha}(y)=\phi_{\delta(\rho,1),\rho,1}(y)+\frac{d}{\lambda}\sum_{k=1}^{\infty}\lambda^k(\theta_\alpha(\{y-k\rho\})-\theta_{\alpha}(\{-k\rho\}))\qquad \forall y\in[0,1).
\end{align}
The closure of $\phi_{\delta(\rho,1),\rho,1}([0,1))$ is the Cantor set $C_1$ corresponding to the attractor of the map $x\mapsto\lambda x+\delta\pmod 1$. In \cite{BuKLN}, it has been shown that  $C_1\setminus\{0,1\}$ contains only transcendental numbers if $\lambda=1/b$, where $b\geq 2$ is an integer. 
It would be interesting to study 
such a property for the Cantor set obtained from \eqref{SUPERCANT}  as a function of the values of $d\in(0,1-\lambda)$ and $\alpha\in(0,1)$. 
\end{remark}

In the sequel of the paper we prove the results presented above. As mentioned before, we always suppose $\lambda\in(0,1)$ and $d\in(0,1-\lambda)$, as stated in the definition $\eqref{LIFT}$ of $F$, and do not recall it from now on. Also, we will mostly write $\phi$ for $\phi_{\delta,\rho,\alpha}$ and $f$ for $f_{\delta,a}$.

\section{Construction and general properties of $\phi$}\label{CGPP}

To relate $f$ with a rotation our starting point is to construct a function $\phi:\R\to\R$ such that 
\begin{equation}\label{CONJFL}
F\circ\phi=\phi\circ L_\rho,\quad\text{where}\quad  
L_\rho(y)=y+\rho\quad\forall y\in \R
\end{equation}
is the lift of the rotation of angle $\rho$. In this section, we first show why the function $\phi=\phi_{\delta,\rho,\alpha}$ given in \eqref{DEFPHI} is a natural candidate and then we establish conditions on this function which ensure \eqref{CONJFL}.

\subsection{Construction of $\phi$}

Following the idea in \cite{Coutinho1999,LN21}, we look for an expression of the points of  $\cap_{n\in\N}F^n(\R)$. Indeed, if $\phi$ satisfies \eqref{CONJFL}  then for every $y\in\R$ we have
\[
\phi(y)=\phi(L^n_\rho(L^{-n}_\rho(y)))=F^n(\phi(L^{-n}_\rho(y)))\in F^n(\R)\quad\forall n\in\N.
\]
As $F$ is injective, the points of $\cap_{n\in\N}F^n(\R)$ are those that have a backward orbit by $F$.

\begin{lemma}\label{ORBLIFT} If $x\in\R$ has a backward orbit $(x_{-k})_{k\in\N}$ by $F$, then 
\[
x=\lf x\rf + \frac{\delta}{1-\lambda}
+\sum_{k=0}^{\infty}\lambda^k\left(\lf x_{-(k+1)}\rf-\lf x_{-k}\rf 
+d\theta_{a}(\{x_{-(k+1)}\})\right).
\] 
\end{lemma}
\begin{proof} Let us suppose that $x=x_0\in\R$ has a backward orbit and let us show by induction that for
any $n\geq 1$, we have 
\begin{equation}\label{SERIE}
x=\lf x\rf + \lambda^n\left(\{x_{-n}\}-\frac{\delta}{1-\lambda}\right)+\frac{\delta}{1-\lambda}
+\sum_{k=0}^{n-1}\lambda^k\left(\lf x_{-(k+1)}\rf-\lf x_{-k}\rf 
+d\theta_{a}(\{x_{-(k+1)}\})\right).
\end{equation}
For $n=1$, the equality holds since its r.h.s.  writes
\[
 \lambda \{x_{-1}\}-\frac{\lambda\delta-\delta}{1-\lambda}
+\lf x_{-1}\rf +d\theta_{a}(\{x_{-1}\})=\lambda x_{-1}+\delta+(1-\lambda)\ent{x_{-1}}+d\theta_{a}(\{x_{-1}\})=F(x_{-1})=x.
\] 
Now, if \eqref{SERIE} holds for some arbitrary $n\geq 1$, then we can check that it holds also for $n+1$ using  
\[
\{x_{-n}\}=F(x_{-(n+1)})-\ent{x_{-n}}=\lambda\{x_{-(n+1)}\}+\delta+\ent{x_{-(n+1)}}-\ent{x_{-n}}
+d\theta_{a}(\{x_{-(n+1)}\}).
\]
As for every $k\in\N$ we have 
\[
|\lf x_{-(k+1)}\rf-\lf x_{-k}\rf|=|\lf x_{-(k+1)}\rf-\lf F(x_{-(k+1)})\rf|
\in\{0,1\}\quand d\theta_{a}(\{x_{-(k+1)}\})\in\{0,d\},
\]
the serie converges and we can take the limit $n\to\infty$ to end the proof. 
\end{proof}

Lemma \ref{ORBLIFT} suggests that a candidate for $\phi$ is given by
\begin{equation}\label{PHIPRIM}
\phi(y):=\lf y\rf + \frac{\delta}{1-\lambda}
+\sum_{k=0}^{\infty}\lambda^k\left(\lf y-(k+1)\rho\rf-\lf y-k\rho\rf 
+d\theta_{\alpha}(\{y-(k+1)\rho\})\right).
\end{equation}
Indeed, if there exist $\rho$ and $\alpha\in[0,1]$ such that for any $y\in\R$ there is a $x\in\cap_{n\in\N}F^n(\R)$ such that $\lf x_{-k}\rf=\lf y-k\rho\rf$ and $\theta_{a}(\{x_{-k}\})=\theta_{\alpha}(\{y-k\rho\})$ for every $k\geq 0$, then $\phi(y)\in\cap_{n\in\N}F^n(\R)$ for all $y\in\R$. The function \eqref{DEFPHI} is obtained writing \eqref{PHIPRIM} as
\begin{align*}
\phi(y)
&=\lf y\rf + \frac{\delta}{1-\lambda}
+\frac{1}{\lambda}\sum_{k=1}^{\infty}\lambda^k\lf y-k\rho\rf-\lf y\rf-\sum_{k=1}^\infty\lambda^k\lf y-k\rho\rf +\frac{d}{\lambda}\sum_{k=1}^\infty\lambda^{k}\theta_{\alpha}(\{y-k\rho\})\\
&=\frac{\delta}{1-\lambda}
+\frac{1-\lambda}{\lambda}\sum_{k=1}^{\infty}\lambda^k\lf y-k\rho\rf+\frac{d}{\lambda}\sum_{k=1}^\infty\lambda^{k}\theta_{\alpha}(\{y-k\rho\})\\
&=\frac{\delta}{1-\lambda}
+\frac{1}{\lambda}\sum_{k=1}^{\infty}\lambda^k\left((1-\lambda)\lf y-k\rho\rf+d\theta_{\alpha}(\{y-k\rho\})\right)\\
&=\frac{\delta}{1-\lambda}
+\frac{1}{\lambda}\sum_{k=1}^{\infty}\lambda^k\psi_{\alpha}(y-k\rho).
\end{align*}

\subsection{Necessary and sufficient condition for $F\circ\phi=\phi\circ L_\rho$}

In this section we give some conditions ensuring that \eqref{DEFPHI} satisfies \eqref{CONJFL}.

\begin{lemma} \label{CONJUG1} For every $\rho, \alpha$ and $a\in[0,1]$, we have 
\[
F(\phi(y))= (1-\lambda)(\lf \phi(y)\rf -\ent{y})+ d(\theta_a(\{\phi(y)\})-\theta_\alpha(\{y\}))  +\phi(y+\rho)\quad \forall y\in\R.
\]
\end{lemma}

\begin{proof} For any $y\in\R$ we have
\begin{align*}
F(\phi(y))
&=  \lambda \phi(y) + \delta + (1-\lambda)\lf \phi(y) \rf + d\theta_a(\{\phi(y)\})\\
&= (1-\lambda)\lf \phi(y)\rf + d\theta_a(\{\phi(y)\}) + \frac{\delta}{1-\lambda} +\sum_{k=1}^{\infty}\lambda^{k}\psi_{\alpha}(y-k\rho)\\
&= (1-\lambda)\lf \phi(y)\rf + d\theta_a(\{\phi(y)\}) + \frac{\delta}{1-\lambda} +\sum_{k=2}^{\infty}\lambda^{k-1}\psi_{\alpha}(y-(k-1)\rho)\\
&= (1-\lambda)\lf \phi(y)\rf + d\theta_a(\{\phi(y)\}) -\psi_{\alpha}(y)+ \frac{\delta}{1-\lambda} 
+\frac{1}{\lambda}\sum_{k=1}^{\infty}\lambda^{k}\psi_{\alpha}(y+\rho -k\rho)\\
&= (1-\lambda)(\lf \phi(y)\rf -\ent{y})+ d(\theta_a(\{\phi(y)\})-\theta_\alpha(\{y\}))  +\phi(y+\rho).
\end{align*}
\end{proof}

A property of $\phi$ that we will use frequently is that 
\begin{equation}\label{TPHI}
\phi(y+l)=\phi(y)+l\qquad\forall y\in\R,
\end{equation}
for any $l\in\Z$. In particular, it implies the following.
\begin{lemma} \label{PHIEF}The assertions
\begin{enumerate}
\item $\phi([0,1))\subset[0,1)$.
\item $\lf \phi(y)\rf = \lf y\rf$ for all $y\in\R$.
\item $\{\phi(y)\}=\phi(\{y\})$ for all $y\in\R$.
\end{enumerate}
are equivalent. 
\end{lemma}

\begin{proof} By \eqref{TPHI} we have,
\[
\lf \phi(y)\rf = \lf \phi(\lf y\rf+\{y\})\rf=\lf\phi(\{y\})+\lf y\rf\rf=\lf \phi(\{y\})\rf+\lf y\rf\qquad\forall y\in\R.
\]
It follows that 1) implies 2). If 2) holds then $\lf \phi(\{y\})\rf=0$ for every $y\in\R$. In particular, for $y\in[0,1)$, we have
$\lf \phi(\{y\})\rf=\lf \phi(y)\rf=0$ and $\phi(y)\in[0,1)$.

\noindent Now, 2) is equivalent to 3), since 
\[
\{\phi(y)\}=\phi(y)- \lf \phi(y)\rf =\phi(\{y\}) + \lf y\rf -\lf \phi(y)\rf, 
\]
for every $y\in\R$.
\end{proof}

From Lemma \ref{CONJUG1} and Lemma \ref{PHIEF} we deduce the following necessary and sufficient  conditions to have $F\circ\phi=\phi\circ L_\rho$.

\begin{proposition}\label{CNSCONJ} For any $\rho\in(0,1)$, $\alpha\in[0,1]$ and $a\in[0,1]$, we have $F\circ\phi=\phi\circ L_\rho$ if and only if 
\begin{equation}\label{CNSR}
\phi([0,1))\subset[0,1) \quand \theta_\alpha(y)=\theta_a(\phi(y))\quad\forall y\in[0,1).
\end{equation}
\end{proposition}

\begin{proof} For any $y\in\R$ we have $F\circ\phi(y)=\phi\circ L_\rho(y)$ if and only if 
\begin{equation}\label{CNS}
\lf \phi(y)\rf = \lf y\rf \quand \theta_\alpha(\{y\})=\theta_a(\{\phi(y)\}).
\end{equation}
Indeed,  from Lemma \ref{CONJUG1} we deduce that $F(\phi(y))=\phi(y+\rho)$ if and only if 
\[
\frac{1-\lambda}{d}(\lf \phi(y)\rf - \lf y\rf) = \theta_\alpha(\{y\})-\theta_a(\{\phi(y)\}).
\]
This equality holds if and only if  $y$ satisfies \eqref{CNS}, since  $\lf \phi(y)\rf - \lf y\rf\in\Z$,  $\theta_\alpha(\{y\})-\theta_a(\{\phi(y)\})\in\{-1,0,1\}$ and $(1-\lambda)/d>1$.

Now, if $F\circ\phi=\phi\circ L_\rho$, then \eqref{CNS} holds for every $y\in\R$, which together with Lemma \ref{PHIEF} implies that 
\[
\phi(y)\in[0,1)\quand 
\theta_\alpha(y)=\theta_\alpha(\{y\})=\theta_a(\{\phi(y)\})=\theta_a(\phi(y)) \quad\forall y\in[0,1).
\] 
Conversely, \eqref{CNSR} together with Lemma \ref{PHIEF} implies that 
\[
 \lf \phi(y)\rf=\lf y\rf \quand
\theta_\alpha(\{y\})=\theta_a(\phi(\{y\}))=\theta_a(\{\phi(y)\})  \quad\forall y\in\R
\]
and therefore $F\circ\phi=\phi\circ L_\rho$.
\end{proof}

\subsection{Monotonicity properties of $\phi$}

Here, we study the monotonicity and continuity properties of $\phi$. Especially, we show that it   is non-decreasing, which will allow us to replace \eqref{CNSR} by equivalent conditions on $\phi(0),\phi(0^-),\phi(\alpha)$ and $\phi(\alpha^{-})$.

\begin{lemma}\label{MONPSIALP} For every $\alpha\in [0,1]$, the function $\psi_{\alpha}:\R\to\R$  is right continuous and non-decreasing. Moreover, 
\begin{enumerate}
\item  For every $z_1<z_2$ in $\R$, we have $\psi_\alpha(z_2) - \psi_\alpha(z_1)\in\{0,d\}$ if $\ent{z_2}=\ent{z_1}$ and $\psi_\alpha(z_2)-\psi_\alpha(z_1)\geq 1-\lambda-d>0$, otherwise.
\item For every $z\in\R$, we have
\begin{equation}\label{JUMPPSI}
\psi_{\alpha}(z)-\psi_{\alpha}(z^-)=(1-\lambda-d)\Ind_{\Z}(z)
+d\Ind_{\Z+\alpha}(z).
\end{equation}
\end{enumerate}
\end{lemma}
\begin{proof} The function $\psi_{\alpha}$ is right continuous, by right continuity of the integer part, the fractional part and $\theta_\alpha$. Now, let $z_1<z_2$, then
\[
\psi_{\alpha}(z_2)-\psi_{\alpha}(z_1)=(1-\lambda)(\lf z_2\rf-\lf z_1\rf)+d(\theta_{\alpha}(\{z_2\})-\theta_{\alpha}(\{z_1\})).
\]
If $\lf z_1\rf=\lf z_2\rf$, then $\{z_2\}-\{z_1\}=z_2-z_1>0$. It follows that 
\[
\theta_{\alpha}(\{z_2\})-\theta_{\alpha}(\{z_1\})\in\{0,1\}\quand \psi_{\alpha}(z_2)-\psi_{\alpha}(z_1)\in\{0,d\}. 
\]
If $\lf z_1\rf\neq\lf z_2\rf$, then
\[
\psi_{\alpha}(z_2)-\psi_{\alpha}(z_1)\geq(1-\lambda)+d(\theta_{\alpha}(\{z_2\})-\theta_{\alpha}(\{z_1\}))\geq 1-\lambda-d>0.
\]
We have proved that $\psi_\alpha$ is non-decreasing and item $1)$. Finally, if $z\notin\Z\cup(\Z+\alpha)$ then, the integer part and the fractional part are continuous at $z$ and $\theta_\alpha$ is continuous at $\{z\}$, which implies that $\psi_\alpha(z)=\psi_\alpha(z^-)$. Suppose  $z\in\Z\cup(\Z+\alpha)$. If $\alpha\in(0,1)$, then either $z\in\Z$ and $\psi_\alpha(z)-\psi_\alpha(z^-)=1-\lambda-d$ or 
$z\in\Z+\alpha$ and $\psi_\alpha(z)-\psi_\alpha(z^-)=d$.  If $\alpha\in\{0,1\}$, then $z\in\Z\cap(\Z+\alpha)$ and $\psi_\alpha(z)-\psi_\alpha(z^-)=1-\lambda$. This ends the proof  of \eqref{JUMPPSI}.
\end{proof}

\begin{proposition}\label{PHISCR}  For any $\alpha\in[0,1]$ and $\rho\in(0,1)$, the function $\phi:\R\to\R$ is right continuous, continuous at every $y\in\R$ such that $y\notin(\Z+k\rho)\cup(\Z+\alpha+k\rho)$ for all $k\geq 1$, non-decreasing  and  increasing if $\rho\in(0,1)\setminus\Q$. 
\end{proposition}

\begin{proof} The continuity properties of $\phi$ follow from those  of $\psi_\alpha$ stated in Lemma \ref{MONPSIALP}. Now, if $y'>y\in\R$, then by $1)$ of Lemma \ref{MONPSIALP}
\begin{equation}\label{DIFPSIK}
\psi_{\alpha}(y'-k\rho)-\psi_{\alpha}(y-k\rho)\geq 0\quad\forall k\in\N
\end{equation}
and
\begin{align*}
\phi(y')-\phi(y)=\frac{1}{\lambda}\sum_{k=1}^{\infty}\lambda^k(\psi_{\alpha}(y'-k\rho)-\psi_{\alpha}(y-k\rho))\geq 0,
\end{align*}
which proves that $\phi$ is a non-decreasing function. Also, by \eqref{DIFPSIK}, $\phi(y')-\phi(y)=0$ if and only if $\psi_{\alpha}(y'-k\rho)-\psi_{\alpha}(y-k\rho)=0$ for every $k\geq 1$. By $1)$ of Lemma \ref{MONPSIALP}, this can hold only if 
\[
\lf y'-k\rho\rf=\lf y-k\rho\rf \quad\forall k\geq 1.
\]
If $\rho\in(0,1)\setminus\Q$, then the sequence $(\{y'-k\rho\})_{k\in\N}$ is dense in $[0,1]$. Therefore, there exists $k\in\N$ such that $0\leq\{y'-k\rho\}<y'-y=y'-k\rho-(y-k\rho)$ and
\[
\ent{y-k\rho}\leq y-k\rho<\ent{y'-k\rho}.
\]
It follows that $\phi$ is  increasing if $\rho\in(0,1)\setminus\Q$. 
\end{proof}

For $\rho\in\Q$, the sum defining the function $\phi$ became finite and finer properties of $\phi$ can be deduced from those of the function $\psi_\alpha$, as shown by the following Lemma.

\begin{lemma}\label{RATPHI} For any $\alpha\in[0,1]$, $q\geq 2$, $p\in\Z$ and $y\in\R$ we have
\begin{equation}\label{SPSIRAT}
\sum_{k=1}^{\infty}\lambda^k\psi_{\alpha}\left(y-k\frac{p}{q}\right)=\frac{1}{1-\lambda^q}\left(\lambda^q(\psi_{\alpha}(y)-p)+\sum_{r=1}^{q-1}\lambda^{r}\psi_{\alpha}\left(y-r\frac{p}{q}\right)\right).
\end{equation}
As a consequence, the function $\phi$ is piecewise constant if $\rho\in\Q$ and for $\rho=p/q$, with $0<p<q$ co-prime, the set of its jump discontinuities in the interval $[0,1)$ is given by 
\[
\{0,\tfrac{1}{q},...,\tfrac{q-1}{q}\}\cup \{\alpha_0,...,\alpha_{q-1}\},
\]
where $\alpha_i:=\{\alpha+\frac{i}{q}\}$ for all $i\in\{0,...,q-1\}$.
\end{lemma}
\begin{proof} First note that the function $\psi_{\alpha}$ has the following property
\begin{equation}\label{PSIPROPERTY}
\psi_{\alpha}(z+l)=\psi_{\alpha}(z)+l(1-\lambda) \quad\forall z\in\R \quad\text{if}\quad l\in\Z.
\end{equation}
Now, let $\alpha\in[0,1]$, $q\geq 2$, $p\in\Z$ and $y\in\R$, then 
\begin{align*}
\sum_{k=0}^{\infty}\lambda^k\psi_{\alpha}\left(y-k\frac{p}{q}\right)
&=\sum_{l=0}^{\infty}\sum_{r=0}^{q-1}\lambda^{ql+r}\psi_{\alpha}\left(y-lp-r\frac{p}{q}\right)    \\
&=\sum_{l=0}^{\infty}\sum_{r=0}^{q-1}\lambda^{ql+r}\psi_{\alpha}\left(y-r\frac{p}{q}\right)  - p(1-\lambda)\sum_{l=0}^{\infty}\sum_{r=0}^{q-1}l\lambda^{ql+r}   \\
&=\frac{1}{1-\lambda^q}\sum_{r=0}^{q-1}\lambda^{r}\psi_{\alpha}\left(y-r\frac{p}{q}\right)  - p(1-\lambda^q)\sum_{l=0}^{\infty}l\lambda^{ql}    \\
&=\frac{1}{1-\lambda^q}\sum_{r=0}^{q-1}\lambda^{r}\psi_{\alpha}\left(y-r\frac{p}{q}\right)  - p\frac{\lambda^q}{1-\lambda^q}.
\end{align*}
So, we obtain that
\begin{equation}
\sum_{k=1}^{\infty}\lambda^k\psi_{\alpha}\left(y-k\frac{p}{q}\right)=\frac{1}{1-\lambda^q}\left(-p\lambda^q+\sum_{r=0}^{q-1}\lambda^{r}\psi_{\alpha}\left(y-r\frac{p}{q}\right) \right)-\psi_{\alpha}(y),
\end{equation}
which leads to \eqref{SPSIRAT}. It follows that for $\rho=p/q$ we have that
\[
\phi(y)=\frac{\delta}{1-\lambda}+\frac{1}{\lambda(1-\lambda^q)}\left(\lambda^q(\psi_{\alpha}(y)-p)+\sum_{r=1}^{q-1}\lambda^{r}\psi_{\alpha}\left(y-r\frac{p}{q}\right) \right)\quad\forall y\in\R.
\]
As $\psi_\alpha$ is a piecewise constant function and $z\in\R$ is a jump discontinuity of $\psi_\alpha$ if and only if $z\in\Z$ or $\{z\}=\alpha$,  the function $\phi$ is piecewise constant and 
$y\in[0,1)$ is a jump discontinuity of $\phi$ if and only if 
\[
y=\left\{r\frac{p}{q}\right\}\quad\text{or}\quad y=\left\{\alpha+r\frac{p}{q}\right\},
\]
for some $r\in\{0,1,\dots,q-1\}$. For any $r\in\{0,1,\dots,q-1\}$, we have that
\[
\left\{r\frac{p}{q}\right\}=\frac{i_r}{q}\quand
\left\{\alpha+r\frac{p}{q}\right\}=\left\{\alpha+\frac{i_r}{q}\right\},
\]
for some $i_r\in\{0,1,\dots,q-1\}$. Now, if $0<p<q$ are co-prime, then  $i_r\neq i_r'$ if $r\neq r'$, which ends the proof. 
\end{proof}

\begin{lemma}\label{LIMINFIM} If $\rho\in\Q$, then for any $z\in(0,1]$, there exists $y\in[0,z)$ such that $\phi(y)=\phi(z^-)$. 
\end{lemma}
\begin{proof}
As $\phi$ is piecewise constant for $\rho\in\Q$, the image of any bounded set by $\phi$ is finite, which implies that $\overline{\phi([0,1])}=\phi([0,1])$. Therefore, $\phi(z^-)\in\phi([0,1])$ and is an isolated point of $\phi([0,1])$. Now, let $\{y_n\}_{n\in\N}$  be a sequence of $[0,z)$ such that $y_n\to z$. Then $\phi(y_n)\to\phi(z^-)$ and for any $n$ large enough 
we have $y_n<z$ and $\phi(y_n)=\phi(z^-)$.
\end{proof}

Now, we can start to restate the conditions \eqref{CNSR} of Proposition \ref{CNSCONJ}.
 
\begin{lemma}\label{INTERVALCERO} Let $\alpha\in[0,1]$. If $\rho\in(0,1)\cap\Q$, then $\phi([0,1))\subset[0,1)$ if and only if $0\in(\phi(0^{-}),\phi(0)]$. If $\rho\in(0,1)\setminus\Q$, then the same holds replacing $(\phi(0^{-}),\phi(0)]$ with $[\phi(0^{-}),\phi(0)]$.
\end{lemma}

\begin{proof}
Suppose that $\phi([0,1))\subset[0,1)$. Then $0\leq\phi(0)$ and $\phi(1^{-})\leq 1$, that is, $\phi(0^{-})=\phi(1^{-})-1\leq 0$. So, for any $\rho\in(0,1)$, we have $0\in[\phi(0^{-}),\phi(0)]$. If $\rho\in(0,1)\cap\Q$, by Lemma \ref{LIMINFIM}, there exists $y\in[0,1)$ such that $\phi(y)=\phi(1^-)$, which implies that $\phi(1^-)<1$ and $\phi(0^{-})< 0$. 

Now, let $y\in[0,1)$. By Proposition \ref{PHISCR}  we have  $\phi(0)\leq\phi(y)\leq\phi(1^{-})$ with $\phi(y)<\phi(1^{-})$ if $\rho\in(0,1)\setminus\Q$. We deduce that $\phi([0,1))\subset[0,1)$ holds if $0\in(\phi(0^{-}),\phi(0)]$ whenever $\rho\in(0,1)\cap\Q$ and if $0\in[\phi(0^{-}),\phi(0)]$ for $\rho\in(0,1)\setminus\Q$.
\end{proof}

\begin{lemma} Let $\alpha, a\in [0,1]$. Suppose $\rho\in (0,1)\cap \Q$, then 
\begin{enumerate}
\item  If $\alpha\in (0,1)$, we have $\theta_\alpha(y)=\theta_a(\phi(y))$ for all $y\in[0,1)$ if and only if $a\in(\phi(\alpha^{-}),\phi(\alpha)]$.
\item If $\alpha\in \{0,1\}$ and $\phi([0,1))\subset [0,1)$, we have
$\theta_\alpha(y)=\theta_a(\phi(y))$ for all $y\in [0,1)$ if and only if $a\in (\phi(\alpha^-),\phi(\alpha)]$.
\end{enumerate}
For $\rho\in (0,1)\setminus \Q$,  1) and 2) hold  replacing $(\phi(\alpha^{-}),\phi(\alpha)]$ with $[\phi(\alpha^{-}),\phi(\alpha)]$. 
\end{lemma}

\begin{proof}
$1)$ Let $\alpha\in (0,1)$. Suppose that $\theta_\alpha(y)=\theta_a(\phi(y))$ for all $y\in[0,1)$. So, for every $y\in[0,1)$ we have, $\phi(y)<a$ if and only $y<\alpha$. This implies that $\phi(\alpha^{-})\leq a\leq \phi(\alpha)$. In the case $\rho\in(0,1)\cap\Q$, by Lemma \ref{LIMINFIM}, there exists $y\in[0,\alpha)$ such that $\phi(y)=\phi(\alpha^{-})$ and therefore  $\phi(\alpha^-)<a$. 

Now, by Proposition \ref{PHISCR}, we have $\phi(\alpha)\leq\phi(y)$  if $\alpha\leq y$, and $\phi(y)\leq\phi(\alpha^{-})$ if $y<\alpha$, with   
$\phi(y)<\phi(\alpha^-)$ if $\rho\in(0,1)\setminus\Q$. So, $\theta_\alpha(y)=\theta_a(\phi(y))$ if $\phi(\alpha^{-})<a\leq\phi(\alpha)$ for $\rho\in(0,1)\cap\Q$ and if $\phi(\alpha^{-})\leq a\leq\phi(\alpha)$ for $\rho\in(0,1)\setminus\Q$.

\noindent $2)$ Let $\alpha=0$. If $\phi([0,1))\subset [0,1)$, by Lemma \ref{INTERVALCERO} we have $\phi(\alpha^-)=\phi(0^-)<0\leq a$ if $\rho\in (0,1)\cap \Q$ and 
$\phi(\alpha^-)=\phi(0^-)\leq 0\leq a$ if $\rho\in (0,1)\setminus \Q$.
Suppose $1=\theta_0(y)=\theta_a(\phi(y))$ for every $y\in [0,1)$. Then, 
$\phi(\alpha)=\phi(0)\geq a$ and it follows that $a\in[\phi(\alpha^-),\phi(\alpha)]$ with left open interval if $\rho\in (0,1)\cap \Q$. Conversely, if $a\leq\phi(\alpha)=\phi(0)$, then, by monotonicity of $\phi$, $a\leq\phi(y)$ and $1=\theta_a(\phi(y))=\theta_0(y)$ for every 
$y\in [0,1)$.

Let $\alpha=1$. If $\phi([0,1))\subset [0,1)$, by Lemma \ref{INTERVALCERO} we have $\phi(\alpha)=\phi(1)=\phi(0)+1\geq 1\geq a$. 
Suppose $0=\theta_1(y)=\theta_a(\phi(y))$ for every $y\in [0,1)$. Then, 
$\phi(y)<a$ for every $y\in [0,1)$ and $\phi(\alpha^-)=\phi(1^-)\leq a$, with strict inequality if $\rho\in\Q$ since $\phi(1^-)=\phi(y)$ for some $y\in[0,1)$. It follows that $a\in[\phi(\alpha^-),\phi(\alpha)]$ with left open interval if $\rho\in (0,1)\cap \Q$. Conversely, if $\phi(\alpha^-)=\phi(1^-)< a$ then $\phi(y)\leq\phi(1^-)<a$ and $0=\theta_a(\phi(y))=\theta_0(y)$ for every $y\in [0,1)$. If $\phi(\alpha^-)=\phi(1^-)\leq a$ and $\rho\in(0,1)\setminus\Q$ then, by strict monotonicity of $\phi$, we have $\phi(y)<\phi(1^-)\leq a$ and $0=\theta_a(\phi(y))=\theta_0(y)$ for every $y\in [0,1)$. 
\end{proof}

\begin{lemma}\label{INTERVALS} For any $\alpha, a\in[0,1]$ and $\rho\in(0,1)$ the following are equivalent
\begin{enumerate}
\item $\phi([0,1))\subset[0,1)$ and $\theta_{a}(\phi(y))=\theta_{\alpha}(y)$
for every $y\in[0,1)$.
\item $0\in(\phi(0^-), \phi(0)]$ and $a\in(\phi(\alpha^-), \phi(\alpha)]$,
with closed intervals if $\rho\in(0,1)\setminus\Q$.
\end{enumerate}
\end{lemma}
\begin{proof} It follows from the two former lemmata.
\end{proof}

\section{Proof of Theorem \ref{THROT}}\label{PTH1}

Here, using Lemma \ref{INTERVALS}, we express the conditions \eqref{CNSR} as a condition of the form \eqref{CONDTH} on the values of $\delta$ and $a$. Then, we study the sizes of the intervals \eqref{CONDTH} and give their expressions for $\rho$ rational.

\begin{lemma}\label{SYM} For any $\alpha\in[0,1]$ and $y\in\R$, we have   
\begin{equation}\label{PHYSYM}
\phi(y)=\frac{\delta+d}{1-\lambda}-1-\frac{1}{\lambda}\sum_{k=1}^{\infty}\lambda^k\psi_{1-\alpha}(k\rho^{-}-y)\quand
\phi(y^-)=\frac{\delta+d}{1-\lambda}-1-\frac{1}{\lambda}\sum_{k=1}^{\infty}\lambda^k\psi_{1-\alpha}(k\rho-y).
\end{equation}
\end{lemma}

\begin{proof} Let us show that for any $\alpha\in[0,1]$ and $z\in\R$, we have
\begin{equation}\label{PSIMZ}
\psi_\alpha(-z)=-(\psi_{1-\alpha}(z^-)+1-\lambda-d)\quand \psi_\alpha(z^-)=-(\psi_{1-\alpha}(-z)+1-\lambda-d).
\end{equation}
Let us first suppose  that $z\in\R\setminus\Z$. Then, $\ent{-z}=-\ent{z}-1$ and $\{-z\}=1-\{z\}$. Therefore, 
\[
\psi_\alpha(-z)=-(1-\lambda)\ent{z}-(1-\lambda)+d\theta_{\alpha}(1-\{z\}).
\]
Since, 
\[
\theta_{\alpha}(1-z)=\left\{\begin{array}{lcr}
1-\theta_{1-\alpha}(z) &\text{if}& z\neq 1-\alpha\\
1-\theta_{1-\alpha}(z)+1 &\text{if}& z=1-\alpha
\end{array}
\right.\quad \forall z\in\R,
\]
we can write
\begin{align*}
\psi_\alpha(-z)
&=-(1-\lambda)\ent{z}-(1-\lambda)+d(1-\theta_{1-\alpha}(\{z\})+\Ind_{\Z+1-\alpha}(z))\\
&=-(\psi_{1-\alpha}(z)-d\Ind_{\Z+1-\alpha}(z)+ 1-\lambda-d)\\
&=-(\psi_{1-\alpha}(z^-)+1-\lambda-d)
\end{align*}
and we deduce that the left hand side equality of \eqref{PSIMZ} holds for $z\not\in\Z$. 

Now if $z\in\Z$, then,  
\begin{align*}
\psi_{1-\alpha}(z^-)
&=\psi_{1-\alpha}(z)-(1-\lambda-d)\Ind_\Z(z) -d\Ind_{\Z+1-\alpha}(z)\\
&=(1-\lambda)\ent{z}+d\theta_{1-\alpha}(0)-(1-\lambda-d)-d\Ind_{\Z+1-\alpha}(z).
\end{align*}
But, on the other hand

\begin{align*}
\psi_{\alpha}(-z)
&=-((1-\lambda)\ent{z}-d\theta_{\alpha}(0))\\
&=-(\psi_{1-\alpha}(z^-)+1-\lambda-d-d\theta_{1-\alpha}(0)+d\Ind_{\Z+1-\alpha}(z)-d\theta_{\alpha}(0))\\
&=-(\psi_{1-\alpha}(z^-)+1-\lambda-d).
\end{align*}
Therefore, we have proved that for any $\alpha\in[0,1]$ and $z\in\R$,  
\[
\psi_\alpha(-z)=-(\psi_{1-\alpha}(z^-)+1-\lambda-d).
\]
As $1-\alpha\in[0,1]$ if $\alpha\in[0,1]$, we also have
\[
\psi_{1-\alpha}(-z)=-(\psi_{\alpha}(z^-)+1-\lambda-d),
\]
which gives the right hand side equality of \eqref{PSIMZ}. Then, \eqref{PHYSYM} follows directly from \eqref{PSIMZ} and the definition of $\phi$.
\end{proof}

\begin{lemma}\label{PROPPHI0} For every $\rho\in(0,1)$ and $\alpha\in[0,1]$, we have
\begin{equation}\label{PHI0PHI0M}
\phi(0^{-})=\frac{1}{1-\lambda}(\delta-\delta(\rho,\alpha))\quand
\phi(0)=\frac{1}{1-\lambda}(\delta-\delta(\rho^{-},\alpha)).
\end{equation}
Also,
\begin{equation}\label{DELTAM}
\delta(\rho^{-},\alpha)=\delta(\rho,\alpha)-\frac{1-\lambda}{\lambda}
\sum_{k=1}^{\infty}\lambda^k\left((1-\lambda-d)\Ind_{\Z}(k\rho) +d\Ind_{\Z+\alpha}(-k\rho)\right).
\end{equation}
\end{lemma}

\begin{proof} The equalities \eqref{PHI0PHI0M} follow from \eqref{PHYSYM}. Indeed,
\begin{align*}
\phi(0^-)&=\frac{\delta+d}{1-\lambda}-1-\frac{1}{\lambda}\sum_{k=1}^{\infty}\lambda^k\psi_{1-\alpha}(k\rho)\\
&=\frac{1}{1-\lambda}\left(\delta-\left(1-\lambda-d+
\frac{1-\lambda}{\lambda}\sum_{k=1}^{\infty}\lambda^k((1-\lambda)\ent{k\rho}+d\theta_{1-\alpha}(\{k\rho\})\right)\right)\\
&=\frac{1}{1-\lambda}(\delta-\delta(\rho,\alpha))
\end{align*}
and in the same way
\[
\phi(0)=\frac{\delta+d}{1-\lambda}-1-\frac{1}{\lambda}\sum_{k=1}^{\infty}\lambda^k\psi_{1-\alpha}(k\rho^-)=\frac{1}{1-\lambda}(\delta-\delta(\rho^-,\alpha)).
\]
On the other hand, by \eqref{JUMPPSI}
\begin{align*}
\delta(\rho,\alpha)-\delta(\rho^-,\alpha)
=&\frac{1-\lambda}{\lambda}\sum_{k=1}^{\infty}\lambda^k(\psi_{1-\alpha}(k\rho)-\psi_{1-\alpha}((k\rho)^{-})\\
=&\frac{1-\lambda}{\lambda}
\sum_{k=1}^{\infty}\lambda^k\left((1-\lambda-d)\Ind_{\Z}(k\rho) +d\Ind_{\Z+1-\alpha}(k\rho)\right),
\end{align*}
which gives \eqref{DELTAM}.
\end{proof}

\begin{lemma}\label{PROPPHIA} For every $\delta, \rho$ and $\alpha\in[0,1]$, we have
\begin{equation}\label{PHIAPHIAM}
\phi(\alpha^{-})=a(\delta,\rho^+,\alpha)\quand \phi(\alpha)=a(\delta,\rho,\alpha).
\end{equation}
Also,
\[
a(\delta,\rho^+,\alpha)=a(\delta,\rho,\alpha)-\frac{1}{\lambda}\sum_{k=1}^{\infty}\lambda^k\left((1-\lambda-d)\Ind_{\Z+\alpha}(k\rho) +d\Ind_{\Z}(k\rho)   \right).
\]
\end{lemma}
\begin{proof} The equality $\phi(\alpha)=a(\delta,\rho,\alpha)$ follows from the definition of $\phi$ and 
\[
\phi(\alpha^-)=\frac{\delta}{1-\lambda}+\frac{1}{\lambda}\sum_{k=1}^{\infty}
\lambda^k\psi_{\alpha}(\alpha^--k\rho)=\frac{\delta}{1-\lambda}+\frac{1}{\lambda}\sum_{k=1}^{\infty}
\lambda^k\psi_{\alpha}(\alpha-k\rho^+)=a(\delta,\rho^+,\alpha).
\]
Also, by \eqref{JUMPPSI}
\begin{align*}
a(\delta,\rho,\alpha)-a(\delta,\rho^+,\alpha)
&=\frac{1}{\lambda}\sum_{k=1}^{\infty}\lambda^k(\psi_{\alpha}(\alpha-k\rho)-\psi_{\alpha}(\alpha^{-}-k\rho))\\
&=\frac{1}{\lambda}\sum_{k=1}^{\infty}\lambda^k(\psi_{\alpha}(\alpha-k\rho)-\psi_{\alpha}((\alpha-k\rho)^-))\\
&=\frac{1}{\lambda}\sum_{k=1}^{\infty}\lambda^k\left((1-\lambda-d)\Ind_{\Z}(\alpha-k\rho) +d\Ind_{\Z+\alpha}(\alpha-k\rho)   \right)\\
&=\frac{1}{\lambda}\sum_{k=1}^{\infty}\lambda^k\left((1-\lambda-d)\Ind_{\Z+\alpha}(k\rho) +d\Ind_{\Z}(k\rho)   \right).
\end{align*}
\end{proof}

\begin{proposition}\label{PROROT} Let $\rho\in(0,1)$ and $\alpha\in[0,1]$. If $\rho\in\Q$, then, for every $\delta\in(1-\lambda-d,1)$ and $a\in[0,1]$ we have 
\begin{equation}\label{CONDIR}
F\circ\phi=\phi\circ L_\rho \ssi \delta\in[\delta(\rho^-,\alpha),\delta(\rho,\alpha))
\quand
a\in(a(\delta,\rho^+,\alpha),a(\delta,\rho,\alpha)].
\end{equation}
If $\rho\not\in\Q$, then the same is true with the intervals closed.
\end{proposition}
\begin{proof} By Proposition \ref{CNSCONJ} and Lemma \ref{INTERVALS} we have $F\circ\phi=\phi\circ L_\rho$ if and only if $0\in(\phi(0^-), \phi(0)]$ and $a\in(\phi(\alpha^-), \phi(\alpha)]$, with closed intervals if $\rho\in(0,1)\setminus\Q$. Using \eqref{PHI0PHI0M} and \eqref{PHIAPHIAM}, we obtain the desired result.
\end{proof}

The following two lemmata give the the monotonicity properties of $\delta(\rho,\alpha)$ with 
$\rho$ and allow to show item 1) of Theorem \ref{THROT}.

\begin{lemma}\label{RANGEDELTA}
Let $\alpha\in[0,1]$. Then, the function $\rho\mapsto\delta(\rho,\alpha)$ is increasing and right continuous in $[0,1]$. Moreover,  
\begin{equation}\label{LIMRHOEXT}
\lim_{\rho\to 0^+}\delta(\rho,\alpha)=1-\lambda-d +d\Ind_{\{1\}}(\alpha)\quand \lim_{\rho\to 1^-}\delta(\rho,\alpha)=1-d\Ind_{\{0\}}(\alpha).
\end{equation}
Also, for any $\rho\in(0,1)$ we have $[\delta(\rho^-,\alpha),\delta(\rho,\alpha)]\subset (1-\lambda-d,1)$. 
\end{lemma}

\begin{proof} We first show that $\rho\mapsto\delta(\rho,\alpha)$ is   increasing. Recall that
\begin{align*}
\delta(\rho,\alpha)
&=1-\lambda-d+\frac{1-\lambda}{\lambda}\sum_{k=1}^{\infty}\lambda^k\left((1-\lambda)\lf k\rho\rf+d\theta_{1-\alpha}(\{k\rho\})\right)\\
&=1-\lambda-d+\frac{1-\lambda}{\lambda}\sum_{k=1}^{\infty}\lambda^k\psi_{1-\alpha}(k\rho).
\end{align*}
Let $\alpha\in[0,1]$ and $0\leq\rho<\rho'\leq 1$.  Let $\rho_1=p_1/q_1$ and $\rho_2=p_2/q_2$ in $\Q$ be such that $\rho\leq\rho_1<\rho_2\leq\rho'$. As Lemma \ref{MONPSIALP} implies that $\rho\mapsto\delta(\rho,\alpha)$ is non-decreasing, we have 
$\delta(\rho,\alpha)\leq\delta(\rho_1,\alpha)\leq\delta(\rho_2,\alpha)\leq\delta(\rho',\alpha)$.
But for $k=q_1q_2\geq 1$, we have $k\rho_1=\ent{k\rho_1}\neq\ent{k\rho_2}=k\rho_2$, which implies together with  Lemma \ref{MONPSIALP} that $\psi_{1-\alpha}(k\rho_1)<\psi_{1-\alpha}(k\rho_2)$. It follows that  $\delta(\rho_1,\alpha)<\delta(\rho_2,\alpha)$ and that $\delta(\rho,\alpha)<\delta(\rho',\alpha)$. So we have proved that $\rho\mapsto\delta(\rho,\alpha)$ is   increasing and the right continuity follows from that of $\psi_{1-\alpha}$.

As $\psi_{1-\alpha}$ is right continuous, we have
\[
\lim_{\rho\to 0^+}\delta(\rho,\alpha)=1-\lambda-d+\frac{1-\lambda}{\lambda}\sum_{k=1}^{\infty}\lambda^k\psi_{1-\alpha}(0)=1-\lambda-d +d\Ind_{\{1\}}(\alpha).
\]
On the other hand, 
\begin{align*}
\lim_{\rho\to 1^-}\delta(\rho,\alpha)
&=1-\lambda-d+\frac{1-\lambda}{\lambda}\sum_{k=1}^{\infty}\lambda^k\left((1-\lambda)(k-1) +d-d\Ind_{\{0\}}(\alpha)\right)\\
&=-d\Ind_{\{0\}}(\alpha)+\frac{1-\lambda}{\lambda}\sum_{k=1}^{\infty}\lambda^k(1-\lambda)k
=1-d\Ind_{\{0\}}(\alpha).
\end{align*}
So, we have proved  \eqref{LIMRHOEXT}. Also,  $1-\lambda-d\leq \delta(0^+,\alpha)<\delta(\rho^-,\alpha)<\delta(\rho,\alpha))<\delta(1^-,\alpha)\leq 1$ for every $\rho\in(0,1)$, by strict monotony. 
\end{proof}

\begin{lemma}\label{AINTER01} Let $\rho\in(0,1)$, $\alpha\in[0,1]$ and suppose $\delta\in[\delta(\rho^-,\alpha),\delta(\rho,\alpha)]$.  Then, $[a(\delta,\rho^+,\alpha),a(\delta,\rho,\alpha)]\cap[0,1]\neq\emptyset$
and $[a(\delta,\rho^+,\alpha),a(\delta,\rho,\alpha)]\subset[0,1]$ if $\alpha\in(0,1)$ or $\rho\notin \Q$.
\end{lemma}
\begin{proof} As $\delta\in[\delta(\rho^-,\alpha),\delta(\rho,\alpha)]$, we deduce from \eqref{PHI0PHI0M} that $\phi(0^-)\leq 0\leq\phi(0)$, which together with \eqref{TPHI} implies $\phi(1^-)\leq 1\leq\phi(1)$.  Since $\phi$ is non-decreasing, we have $0\leq\phi(\alpha)$ and $\phi(\alpha^-)\leq 1$ for every $\alpha\in[0,1]$. From \eqref{PHIAPHIAM} we obtain  that   $[a(\delta,\rho^+,\alpha),a(\delta,\rho,\alpha)]\cap[0,1]\neq\emptyset$. Now, if $\alpha\in(0,1)$ then $0\leq\phi(\alpha^-)\leq\phi(\alpha)\leq 1$. If $\alpha\in\{0,1\}$ and $\rho\notin\Q$, then
$\phi$ is continuous in $\{0,1\}$ and  $\phi(\alpha^-)=\phi(\alpha)\in\{0,1\}$.
We conclude from \eqref{PHI0PHI0M} that $[a(\delta,\rho^+,\alpha),a(\delta,\rho,\alpha)]\subset[0,1]$ if $\alpha\in(0,1)$ or $\rho\notin \Q$.
\end{proof}

Now, let us study the size of the intervals in \eqref{CONDIR}. We start by the case where $\rho$ is irrational.

\begin{lemma}\label{LINTIR} Let $\rho\in(0,1)\setminus\Q$ and $\alpha\in[0,1]$. If $D:=\{k\in\Z^*:k\rho\in\Z+\alpha\}$ is empty, then 
\[
\delta(\rho^-,\alpha)=\delta(\rho,\alpha) \quand a(\delta,\rho^+,\alpha)=a(\delta,\rho,\alpha).
\]
If $D$ is not the empty set, then it has a unique element $k_\alpha$ and 
\[
\delta(\rho^{-},\alpha)=\delta(\rho,\alpha)-\Ind_{\Z^{-}}(k_\alpha)(1-\lambda)d\lambda^{-k_\alpha-1}
\quand  
a(\delta,\rho^+,\alpha)=a(\delta,\rho,\alpha)-\Ind_{\Z^{+}}(k_\alpha)(1-\lambda-d)\lambda^{k_\alpha-1}.
\]
\end{lemma}
\begin{proof} If there exists $k,k'\in\Z$ such that  
$k\rho=\alpha+p$ and $k'\rho=\alpha+p'$ for some $p,p'\in \Z$, then $(k-k')\rho\in\Z$ and  $k'=k$,
since $\rho$ is irrational. We deduce that $D$ is either the empty set or is a singleton. It follows then by Lemma \ref{PROPPHI0} and Lemma \ref{PROPPHIA} that for every $k_0\in\Z$
we have
\[
\delta(\rho,\alpha)-\delta(\rho^{-},\alpha)=\frac{1-\lambda}{\lambda}
\sum_{k=1}^{\infty}\lambda^k\left((1-\lambda-d)\Ind_{\Z}(k\rho) +d\Ind_{\Z+\alpha}(-k\rho)   \right)
=\Ind_{\Z^{-}\cap D}(k_0)(1-\lambda)d\lambda^{-k_0-1}
\]
and  
\[
a(\delta,\rho,\alpha)-a(\delta,\rho^+,\alpha)=\frac{1}{\lambda}\sum_{k=1}^{\infty}\lambda^k\left((1-\lambda-d)\Ind_{\Z+\alpha}(k\rho) +d\Ind_{\Z}(k\rho)   \right)=\Ind_{\Z^{+}\cap D}(k_0)(1-\lambda-d)\lambda^{k_0-1},
\]
which ends the proof.
\end{proof}

\begin{lemma}\label{LINTRAT} Let $0<p<q$ be two co-prime natural numbers, $\rho=p/q$ and $\alpha\in[0,1]$. Then, 
\begin{equation}\label{DELTAMRAT}
\delta(\rho^-,\alpha)=\delta(\rho,\alpha)-\frac{\lambda^{q-1}}{1-\lambda^q}(1-\lambda)(1-\lambda-d+d\Ind_{\{0,1\}}(\alpha)+\lambda^{-r_\alpha}d\Ind_{D_{\rho,\alpha}}(r_\alpha)),
\end{equation}
and for every $\delta\in\R$,
\[
a(\delta,\rho^+,\alpha)=a(\delta,\rho,\alpha)-\frac{\lambda^{q-1}}{1-\lambda^q}(d+
(1-\lambda-d)\Ind_{\{0,1\}}(\alpha))-\frac{\lambda^{r_\alpha-1}}{1-\lambda^q}(1-\lambda-d)\Ind_{D_{\rho,\alpha}}(r_\alpha),
\]
where $D_{\rho,\alpha}:=\{r\in\{1,\dots,q-1\}: rp/q\in\Z+\alpha\}$ and $r_\alpha$ is the unique element of $D_{\rho,\alpha}$ if not empty and $r_\alpha=0$ otherwise.

\end{lemma}

\begin{proof}  From Lemma \ref{PROPPHI0}, we have that
\[
\delta(\rho,\alpha)-\delta(\rho^{-},\alpha)
=\frac{1-\lambda}{\lambda}
\sum_{k=1}^{\infty}\lambda^k\left((1-\lambda-d)\Ind_{\Z}(k\rho)+d\Ind_{\Z+\alpha}(-k\rho)    \right)
\]
But, for $\rho=p/q$,
\begin{align*}
\sum_{k=0}^{\infty}\lambda^k\left((1-\lambda-d)\Ind_{\Z}(k\rho) +d\Ind_{\Z+\alpha}(-k\rho) \right)
=&\sum_{l=0}^{\infty}\sum_{r=0}^{q-1}\lambda^{ql+r}\left((1-\lambda-d)\Ind_{\Z}\left(lp+r\frac{p}{q}\right)+d\Ind_{\Z+\alpha}\left(-lp-r\frac{p}{q}\right) \right) \\
=&\frac{1}{1-\lambda^q}\sum_{r=0}^{q-1}\lambda^{r}\left((1-\lambda-d)\Ind_{\Z}\left(r\frac{p}{q}\right)+d\Ind_{\Z+\alpha}\left(-r\frac{p}{q}\right) \right)
\end{align*}
First, $rp/q\in\Z$ only if $r=0$, since $p$ and $q$ are co-prime. Second,  the set $D_{\rho,\alpha}$ is either empty or contains a unique element $r_\alpha$. Indeed, if there exist $r_0$ and $r_1$ in this set, then $|r_1-r_0|/q\in\Z$ because $p$ and $q$ are co-prime, and it follows that $r_0=r_1$ since $0\leq |r_1-r_0|/q<1$. Also, the set $\{r\in\{1,\dots,q-1\}:-rp/q\in\Z+\alpha\}$ is empty if and only if $D_{\rho,\alpha}$ is empty and, if  it is not empty, then it contains a unique element $r'_\alpha=-(r_\alpha-q)$. It follows that, 
\begin{align*}
\delta(\rho,\alpha)-\delta(\rho^{-},\alpha)
=&\frac{1-\lambda}{\lambda}\left(\frac{1}{1-\lambda^q}\sum_{r=0}^{q-1}\lambda^r\left((1-\lambda-d)\Ind_{\Z}\left(r\frac{p}{q}\right)+d\Ind_{\Z+\alpha}\left(-r\frac{p}{q}\right) \right)-  (1-\lambda-d)\Ind_{\Z}(0)-d\Ind_{\Z+\alpha}(0)\right)\\
=&\frac{1-\lambda}{\lambda}\left(\frac{1}{1-\lambda^q}\left(
1-\lambda-d+d\Ind_{\Z+\alpha}(0)+\Ind_{D_{\rho,\alpha}}(r_\alpha)d\lambda^{q-r_\alpha}\right)-  (1-\lambda-d)-d\Ind_{\Z+\alpha}(0))\right)\\
=&\frac{1-\lambda}{\lambda}\left(\left(\frac{1}{1-\lambda^q}-1\right)(
1-\lambda-d+d\Ind_{\{0,1\}}(\alpha))+\frac{1}{1-\lambda^q}\Ind_{D_{\rho,\alpha}}(r_\alpha)d\lambda^{q-r_\alpha}\right)\\
=&\frac{\lambda^{q-1}}{1-\lambda^q}(1-\lambda)(1-\lambda-d+d\Ind_{\{0,1\}}(\alpha)+\lambda^{-r_\alpha}d\Ind_{D_{\rho,\alpha}}(r_\alpha)).
\end{align*}

\noindent In a same way, using Lemma \ref{PROPPHIA}, we obtain 
\begin{align*}
a(\delta,\rho,\alpha)-a(\delta,\rho^+,\alpha)
=&\frac{1}{\lambda}\left(\frac{1}{1-\lambda^q}\sum_{r=0}^{q-1}\lambda^r\left((1-\lambda-d)\Ind_{\Z+\alpha}\left(r\frac{p}{q}\right)+d\Ind_{\Z}\left(r\frac{p}{q}\right) \right)-(1-\lambda-d)\Ind_{\Z+\alpha}(0)-d\Ind_{\Z}(0)\right)\\
=&\frac{1}{\lambda}\left(\frac{1}{1-\lambda^q}\left((1-\lambda-d)\Ind_{\{0,1\}}(\alpha)+d+\Ind_{D_{\rho,\alpha}}(r_\alpha)(1-\lambda-d)\lambda^{r_\alpha}\right)-(1-\lambda-d)\Ind_{\{0,1\}}(\alpha)-d\right)\\
=&\frac{1}{\lambda}\left(\left(\frac{1}{1-\lambda^q}-1\right)(
(1-\lambda-d)\Ind_{\{0,1\}}(\alpha)+d)+\frac{1}{1-\lambda^q}\Ind_{D_{\rho,\alpha}}(r_\alpha)(1-\lambda-d)\lambda^{r_\alpha}\right)\\
=&\frac{\lambda^{q-1}}{1-\lambda^q}(d+
(1-\lambda-d)\Ind_{\{0,1\}}(\alpha))+\frac{\lambda^{r_\alpha-1}}{1-\lambda^q}(1-\lambda-d)\Ind_{D_{\rho,\alpha}}(r_\alpha).
\end{align*}

\end{proof}

Now, we give in the following lemma the expressions of $\delta(\rho,\alpha)$
and $a(\delta,\rho,\alpha)$ for $\rho\in\Q$.

\begin{lemma}\label{DELT_ARAT} Let $0<p<q$ be two co-prime natural numbers, $\rho=p/q$ and $\alpha\in[0,1]$. Then,  
\[
\delta(\rho,\alpha)=1-\lambda-d+\frac{1-\lambda}{\lambda(1-\lambda^q)}\left(\lambda^q(p+d\Ind_{\{1\}}(\alpha)) + \sum_{r=1}^{q-1}\lambda^{r}\psi_{1-\alpha}\left(r\frac{p}{q}\right) \right),
\]
and
\[
a(\delta,\rho,\alpha)=\frac{\delta}{1-\lambda}-\frac{\lambda^{q-1}}{1-\lambda^q}(p-d-(1-\lambda-d)\Ind_{\{1\}}(\alpha))+\frac{1}{\lambda(1-\lambda^q)}\sum_{r=1}^{q-1}\lambda^{r}\psi_{\alpha}\left(\alpha-r\frac{p}{q}\right).
\]
\end{lemma}

\begin{proof} Let compute $\delta(\rho,\alpha)$, we note that \eqref{SPSIRAT} holds for any $\alpha\in[0,1]$, $q\geq 1$, $p\in\Z$ and $y\in\R$. Therefore, 
\[
\sum_{k=1}^{\infty}\lambda^k\psi_{1-\alpha}\left(k\frac{p}{q}-y\right)
=\frac{1}{1-\lambda^q}\left(\lambda^q(\psi_{1-\alpha}(-y)+p) + \sum_{r=1}^{q-1}\lambda^{r}\psi_{1-\alpha}\left(r\frac{p}{q}-y\right)  \right) \qquad\forall y \in\R.
\]
It follows that, 
\begin{align*}
\delta\left(\frac{p}{q}, \alpha\right)
&=1-\lambda-d+\frac{1-\lambda}{\lambda}\sum_{k=1}^{\infty}\lambda^{k}\psi_{1-\alpha}\left(k\frac{p}{q}\right)\\
&=1-\lambda-d+\frac{1-\lambda}{\lambda(1-\lambda^q)}\left(\lambda^q(p+d\Ind_{\{1\}}(\alpha)) + \sum_{r=1}^{q-1}\lambda^{r}\psi_{1-\alpha}\left(r\frac{p}{q}\right) \right),
\end{align*}
since $\psi_{1-\alpha}(0)=d\Ind_{\{1\}}(\alpha)$. On the other hand, using \eqref{SPSIRAT}  with $y=\alpha$ and noting that
$\psi_{\alpha}(\alpha)=d+(1-\lambda-d)\Ind_{\{1\}}(\alpha)$,
 we obtain
\begin{align*}
a(\delta,\rho,\alpha)&=\frac{\delta}{1-\lambda}+\frac{1}{\lambda}
\sum_{k=1}^{\infty}\lambda^k\psi_{\alpha}\left(\alpha-k\frac{p}{q}\right)\\
&=\frac{\delta}{1-\lambda}+\frac{1}{\lambda(1-\lambda^q)}\left(\lambda^q(\psi_{\alpha}(\alpha)-p)+\sum_{r=1}^{q-1}\lambda^{r}\psi_{\alpha}\left(\alpha-r\frac{p}{q}\right)\right)\\
&=\frac{\delta}{1-\lambda}+\frac{1}{\lambda(1-\lambda^q)}\left(\lambda^q(d+(1-\lambda-d)\Ind_{\{1\}}(\alpha)-p)+\sum_{r=1}^{q-1}\lambda^{r}\psi_{\alpha}\left(\alpha-r\frac{p}{q}\right)\right)\\
&=\frac{\delta}{1-\lambda}-\frac{\lambda^{q-1}}{1-\lambda^q}(p-d-(1-\lambda-d)\Ind_{\{1\}}(\alpha))+\frac{1}{\lambda(1-\lambda^q)}\sum_{r=1}^{q-1}\lambda^{r}\psi_{\alpha}\left(\alpha-r\frac{p}{q}\right).
\end{align*}
\end{proof}

\begin{proof}[Proof of Theorem \ref{THROT}] Let $\rho\in(0,1)$ and $\alpha\in[0,1]$.  The item 1) of Theorem \ref{THROT} is deduced from Lemma \ref{RANGEDELTA} and Lemma \ref{AINTER01}. The items 2) and 3)  follow from Lemma \ref{LINTIR} and Lemma \ref{LINTRAT}.  Now, suppose that $\delta\in(1-\lambda-d,1)$ and $a\in[0,1]$ satisfy 
\begin{equation}\label{CONDTHAB}
\delta\in[\delta(\rho^-,\alpha),\delta(\rho,\alpha))
\quand
a\in(a(\delta,\rho^+,\alpha),a(\delta,\rho,\alpha)], 
\end{equation}
with closed intervals if  $\rho\in(0,1)\setminus\Q$. Then, by Proposition \ref{PROROT}, we have 
\[
F\circ\phi(y)=\phi\circ L_\rho(y)\quad\forall y\in\R.
\]
Also, $\phi([0,1))\subset[0,1)$ by Proposition \ref{CNSCONJ}, which implies that $\ent{\phi(y)}=\ent{y}$ for all $y\in\R$ by $2)$ of Lemma \ref{PHIEF}. Let $x=\phi(0)$, then
\[
\ent{F^n(x)}=\ent{\phi(n\rho)}=\ent{n\rho}=n\rho-\{n\rho\}
\]
and
\[
\lim_{n\to\infty}\frac{F^n(x)}{n}=\lim_{n\to\infty}\frac{\ent{F^n(x)}}{n}=\rho.
\]
To end the proof of Theorem \ref{THROT}, let us suppose that $\rho\in\Q$ and that $(\delta,a)$ satisfies \eqref{CONDTH} with $\delta=\delta(\rho,\alpha)$ or $a=a(\delta,\rho^+,\alpha)$. Using continuity arguments of the rotation number given in \cite{Rhodes1991} and adapted to our two parameters context in Appendix \ref{ACONT}, let us show that the rotation number of $F$ is $\rho$. 
As $\rho$ is a rational number, by Lemma \ref{LINTRAT}, the intervals $[\delta(\rho^-,\alpha),\delta(\rho,\alpha))$ and $(a(\delta,\rho^+,\alpha),a(\delta,\rho,\alpha)]$ are non-empty. So, let $(\delta_n)_{n\in \N}\subset [\delta(\rho^-,\alpha),\delta(\rho,\alpha))$ be a sequence converging to $\delta$.

\medskip
\noindent
{\bf Case 1:} If $a(\delta,\rho^+,\alpha)<1$ and $a(\delta,\rho,\alpha)>0$, then there exist a sequence  $(a_m)_{m\in \N}$ contained in $(a(\delta,\rho^+,\alpha),a(\delta,\rho,\alpha))\cap [0,1]$ and converging to $a$, since $a\in[0,1]$ and satisfies \eqref{CONDTH}. Let $m\in \N$ be fixed. Since  $\delta'\mapsto a(\delta',\rho^+,\alpha)$ and $\delta'\mapsto a(\delta',\rho,\alpha)$ are continuous, there exists $n_0\in \N$ such that 
\[
a_m\in (a(\delta_n,\rho^+,\alpha),a(\delta_n,\rho,\alpha)],\quad \forall n\ge n_0.
\] Therefore, $\rho(F_{\delta_n,a_m})=\rho$ for all $n\ge n_0$, since $\delta_n$ and $a_m$ satisfy  \eqref{CONDTHAB} for all $n\ge n_0$. It follows that, $\rho(F_{\delta,a_m})=\lim\limits_{n\to \infty} \rho(F_{\delta_n,a_m})=\rho$, by Lemma \ref{ContRot}. We have proved that $\rho(F_{\delta,a_m})=\rho$ for all $m\in \N$, which implies that $\rho(F_{\delta,a})=\rho$, by Lemma \ref{ContRot}.

\medskip
\noindent
{\bf Case 2:} Suppose $a(\delta,\rho^+,\alpha)\geq 1$. 
As $a\in[0,1]$ satisfies \eqref{CONDTH},  we have $a(\delta,\rho^+,\alpha)=1=a$. If $\delta\in [\delta(\rho^-,\alpha),\delta(\rho,\alpha))$, then $\phi(1^-)=\phi(0^-)+1<1$  by Lemma \ref{PROPPHI0}, which implies $\phi(\alpha^-)\leq \phi(1^-)<1$ by Proposition \ref{PHISCR} and $a(\delta,\rho^+,\alpha)<1$ by Lemma \ref{PROPPHIA}, which is a contradiction. It follows that $\delta=\delta(\rho,\alpha)$
and $\delta_n<\delta$ for all $n\in\N$. As  $\delta'\mapsto a(\delta',\rho^+,\alpha)$ is increasing and continuous, there exists $n_0\in\N$ such that  $0<a(\delta_n,\rho^+,\alpha)<a(\delta,\rho^+,\alpha)=1
$ for all $n\geq n_0$. Let $(a_n)_{n\in\N}$ be such that $a_n\in (a(\delta_n,\rho^+,\alpha),a(\delta_n,\rho,\alpha)]\cap [0,1]$ for every $n\geq n_0$. Then, $(a_n)_{n\in\N}$ converges to $a$, since $a(\delta_n,\rho^+,\alpha)$ converges to  $a(\delta,\rho^+,\alpha)=a=1$. As $(\delta_n,a_n)$ satisfies \eqref{CONDTHAB} for all $n\geq n_0$, we have $\rho(F_{\delta_n,a_n})=\rho$ for all $n\geq n_0$ and $\rho(F_{\delta,a})=\rho$, by Lemma \ref{ContRot}. 

\medskip
\noindent
{\bf Case 3:} Suppose $a(\delta,\rho,\alpha)\leq 0$. 
As $a\in[0,1]$ satisfies \eqref{CONDTH},  we have $a(\delta,\rho,\alpha)= 0=a$. 
If $\delta\in (\delta(\rho^-,\alpha),\delta(\rho,\alpha)]$, then $\phi(0)> 0$  by Lemma \ref{PROPPHI0}, which implies $\phi(\alpha)\geq \phi(0)>0$ by Proposition \ref{PHISCR} and $a(\delta,\rho,\alpha)>0$ by Lemma \ref{PROPPHIA}, which is a contradiction.  It follows that $\delta=\delta(\rho^-,\alpha)$. So, we have $\rho(F_{\delta,a})=\rho$, since $\delta$ and $a$ satisfy \eqref{CONDTHAB}.
\end{proof}

\section{Proof of Proposition \ref{RANGE}}\label{PP}

We start with studying  the position of interval $[\delta(\rho^-,\alpha),\delta(\rho,\alpha)]$ according to the position of $\alpha$ with respect to $1-\rho$. 

\begin{lemma}\label{INTERVALDELTA} For any $\rho\in(0,1)$ we have $[\delta(\rho^-,\alpha),\delta(\rho,\alpha)]\subset (1-\lambda-d,1-d)$ if $\alpha\in[0,1-\rho)$ and  $[\delta(\rho^-,\alpha),\delta(\rho,\alpha)]\subset (1-\lambda,1)$ if  $\alpha\in(1-\rho,1]$. 
\end{lemma}

\begin{proof} Let $\rho\in(0,1)$ and $\alpha\in[0,1-\rho)$.  By Lemma \ref{RANGEDELTA} we have  that $\delta(\rho^{-},\alpha)>1-\lambda-d$.
Now, let $k_1=\ent{\frac{1}{1-\rho}}$. Then $\ent{k\rho}=k-1$ for all $k\in \N$ such that $1\le k\le k_1$ and $\ent{k\rho}\le k-2$ for all $k\ge k_1+1$. Since $\rho<1-\alpha$, we have
\[
\{k\rho\}=k\rho-\ent{k\rho}=k\rho-k+1<1-k\alpha\le 1-\alpha\qquad \forall k\in\{1,\dots,k_1\},
\] 
and using  $d<1-\lambda$,  we obtain
\begin{align*}
		\delta(\rho,\alpha)&=1-\lambda-d+\frac{1-\lambda}{\lambda}\left(\sum_{k=1}^{\infty} \lambda^k((1-\lambda)\ent{k\rho}+d\theta_{1-\alpha}(\{k\rho\})) \right)\\
		&<1-\lambda-d+\frac{1-\lambda}{\lambda}\left(
		\sum_{k=1}^{k_1}\lambda^k(1-\lambda)(k-1) + \sum_{k=k_1+1}^{\infty}\lambda^k(1-\lambda)(k-2)+(1-\lambda)\sum_{k=k_1+1}^{\infty} \lambda^k \right)\\
		&=1-\lambda-d+\frac{1-\lambda}{\lambda}\left((1-\lambda)\sum_{k=1}^{\infty}\lambda^k(k-1)-(1-\lambda)\sum_{k=k_1+1}^{\infty}\lambda^k + (1-\lambda)\sum_{k=k_1+1}^{\infty}\lambda^k \right)\\
		&=1-\lambda-d+\frac{1-\lambda}{\lambda}\left((1-\lambda)\lambda\sum_{k=1}^{\infty}\lambda^{k-1} k-(1-\lambda)\sum_{k=1}^{\infty}\lambda^k\right)\\
		&=1-\lambda-d+\frac{1-\lambda}{\lambda}\left(\frac{\lambda}{1-\lambda}-\lambda\right)\\
		&=1-d.
	\end{align*}

Let $\rho\in(0,1)$ and $\alpha\in(1-\rho,1]$. By Lemma \ref{RANGEDELTA} we have $\delta(\rho,\alpha)<1$.
Now, let $k_1:=\max\{k\geq 1 : \ent{k\rho}=0\}$, then we can write
\[
\delta(\rho,\alpha)
=(1-\lambda)\left(1-\frac{d}{1-\lambda}+
\frac{1}{\lambda}\sum_{k=1}^{k_1}\lambda^kd\theta_{1-\alpha}(k\rho)
+\frac{1}{\lambda}\sum_{k=k_1+1}^{\infty}\lambda^k\left((1-\lambda)\lf k\rho\rf+d\theta_{1-\alpha}(\{k\rho\})\right)\right).
\]
As $\alpha>1-\rho$, we have $\{k\rho\}=k\rho>1-\alpha$ for every $1\leq k\leq k_1$. Also, for $k\geq k_1+1$ we have $\ent{k\rho}\geq 1$, $\theta_{1-\alpha}(\{k\rho\})\geq 0$ and using $1-\lambda>d$ we obtain 
\[
\delta(\rho,\alpha)>(1-\lambda)\left(1-\frac{d}{1-\lambda}+
\frac{d}{\lambda}\sum_{k=1}^{k_1}\lambda^k
+\frac{d}{\lambda}\sum_{k=k_1+1}^{\infty}\lambda^k\right)=1-\lambda.
\]
We have proved that $\delta(\rho,\alpha)>1-\lambda$ for every $\rho\in(0,1)$ and $\alpha\in(1-\rho,1)$. Now, let $\rho'\in(0,1)$ be such that $1-\alpha<\rho'<\rho$. Then, $1-\lambda<\delta(\rho',\alpha)$ and $\delta(\rho',\alpha)\leq \delta(\rho^{-},\alpha)$ by monotony, see Lemma \ref{RANGEDELTA}. It follows that $\delta(\rho^{-},\alpha)>1-\lambda$.
\end{proof}

The next two lemmata will help us to study in the forthcoming Proposition \ref{INTERVALA} the position of $\phi(\alpha^-)$ and $\phi(\alpha)$ with respect to $\eta_1$ and $\eta_2$ as a function of the position of $\alpha$ with respect to $1-\rho$.  
\begin{lemma}\label{LPHI1MR} Let $\rho\in(0,1)$, $\alpha\in[0,1]$ and $\delta\in(1-\lambda-d,1)$.  For $\alpha\neq 1-\rho$, we have
\begin{equation}\label{ETAPHIDIF}
\phi((1-\rho)^-)
=\eta+\frac{\delta-\delta(\rho,\alpha)}{\lambda(1-\lambda)}
\quand
\phi(1-\rho)=\eta+\frac{\delta-\delta(\rho^-,\alpha)}{\lambda(1-\lambda)},
\end{equation}
where
\begin{equation}\label{ETAALPHA}
\eta:=\left\{
\begin{array}{ccl}
\eta_1 &\text{if} &  \alpha\in[0,1-\rho)\\
\eta_2 &\text{if} &  \alpha\in(1-\rho,1]   
\end{array}
\right..
\end{equation}
 For $\alpha=1-\rho$, we have
\begin{equation}\label{ETAPHIEQUAL}
\phi((1-\rho)^-)=\eta_2+\frac{\delta-\delta(\rho,\alpha)}{\lambda(1-\lambda)}
\quand
\phi(1-\rho)=\eta_1+\frac{\delta-\delta(\rho^-,\alpha)}{\lambda(1-\lambda)}.
\end{equation}

\end{lemma}

\begin{proof}  Suppose $\rho\in(0,1)$, $\alpha\in[0,1]$ and $\delta\in(1-\lambda-d,1)$. Using \eqref{TPHI}, \eqref{PHYSYM} and \eqref{PHI0PHI0M}, we obtain on one hand  
\begin{align*}
\phi((1-\rho)^-)&=1+\phi((-\rho)^-)\\
&=1+\frac{\delta+d}{1-\lambda}-1-\frac{1}{\lambda}\sum_{k=1}^{\infty}\lambda^k\psi_{1-\alpha}(k\rho+\rho)\\
&=\frac{\delta+d}{1-\lambda}-\frac{1}{\lambda}\left(\frac{1}{\lambda}\sum_{k=1}^{\infty}\lambda^k\psi_{1-\alpha}(k\rho)-\frac{\delta+d}{1-\lambda}+1-\psi_{1-\alpha}(\rho)+\frac{\delta+d}{1-\lambda}-1
\right)\\
&=\frac{\delta+d}{1-\lambda}-\frac{1}{\lambda}\left(-\phi(0^{-})-\psi_{1-\alpha}(\rho)+\frac{\delta+d}{1-\lambda}-1
\right)\\
&=\frac{1-\delta-d}{\lambda}+\frac{1}{\lambda}\left(\phi(0^{-})+\psi_{1-\alpha}(\rho)\right)\\
&=\frac{1-\delta-d}{\lambda}+\frac{1}{\lambda}\left(\phi(0^{-})+(1-\lambda)\ent{\rho}+d\theta_{1-\alpha}(\{\rho\})
\right)\\
&=\frac{1-\delta-d+d\theta_{1-\alpha}(\rho)}{\lambda}+\frac{\delta-\delta(\rho,\alpha)}{\lambda(1-\lambda)}
\end{align*}
and on the other hand 
\begin{align*}
\phi(1-\rho)
&=1+\phi(-\rho)\\
&=1+\frac{\delta}{1-\lambda}
+\frac{1}{\lambda}\sum_{k=1}^{\infty}\lambda^k\psi_{\alpha}(-(k+1)\rho)\\
&=1+\frac{\delta}{1-\lambda}+\frac{1}{\lambda}\left(\frac{1}{\lambda}\sum_{k=1}^{\infty}\lambda^k\psi_{\alpha}(-k\rho)+\frac{\delta}{1-\lambda}-\psi_{\alpha}(-\rho)-\frac{\delta}{1-\lambda}
\right)\\
&=1+\frac{\delta}{1-\lambda}+\frac{1}{\lambda}\left(\phi(0)-\psi_{\alpha}(-\rho)-\frac{\delta}{1-\lambda}
\right)\\
&=1+\frac{\delta}{1-\lambda}+\frac{1}{\lambda}\left(\phi(0)-(1-\lambda)\ent{-\rho}-d\theta_{\alpha}(\{-\rho\})-\frac{\delta}{1-\lambda}
\right)\\
&=1+\frac{\delta}{1-\lambda}+\frac{1-\lambda}{\lambda}-\frac{\delta}{\lambda(1-\lambda)}+
\frac{1}{\lambda}\left(\phi(0)-d\theta_{\alpha}(\{-\rho\})\right)\\
&=\frac{1-\delta}{\lambda}+
\frac{1}{\lambda}\left(\phi(0)-d\theta_{\alpha}(1-\rho)\right)\\
&=\frac{1-\delta-d\theta_{\alpha}(1-\rho)}{\lambda}+\frac{\delta-\delta(\rho^-,\alpha)}{\lambda(1-\lambda)}.
\end{align*}
Then, the result follows from \eqref{DETA1ETA2} and \eqref{ETAALPHA}.
\end{proof}

\begin{lemma}\label{ETARHO} Let $\rho\in[0,1]$, $\alpha\in[0,1]$ and $\delta\in(1-\lambda-d,1)$. Let $\eta$ be defined as in \eqref{ETAALPHA}. 
\begin{enumerate}
\item\label{ETARHO1} If $\delta\in[\delta(\rho^-,\alpha), \delta(\rho,\alpha)]$, then 
\begin{enumerate}
\item for $\alpha\neq 1-\rho$, we have
\[
\phi(y)\leq\eta\quad\forall y\in[0,1-\rho)\quand \eta\leq\phi(y)\quad\forall y\in[1-\rho,1),
\]
\item and for $\alpha=1-\rho$, we have
\[
\phi(y)\leq\eta_2\quad\forall y\in[0,1-\rho)\quand \eta_1\leq\phi(y)\quad\forall y\in[1-\rho,1),
\]
\end{enumerate}
where the first inequality of {\it (a)} and {\it (b)} are strict if $\rho\in(0,1)\setminus\Q$ or $\delta\neq\delta(\rho,\alpha)$.
\item Conversely, if $\alpha\neq 1-\rho$ and
\[
\phi(y)<\eta\quad\forall y\in[0,1-\rho)\quand \eta\leq\phi(y)\quad\forall y\in[1-\rho,1),
\]
or if $\alpha=1-\rho$ and
\[
\phi(y)<\eta_2\quad\forall y\in[0,1-\rho)\quand \phi(y)\geq\eta_1\quad\forall y\in[1-\rho,1),
\]
then, $\delta\in[\delta(\rho^-,\alpha), \delta(\rho,\alpha)]$ with  $\delta\neq\delta(\rho,\alpha)$
if $\rho\in(0,1)\cap\Q$.
\end{enumerate}
\end{lemma}
\begin{proof} 
The proof relies on the expressions of $\phi((1-\rho)^-)$ and $\phi(1-\rho)$ given in Lemma \ref{LPHI1MR}.

\noindent $1)$ Suppose $\alpha\neq 1-\rho$ and let $y\in[0,1)$. As $\delta\leq\delta(\rho,\alpha)$, we have $\phi((1-\rho)^-)\leq\eta$ with strict inequality for $\delta\neq\delta(\rho,\alpha)$. Therefore, it follows from Proposition \ref{PHISCR} that $\phi(y)\leq\eta$ if $y<1-\rho$ and the inequality is strict if $\delta\neq\delta(\rho,\alpha)$ or $\rho\in(0,1)\setminus\Q$. On the other hand, as $\delta\geq\delta(\rho^-,\alpha)$, we have $\phi(1-\rho)\geq\eta$ and $\phi(y)\geq \eta$ for $y\geq 1-\rho$. The proof of the case $\alpha=1-\rho$ is analogous.  

\noindent $2)$ Suppose $\alpha\neq 1-\rho$. If $\phi(y)<\eta$ for every 
$y\in[0,1-\rho)$, then $\phi((1-\rho)^-)\leq \eta$. If $\rho\in(0,1)\cap\Q$, then there exists $y\in[0,1-\rho)$ such that $\phi(y)=\phi((1-\rho)^-)$, by Lemma \ref{LIMINFIM}. It follows that $\phi((1-\rho)^-)\leq \eta$ and that 
$\phi((1-\rho)^-)<\eta$ if $\rho\in(0,1)\cap\Q$. We deduce that $\delta\leq\delta(\rho,\alpha)$ with  $\delta\neq\delta(\rho,\alpha)$ if $\rho\in(0,1)\cap\Q$. On the other hand, if $\phi(y)\geq\eta$ for every 
$y\in[1-\rho,1)$, then $\phi((1-\rho))\geq \eta$ and $\delta\geq\delta(\rho^-,\alpha)$. The proof of the case $\alpha=1-\rho$ is analogous.
\end{proof}

\begin{corollary}\label{01COD} Let $\rho\in[0,1]$ and $\alpha\in[0,1]$. If $\alpha\neq 1-\rho$, then
$\phi([0,1))\subset [0,1)$ if and only if $\theta_\eta(\phi(y))=\theta_{1-\rho}(y)$ for all $y\in[0,1)$, where $\eta=\eta_1$ if $\alpha\in[0,1-\rho)$ and $\eta=\eta_2$ if $\alpha\in[1-\rho,1)$. If $\alpha= 1-\rho$, then $\phi([0,1))\subset [0,1)$ if and only if $\phi(y)<\eta_2$ for all $y\in[0,1-\rho)$ and $\phi(y)\geq\eta_1$ for all $y\in[1-\rho,1)$.
\end{corollary}
\begin{proof} By Lemma \ref{PROPPHI0} and \ref{INTERVALCERO} we have  $\phi([0,1))\subset [0,1)$ if and only if $\delta\in[\delta(\rho^-,\alpha),\delta(\rho,\alpha)]$, with $\delta\neq\delta(\rho,\alpha)$ if $\rho\in\Q$. The corollary follows then by applying Lemma \ref{ETARHO}.
\end{proof}

\begin{proposition}\label{INTERVALA} Let $\rho\in (0,1)$ and $\alpha\in[0,1]$ and $\delta\in[\delta(\rho^-,\alpha), \delta(\rho,\alpha)]$. Denote $\phi=\phi_{\delta,\rho,\alpha}$, then
\begin{enumerate}
\item If $\alpha=0$, then 
\[
\phi(\alpha^-)=\phi(\alpha)=0 \quif \rho\notin\Q \quand \phi(\alpha^-)\leq 0\leq \phi(\alpha)\leq\eta_1\quif \rho\in\Q,
\]
where the first and the last inequalities are strict if $\delta\neq\delta(\rho,\alpha)$.

\item If $\alpha\in(0,1-\rho)$, then 
\begin{equation}\label{INTALPHALESS}
[\phi(\alpha^-), \phi(\alpha)]\subset (0, \eta_1) \quif \rho\notin\Q
\quand
[\phi(\alpha^-), \phi(\alpha)]\subset [0, \eta_1]
\quif  \rho\in\Q. 
\end{equation}
In the case $\rho\in\Q$, we have  $\phi(\alpha)< \eta_1$  if $\delta\neq\delta(\rho,\alpha)$.

\item If $\alpha\in(1-\rho,1)$, then 
\begin{equation}\label{INTALPHAGREATER}
[\phi(\alpha^-), \phi(\alpha)]\subset (\eta_2,1) \quif \rho\notin\Q
\quand
[\phi(\alpha^-), \phi(\alpha)]\subset [\eta_2,1]
\quif  \rho\in\Q. 
\end{equation}
In the case $\rho\in\Q$, we have $\eta_2<\phi(\alpha^-)$ if $\delta\neq\delta(\rho^-,\alpha)$ and $\phi(\alpha)< 1$ if $\delta\neq\delta(\rho,\alpha)$.

\item If $\alpha=1$, then 
\[
\phi(\alpha^-)=\phi(\alpha)=1 \quif \rho\notin\Q \quand \eta_2\leq\phi(\alpha^-)\leq1 \leq\phi(\alpha)\quif \rho\in\Q,
\]
where the first inequality is strict if $\delta\neq\delta(\rho^-,\alpha)$ the second one  if $\delta\neq\delta(\rho,\alpha)$.

\item If $\alpha=1-\rho$, then 
\[
\phi(\alpha^-)=\phi(\alpha)\in[\eta_1,\eta_2] \quif \rho\notin\Q \quand [\phi(\alpha^-),\phi(\alpha)]\cap[\eta_1,\eta_2]\neq\emptyset\quif \rho\in\Q,
\]
with $\phi(\alpha^-)=\eta_2$  if $\delta=\delta(\rho,\alpha)$ and $\phi(\alpha)= \eta_1$  if $\delta=\delta(\rho^-,\alpha)$.
\end{enumerate}
\end{proposition}

\begin{proof} As $\delta\in [\delta(\rho^-,\alpha),\delta(\rho,\alpha)]$, part of the proof consists in applying Item 1 of Lemma \ref{ETARHO} to $y=\alpha$, to obtain the positions of $\phi(\alpha^-)$ and $\phi(\alpha)$ with respect to $\eta_1$ and $\eta_2$.  Also, by Lemma \ref{PROPPHI0} we have $\phi(0^-)\le 0\le \phi(0)$ and $\phi(0^-)< 0$ if $\delta\neq\delta(\rho,\alpha)$. Keeping this in mind, let us prove the different items:

\noindent 1) Suppose $\alpha=0$. If $\rho\notin\Q$, then $\phi$ is continuous at $\alpha$, by Proposition \ref{PHISCR}. Since $\phi(0^-)\le 0\le \phi(0)$, it follows that $\phi(0)=0$. If $\rho\in\Q$, it remains to prove that $\phi(0)\leq\eta_1$,
with strict inequality if $\delta\neq\delta(\rho,\alpha)$, which  follows directly from Lemma \ref{ETARHO}.

\noindent 2) Suppose that $\alpha\in (0,1-\rho)$. If $\rho\notin \Q$, then $\phi$ is increasing and we obtain 
\[
0\leq\phi(0)<\phi(\alpha^-)\le \phi(\alpha)<\eta_1.
\] 
If $\rho\in \Q$, then  $\phi$ is non-decreasing and discontinuous at $\alpha$, therefore   
\[
0\le \phi(0)\leq\phi(\alpha^-)<\phi(\alpha)\leq\eta_1,
\]
where the last inequality is strict if  $\delta\neq\delta(\rho,\alpha)$ by Lemma \ref{ETARHO}.

\noindent 3) For $\alpha\in (1-\rho,1)$, using the same arguments that in the previous case, for $\rho\notin \Q$ we obtain that 
\[
\eta_2\le \phi(1-\rho)< \phi(\alpha^-)
\leq \phi(\alpha)<\phi(1^-)=\phi(0^-)+1\leq 1
\]
and for $\rho\in \Q$  
\[
\eta_2\le \phi(1-\rho)\leq \phi(\alpha^-)
< \phi(\alpha)\leq\phi(1^-)\leq 1,
\]
By \eqref{ETAPHIDIF} the first inequality is strict if $\delta\neq\delta(\rho^-,\alpha)$. The last one is strict if $\phi(0^-)\neq 0$, that is, if $\delta\neq\delta(\rho,\alpha)$.

\noindent 4) Suppose $\alpha=1$. Since, $\phi(1)=\phi(0)+1$ and $\phi(1^-)=\phi(0^-)+1$, we always have $\phi(1^-)\leq 1 \leq\phi (1)$, with $\phi(1^-)<1$ for $\delta\neq\delta(\rho,\alpha)$. If $\rho\notin\Q$, then $\phi$ is continuous at $\alpha$, by Proposition \ref{PHISCR}. It follows that $\phi(1^-)=\phi(1)=1$. If $\rho\in\Q$, it remains to prove that $\phi(1^-)\geq\eta_2$, which  follows directly from $\phi(1^-)\geq\phi(1-\rho)\geq\eta_2$ of Lemma \ref{ETARHO}. Also, by \eqref{ETAPHIDIF}, the inequality is strict if $\delta\neq\delta(\rho^-,\alpha)$.

\noindent 5) Suppose $\alpha=1-\rho$ with $\rho\in(0,1)$. Then, \eqref{ETAPHIEQUAL} gives
\[
\phi(\alpha^-)=\eta_2+\frac{1}{\lambda(1-\lambda)}
(\delta-\delta(\rho,\alpha))\leq\eta_2
\quand
\phi(\alpha)=\eta_1+\frac{1}{\lambda(1-\lambda)}
(\delta-\delta(\rho^-,\alpha))\geq\eta_1.
\]
Therefore, we always have $[\phi(\alpha^-),\phi(\alpha)]\cap[\eta_1,\eta_2]\neq\emptyset$, since $\phi(\alpha^-)\leq\phi(\alpha)$ by monotonicity of $\phi$.
Also, 
\[
\phi(\alpha)-\phi(\alpha^-)=\eta_1-\eta_2+\frac{1}{\lambda(1-\lambda)}
(\delta(\rho,\alpha)-\delta(\rho^-,\alpha))=-\frac{d}{\lambda}+\frac{1}{\lambda(1-\lambda)}
(\delta(\rho,\alpha)-\delta(\rho^-,\alpha)).
\]
If $\rho\in(0,1)\setminus\Q$, then the set $D$ of Lemma \ref{LINTIR} is equal to $\{-1\}$, since 
$\alpha=1-\rho$. From the same lemma we obtain that 
\[
\delta(\rho,\alpha)-\delta(\rho^-,\alpha)=d(1-\lambda).
\]
It follows that $\eta_1\leq\phi(\alpha)=\phi(\alpha^-)\leq\eta_2$.
\end{proof}

\begin{proof}[Proof of Proposition \ref{RANGE}] Let  $\rho\in(0,1)$, $\alpha\in[0,1]$ and choose $\delta\in[\delta(\rho^-,\alpha),\delta(\rho,\alpha)]$. Suppose $\alpha\in[0,1-\rho)$. Then, $\delta\in(1-\lambda-d,1-d)$, by Lemma \ref{INTERVALDELTA}. Also, Items $1(a)$ and $1(b)$ are respectively Items 2) and 1) of Proposition \ref{INTERVALA}, recalling that $\phi(\alpha^-)=a(\delta,\rho^+,\alpha)$ and $\phi(\alpha)=a(\delta,\rho,\alpha)$, see Lemma \ref{PROPPHIA}.
Suppose $\alpha\in(1-\rho,1]$. Then, $\delta\in(1-\lambda,1)$, by Lemma \ref{INTERVALDELTA}. Also, Items $2(a)$ and $2(b)$ are respectively Items 3) and 4) of Proposition \ref{INTERVALA}.
Finally, for $\alpha=1-\rho$, we have $\delta\in(1-\lambda-d,1)$, by Lemma \ref{RANGEDELTA}
and Item 3) is Item 5) of Proposition \ref{INTERVALA}.
\end{proof}

\section{Proof of Theorem \ref{CODDING2}}\label{PTH2}

We prove Theorem \ref{CODDING2} in the next three subsections.

\subsection{Proof of the equivalences}

We start with a remark which gathers some results obtained in the previous sections  that will help us to prove  the equivalence of the assertions 1), 2) and 3) of Theorem \ref{CODDING2}.

\begin{remark}\label{REMEQU} Until now, we have proved in Proposition \ref{CNSCONJ} and Proposition \ref{PROROT}, that for any $\rho\in(0,1)$, $\alpha\in[0,1]$, $\delta\in(0,1)$ and $a\in[0,1]$ the following are equivalent:
\begin{enumerate}[(a)]
\item $F\circ\phi=\phi\circ L_\rho$.
\item $\phi([0,1))\subset[0,1)$ and $\theta_{a}(\phi(y))=\theta_{\alpha}(y)$
for every $y\in[0,1)$.
\item $\delta\in[\delta(\rho^-,\alpha),\delta(\rho,\alpha))$ and
$a\in(a(\delta,\rho^+,\alpha),a(\delta,\rho,\alpha)]$, with closed intervals if $\rho\in(0,1)\setminus\Q$.
\end{enumerate}
\end{remark}

Now, we start the proof. In the following $1)$, $2)$ and $3)$ refer to the items of Theorem \ref{CODDING2}.

\medskip
\noindent {\bf $1)\Rightarrow 2)$}. If \eqref{CONDTHIR} holds,  then by (a) and (c) of Remark \ref{REMEQU}, we have 
\[
F\circ\phi(y)=\phi\circ L_\rho(y)\quad\forall y\in[0,1).
\]
Also, $\phi([0,1))\subset[0,1)$  by (b) and (c) of Remark \ref{REMEQU}. It follows that $\{\phi(y)\}=\phi(\{y\})$ for all $y\in\R$ by  Lemma \ref{PHIEF}.
So, we have that
\[
f\circ\phi(y)=\{F\circ\phi(y)\}=\phi(\{L_\rho(y)\})=\phi\circ R_\rho(y)\quad\forall y\in[0,1),
\]
which proves \eqref{CONJUGf} and item $2)$.

\noindent $2)\Rightarrow 3).$ Recall that for any $\alpha\in[0,1]$ we have
\[
F(\phi(y))= (1-\lambda)(\lf \phi(y)\rf -\ent{y})+ d(\theta_a(\{\phi(y)\})-\theta_\alpha(\{y\}))  +\phi(y+\rho)\quad\forall y\in\R,
\]
see Lemma \ref{CONJUG1}. As $\phi([0,1))\subset [0,1)$ when $2)$ holds, we have that $\{\phi(y)\}=\phi(\{y\})$ and $\ent{\phi(y)}=\ent{y}$, therefore \eqref{CONJUGf} implies that
\[
\{F(\phi(y))\}=\{d(\theta_a(\phi(y))-\theta_\alpha(y))  +\phi(y+\rho)\} =\phi(\{y+\rho\})=\{\phi(y+\rho)\}\quad\forall y\in[0,1).
\]
It follows that $d(\theta_a(\phi(y))-\theta_\alpha(y))\in\Z$ for every $y\in[0,1)$, which occurs only if 
\begin{equation}\label{CODDEM}
\theta_a(\phi(y))=\theta_\alpha(y)\quad\forall y\in[0,1),
\end{equation}
since $d\notin\Z$. 

Now, still because $\phi([0,1)\subset[0,1)$, by Corollary \ref{01COD}, if $\alpha\neq 1-\rho$ we have

\begin{equation}\label{CODET}
\theta_{\eta}(\phi(y))=\theta_{1-\rho}(y)\quad\forall y\in[0,1)\quad \text{with}\quad
\eta=\begin{cases}
\eta_1&\text{if}\quad\alpha\in[0,1-\rho)\\
\eta_2&\text{if}\quad \alpha\in(1-\rho,1]
\end{cases}
\end{equation}
and if $\alpha= 1-\rho$ we have
\begin{equation}\label{COD1MR}
\phi(y)<\eta_2\quad\forall y\in[0,1-\rho)\quand \phi(y)\geq\eta_1\quad\forall y\in[1-\rho,1).
\end{equation}
Suppose $(\delta,a)\in M_1$. If $\alpha\in(1-\rho,1]$, then, for 
$y\in[1-\rho,\alpha)$ we have $\theta_{\eta_2}(\phi(y))=\theta_{1-\rho}(y)=1$, by \eqref{CODET}, and on the other hand $\theta_{a}(\phi(y))=\theta_{\alpha}(y)=0$, by \eqref{CODDEM}. It follows that $\eta_2\leq \phi(y)<a$, which is a contradiction, since $a<\eta_1$ when $(\delta,a)\in M_1$ and $\eta_1<\eta_2$.
We deduce that $\alpha\in[0,1-\rho]$. If $\alpha\in[0,1-\rho)$, then
\begin{equation}\label{CODDEMETA}
\theta_{\eta}(\phi(y))=\theta_{1-\rho}(y)\quad\forall y\in[0,1)\quad \text{with}\quad
\eta=\eta_1,
\end{equation}
by \eqref{CODET}. If $\alpha= 1-\rho$,  by \eqref{CODDEM} and \eqref{COD1MR}, we have that
\[
\phi(y)<a<\eta_1\quad\forall y\in[0,1-\rho)\quand \phi(y)\geq\eta_1\quad\forall y\in[1-\rho,1).
\]
It follows that \eqref{CODDEMETA} holds also if $\alpha=1-\rho$.
We conclude that if $(\delta,a)\in M_1$, then \eqref{CODDEM} and 
\eqref{CODDEMETA}  hold with $\eta=\eta_1$ and $\alpha\in[0,1-\rho]$. Together with \eqref{CONJUGf}, this proves \eqref{CODES} in the case where $(\delta,a)\in M_1$. The proof of \eqref{CODES} in the case where $(\delta,a)\in M_2$ or
$(\delta,a)\in M_3$ is analogous.

\medskip
\noindent $3)\Rightarrow 1).$ By hypothesis we have
\[
\theta_{a}(\phi(y))=\theta_{\alpha}(y)\quad\forall y\in[0,1).
\]
So, by (b) and (c) of Remark \ref{REMEQU}, we only have to prove that $\phi([0,1))\subset[0,1)$.

Suppose $(\delta,a)\in M_1$. Then,  by hypothesis, we have $\alpha\in[0,1-\rho]$ and
\[
\theta_{\eta_1}(\phi(y))=\theta_{1-\rho}(y)\quad\forall y\in[0,1),
\]
that is, 
\[
\phi(y)<\eta_1<\eta_2\quad\forall y\in[0,1-\rho)\quand \phi(y)\geq\eta_1\quad\forall y\in[1-\rho,1).
\]
Applying Corollary \ref{01COD} we obtain that $\phi([0,1))\subset[0,1)$ as desired. We have proved that
$3)\Rightarrow 1)$ for  $(\delta,a)\in M_1$. The same can be shown with an analogous proof  for $(\delta,a)\in M_3$. Finally, if $(\delta,a)\in M_2$, then
$\alpha=1-\rho$ and 
\[
\theta_{a}(\phi(y))=\theta_{\alpha}(y)=\theta_{1-\rho}(y)\quad\forall y\in[0,1),
\]
that is, 
\[
\phi(y)<a\leq\eta_2\quad\forall y\in[0,1-\rho)\quand \phi(y)\geq a\geq \eta_1\quad\forall y\in[1-\rho,1),
\]
since $a\in[\eta_1,\eta_2]$. From Corollary \ref{01COD} we deduce that $\phi([0,1))\subset[0,1)$ and we conclude that $3)\Rightarrow 1)$ also for $(\delta,a)\in M_2$.

\subsection{Image of $[0,1)$ by $\phi$}

When \eqref{CONDTHIR} holds, the relations \eqref{CONJUGf} and \eqref{CODES} hold on the set $\phi([0,1))$. In this section we first give the structure of this set. Then, supposing that $\delta$ and $a$ satisfy \eqref{CONDTHIR} we show how it is related to 
the asymptotic set $\Lambda=\bigcap_{n\in\N}f^n_{\delta,a}([0,1))$. 

By Lemma \ref{RATPHI}, we already know that $\phi([0,1))$ is a finite set if $\rho$ is rational. The following lemma shows that its closure is a Cantor set if $\rho$ is irrational.

\begin{lemma}\label{PHICANT} Let $\delta\in[0,1]$, $\rho\in(0,1)\setminus\Q$, $\alpha\in[0,1]$ and denote $\phi=\phi_{\delta,\rho,\alpha}$. Then, the closure of $\phi([0,1))$ is a Cantor set. Moreover, if $\{\alpha+k\rho\}\neq 0$ for every $k\geq 1$ and $\delta$ satisfy \eqref{CONDTHIR}, then the Cantor set extremities are the points $0$ and $1$. 
\end{lemma}
\begin{proof} Let us show that $\phi([0,1))$ has no isolated point. Let $x_0\in \phi([0,1))$ and $y_0\in[0,1)$ be such that $x_0=\phi(y_0)$. As $\phi$ is right continuous at $y_0$ and increasing, for every $\epsilon>0$ there exists $\delta>0$ and $y\in(y_0,y_0+\delta)$ such that $\phi(y)\neq\phi(y_0)=x$ and $|x_0-\phi(y)|<\epsilon$. It follows that $\overline{\phi([0,1))}$ has no isolated point.

It remains to show that $\overline{\phi([0,1))}$ is totally disconnected. Let $A\subset\overline{\phi([0,1))}$ be a non-empty connected set and suppose that  $a:=\inf A<b:=\sup A$.   Since $A$ is connected $A\neq\{a,b\}$. Thus, there exists $x\in(a,b)\subset A$ and $y\in[0,1)$ such that $\phi(y)=x$. As the sequence $(\{l\rho\})_{l\in\N}$ is dense in $[0,1)$ and $\phi$ is right continuous
and increasing, there exists $l\geq 1$ such that $y<\{l\rho\}$ and $\phi(\{l\rho\}^-),\phi(\{l\rho\})\in(a,b)$. Thus,
\[
A\subset[a,\phi(\{l\rho\}^-)]\cup(\phi(\{l\rho\}^-),\phi(\{l\rho\}))\cup[\phi(\{l\rho\}),b].
\]
Also, by \eqref{JUMPPSI}
\[
\phi(\{l\rho\})-\phi(\{l\rho\}^-)=\frac{1}{\lambda}
\sum_{k=1}^\infty\lambda^k((1-\lambda-d)\Ind_{\Z}(\{l\rho\}-k\rho)+d\Ind_{\Z+\alpha}(\{l\rho\}-k\rho))>0,
\]
since $\{l\rho\}-k\rho\in\Z$ for $k= l$. Therefore, $(\phi(\{l\rho\}^-),\phi(\{l\rho\}))$ is a gap of $\phi$ and 
\[
A=(A\cap[a,\phi(\{l\rho\}^-)])\cup(A\cap[\phi(\{l\rho\}),b])
\]
is a not a connected set. It follows that if $A$ is connected, then  $\inf A=\sup A$.

Finally, if $\{\alpha+k\rho\}\neq 0$ for every $k\geq 1$, then $\delta(\rho,\alpha)=\delta(\rho^-,\alpha)$ by Lemma \ref{PROPPHI0}. If $\delta$ satisfies \eqref{CONDTHIR} for such an $\alpha$, the same lemma implies that $\phi(0)=0=\phi(0^-)$ and $\phi(1^-)=\phi(0^-)+1=1$.
\end{proof}

Now, let us suppose that \eqref{CONDTHIR} holds. Then, as expected from the construction of $\phi$, we can verify that $\phi([0,1))\subset\Lambda$ in the following lemma.

\begin{lemma}\label{PHIINCLA} Let $(\delta,a)\in M$ and denote $f=f_{\delta, a}$ the map \eqref{PROJ}. Let $\rho\in(0,1)$, $\alpha\in[0,1]$ and denote $\phi=\phi_{\delta,\rho,\alpha}$. If $\delta$ and $a$ satisfy \eqref{CONDTHIR}, then 
\[
\phi([0,1))\subset  \bigcap_{n\ge 1} f^n([0,1)).
\]
\end{lemma}
\begin{proof} Let $x\in \phi([0,1))$ and $y\in[0,1)$ be such that $x=\phi(y)$. Consider the
sequence $(y_n)_{n\in\N}$ defined by $y_n=\{y-n\rho\}$ for every $n\in\N$.
As $\delta$ and $a$ satisfy \eqref{CONDTHIR}, by \eqref{CONJUGf} we have
\[
f^n(\phi(y_n))=\phi(R^n_\rho(y_n))=\phi(y)=x\quad\forall n\in \N.
\]
We have proved that for every $n\in\N$ there exists $x_n:=\phi(y_n)\in[0,1)$ such that 
$x=f^n(x_n)$. 
\end{proof}

In the next subsections, we show that under some conditions on $\rho$ and $\alpha$ it also holds that $\Lambda\subset\phi([0,1))$, and therefore $\phi([0,1))=\Lambda$.

\subsubsection{Image of $\phi$ and attractor of $f$ for $\rho$ irrational}

\begin{lemma}\label{DMDP} Let $(\delta,a)\in M$ and denote $f=f_{\delta, a}$ the map \eqref{PROJ}. Let $\rho\in(0,1)\setminus\Q$, $\alpha\in[0,1]$ and denote $\phi=\phi_{\delta,\rho,\alpha}$. Suppose that $\delta$ and $a$ satisfy \eqref{CONDTHIR} and define the set 
\[
D^-_{\alpha,\rho}:=\{k\geq 1:\{\alpha-k\rho\}=0\} 
\quand D^+_{\alpha,\rho}:=\{k\geq 1:\{\alpha+k\rho\}=0\}.
\]
Then, $D^-_{\alpha,\rho}$ and $D^+_{\alpha,\rho}$ contain at most one element and 
$D^-_{\alpha,\rho}=\emptyset$ or $D^+_{\alpha,\rho}=\emptyset$. Moreover,
\begin{enumerate}
\item For every $k\geq 1$ such that $k\notin D^-_{\alpha,\rho}$ we have
\begin{equation}\label{H1}
f([\phi_{\alpha}(\{k\rho\}^-),\phi_{\alpha}(\{k\rho\}))=[\phi_{\alpha}(\{(k+1)\rho\}^-),\phi_{\alpha}(\{(k+1)\rho\})).
\end{equation}

\item For every $k\geq 1$ such that $k\notin D^+_{\alpha,\rho}$ and $k+1\notin D^+_{\alpha,\rho}$ we have
\begin{equation}\label{H2}
f([\phi_{\alpha}(\{k\rho+\alpha\}^-),\phi_{\alpha}(\{k\rho+\alpha\}))=[\phi_{\alpha}(\{(k+1)\rho+\alpha\}^-),\phi_{\alpha}(\{(k+1)\rho+\alpha\})).
\end{equation}
\item If $D^-_{\alpha,\rho}=\emptyset$, then $\phi(\alpha^-)=\phi(\alpha)=a$, with $\alpha\in(0,1)$ if and only if $a\in(0,1)$. Also, for any $a\in(0,1)$ we have
\begin{equation}\label{G2}
[f(a^-),f(a))=[\phi_{\alpha}(\{\rho+\alpha\}^-)),\phi_{\alpha}(\{\rho+\alpha\})
 \quif \alpha\neq 1-\rho
\end{equation}
and
\begin{equation}\label{GH} 
[0,f(a))\cup[f(a^-),1)=[0,\phi(0))\cup[\phi(1^-),1) \quif \alpha=1-\rho.
\end{equation}
\item If $D^+_{\alpha,\rho}=\emptyset$, then $\phi(0^-)=\phi(0)=0$ and $[f(1^-),f(0))=[\phi_{\alpha}(\rho^-),\phi_{\alpha}(\rho))$.
\end{enumerate}
\end{lemma}

\begin{proof} As $\rho\in(0,1)\setminus\Q$, the set $D^-_{\rho,\alpha}\cup
D^+_{\rho,\alpha}=\{k\in\Z^*:\{k\rho+\alpha\}=0\}$ contains at most one element. Therefore, $D^-_{\rho,\alpha}=\emptyset$ if $D^+_{\rho,\alpha}\neq\emptyset$ and $D^+_{\rho,\alpha}=\emptyset$ if $D^-_{\rho,\alpha}\neq\emptyset$.

Let $\beta\in[0,1]$ and $k\geq 1$. Let us suppose that $\{k\rho+\beta\}\notin\{0,1-\rho,\alpha\}$ and that $\{(k+1)\rho+\beta\}\neq 0$. Let us show that 
\begin{equation}\label{FPB}
f([\phi_{\alpha}(\{k\rho+\beta\}^-),\phi_{\alpha}(\{k\rho+\beta\}))=[\phi_{\alpha}(\{(k+1)\rho+\beta\}^-),\phi_{\alpha}(\{(k+1)\rho+\beta\})).
\end{equation}
Since\eqref{CONDTHIR} holds, we have also \eqref{CODES} which implies that
\[
\phi(\alpha^-)\leq a \leq\phi(\alpha)\quand \phi((1-\rho)^-)\leq \eta \leq\phi(1-\rho),
\]
$\eta=\eta_1$ if $\alpha\in[0,1-\rho)$, $\eta=\eta_2$ if $\alpha\in(1-\rho,1]$ and $\eta=a$ if $\alpha=1-\rho$. As $\phi$ is increasing and $\{k\rho+\beta\}\notin\{\alpha,1-\rho\}$, the interval $[\phi_{\alpha}(\{k\rho+\beta\}^-),\phi_{\alpha}(\{k\rho+\beta\})$ does not contain $\eta$ or $a$. Therefore, $f$
is increasing and continuous in this interval. We deduce that
\[
f([\phi_{\alpha}(\{k\rho+\beta\}^-),\phi_{\alpha}(\{k\rho+\beta\}))=[f(\phi_{\alpha}(\{k\rho+\beta\}^-)),f(\phi_{\alpha}(\{k\rho+\beta\})).
\] 
Now, as $\{k\rho+\beta\}\neq 0$, there exists a positive sequence $(\epsilon_n)_{n\in\N}$ converging to zero and such that $0\leq\{k\rho+\beta\}-\epsilon_n<1$ for all $n\in\N$.
Then, using the continuity of $f$ and \eqref{CONJUGf}, we obtain
\[
f(\phi_{\alpha}(\{k\rho+\beta\}^-))=\lim_{n\to\infty}f(\phi_{\alpha}(\{k\rho+\beta\}-\epsilon_n))=\lim_{n\to\infty}\phi_{\alpha}(R_\rho(\{k\rho+\beta\}-\epsilon_n))=\lim_{n\to\infty}\phi_{\alpha}(\{(k+1)\rho+\beta-\epsilon_n\}).
\]
As $(k+1)\rho+\beta\notin\Z$, for $n$ large enough, we have $\ent{(k+1)\rho+\beta-\epsilon_n}=\ent{(k+1)\rho+\beta}$. It follows that 
\[
f(\phi_{\alpha}(\{k\rho+\beta\}^-))=\lim_{n\to\infty}\phi_{\alpha}(\{(k+1)\rho+\beta\}-\epsilon_n)=\phi_{\alpha}(\{(k+1)\rho+\beta\}^-),
\]
which ends to prove \eqref{FPB}, since $f(\phi_{\alpha}(\{k\rho+\beta\}))=
\phi_{\alpha}(\{(k+1)\rho+\beta\})$ follows directly from \eqref{CONJUGf}.
Now, let us prove the items of the lemma.

\medskip
\noindent 1) Let $k\geq 1$ be such that $k\notin D_{\alpha,\rho}^-$. Then $\{k\rho\}\notin\{0,1-\rho\}$ and 
$\{(k+1)\rho\}\neq 0$ since $\rho\in(0,1)\setminus\Q$. Also, if $\alpha\in\{0,1\}$, then $\{k\rho\}\neq\alpha$. If $\alpha\in(0,1)$, then 
$\{k\rho\}\neq \alpha$ since $\{\alpha-k\rho\}\neq 0$. We deduce that 
\eqref{FPB} holds for $\beta=0$. 

\medskip
\noindent 2) Let $k\geq 1$ be such that $k$ and $k+1\notin D_{\alpha,\rho}^+$.
Then, $\{(k+1)\rho+\alpha\}\neq 0$ and $\{k\rho+\alpha\}\notin\{0,1-\rho\}$. Also,
$\{k\rho+\alpha\}\neq\alpha$ since $\rho\in(0,1)\setminus\Q$. We deduce that 
\eqref{FPB} holds for $\beta=\alpha$. 

\medskip
\noindent 3) If $D_{\alpha,\rho}^-=\emptyset$, then $k\rho\notin\Z+\alpha$ for every $k\geq 1$. It follows from Lemma \ref{PROPPHIA} and \eqref{CONDTHIR} that $\phi(\alpha^-)=\phi(\alpha)=a$. 
Also, by Proposition \ref{RANGE}, we have $\alpha\in(0,1)$ if and only if $a\in(0,1)$, since $\rho\in(0,1)\setminus\Q$. Now, suppose $a\in(0,1)$ and let $(\epsilon_n)_{n\in\N}$ be a positive sequence converging to zero and such that $0\leq\alpha-\epsilon_n<1$ for all $n\in\N$. Then, $\phi(\alpha-\epsilon_n)\in[0,a)$ for every $n\in\N$ and $\phi(\alpha-\epsilon_n)\to a$. Therefore,
\[
f(a^-)=\lim_{n\to \infty} f(\phi(\alpha-\epsilon_n))=\lim_{n\to \infty}\phi(\{\alpha+\rho-\epsilon_n\})=\begin{cases}
\phi(\{\alpha+\rho\}^-)& \quif \alpha\neq 1-\rho\\
\phi(1^-)& \quif \alpha=1-\rho
\end{cases}.
\]
On the other hand, $f(a)=f(\phi(\alpha))=\phi(\{\alpha+\rho\})$, which allows to conclude.

\medskip
\noindent 4) If $D_{\alpha,\rho}^+=\emptyset$, then $-k\rho\notin\Z+\alpha$ for every $k\geq 1$. It follows from Lemma \ref{PROPPHI0} that $\delta(\rho^-,\alpha)=\delta(\rho,\alpha)$ which implies that $\phi(0^-)=\phi(0)$ and  $\phi(0)=0$ by \eqref{CONDTHIR}. Now, let $(\epsilon_n)_{n\in\N}$ be a positive sequence converging to zero and such that $0\leq 1-\epsilon_n<1$ for all $n\in\N$. Then, $\phi(1-\epsilon_n)\in[0,1)$ for every $n\in\N$ and $\phi(1-\epsilon_n)\to 1$, since $\phi(1^-)=\phi(0^-)+1=\phi(1)$. Therefore,
\[
f(1^-)=\lim_{n\to \infty} f(\phi(1-\epsilon_n))=\lim_{n\to \infty}\phi(\{1-\epsilon_n+\rho\})=\lim_{n\to \infty}\phi(\{1+\rho\}-\epsilon_n)=\phi(\rho^-).
\]
As $f(0)=f(\phi(0))=\phi(\rho)$, we obtain the desired result.
\end{proof}

\begin{proposition}\label{ATTIR} Let $(\delta,a)\in M$ and denote $f=f_{\delta, a}$ the map \eqref{PROJ}. Let $\rho\in(0,1)\setminus\Q$, $\alpha\in[0,1]$ and denote $\phi=\phi_{\delta,\rho,\alpha}$. If $\delta$ and $a$ satisfy \eqref{CONDTHIR}, then,
\begin{equation}\label{CantorSet}
\bigcap_{n\ge 1} f^n([0,1))=\phi([0,1)).
\end{equation} 
\end{proposition}

\begin{proof} We already know by Lemma \ref{PHIINCLA} that $\phi([0,1))\subset\Lambda$. So it remains to show that $\Lambda\subset\phi([0,1))$. We note that
for $a\in(0,1)$, we have
\begin{equation}\label{DECOMPAT}
\Lambda=[0,1)\setminus\bigcup_{j\geq 0}G_1^j\cup G_2^j,
\quad\text{where}\quad
G_1^j=f^j([f(1^-),f(0)))\quand G_2^j=f^j(G_2^0),
\end{equation}
with $G_2^0=[f(a^-),f(a))$ if $f$ is of the type \eqref{MAP2} or \eqref{MAP} and $G_2^0=[0,f(a))\cup[f(a^-),1)$ if $f$ is of the type \eqref{MAP3}. 
Now, if $a\in\{0,1\}$, then \eqref{DECOMPAT} still holds but with $G_2^0=\emptyset$. Also, by  Proposition \ref{RANGE}, we have $a\in\{0,1\}$ if and only if $\alpha\in\{0,1\}$. 

On the other hand,  $\phi([0,1))\subset[0,1)$ and the gaps of $\phi$ can only be $H_0=[0,\phi(0))\cup[\phi(1^-),1)$ and the sets
\[
H_1^k=[\phi(\{k\rho\}^-),\phi(\{k\rho\}))
\quand H_2^k=[\phi(\{k\rho+\alpha\}^-),\phi(\{k\rho+\alpha\}))\quad\forall k\geq 1,
\]
by Proposition \ref{PHISCR}.
It follows that
\[
\phi([0,1))=[0,1)\setminus\bigcup_{k\geq 1}H_1^k\cup H_2^k\cup H_0.
\]
Also, for any $\alpha\in[0,1]$, we have always $\phi(\alpha^-)\leq a \leq\phi(\alpha)$ and $\phi((1-\rho)^-)\leq \eta \leq\phi(1-\rho)$, with $\eta=\eta_1$ if $\alpha\in[0,1-\rho)$, $\eta=\eta_2$ if $\alpha\in(1-\rho,1]$ and $\eta=a$ if $\alpha=1-\rho$.

So, we have to show that all the sets $H$ are contained in the union of the sets $G$. We distinguish three cases using the sets $D^-_{\rho,\alpha}$ and $D^{+}_{\rho,\alpha}$ defined in Lemma \ref{DMDP}. In the following 1), 2), 3) and 4) refer to the items of Lemma \ref{DMDP}.

\medskip
\noindent
{\bf Case 1:} $D^-_{\rho,\alpha}=D^{+}_{\rho,\alpha}=\emptyset$. Then,  $\phi(0^-)=\phi(0)=0$ and
$\phi(1^-)=\phi(0^-)+1=1$ by 4). It follows that $H_0=\emptyset$. Also, by 4) and 1) we have
\[
H_1^1=G_1^0\quand H_1^{k}=f(H_1^{k-1})=f^{k-1}(G_1^0)=G_1^{k-1}\quad\forall k\geq 1.
\]
If $a\in\{0,1\}$, then $\alpha\in\{0,1\}$ and $H_2^k=H_1^k$ for every $k\geq 1$. Thus, there is nothing more to prove. If $a\in(0,1)$, then $\alpha\in(0,1)$. Also, $D^{+}_{\rho,\alpha}=\emptyset$ implies $\alpha\neq 1-\rho$. It follows from 3) and 2) that 
\[
H_2^1=G_2^0\quand H_2^{k}=f(H_2^{k-1})=f^{k-1}(G_2^0)=G_2^{k-1}\quad\forall k\geq 1.
\]

\medskip
\noindent
{\bf Case 2:} $D^-_{\rho,\alpha}\neq\emptyset$. Then, $D^+_{\rho,\alpha}=\emptyset$ and there exists a unique $l\geq 1$ such that $\{\alpha-l\rho\}=0$. This implies $\alpha\notin\{0,1-\rho,1\}$ and $\{k\rho\}=\{(k-l)\rho+\alpha\}$ for every $k\geq 1$. So, we have
\[
\bigcup_{k\geq 1}H_1^k\cup H_2^k\cup H_0=\left(\bigcup_{k= 1}^{l}H_1^k\right)
\cup\left(\bigcup_{m\geq 1}H_2^m\cup H_0\right)\quad\text{with}\quad H_1^l=[\phi(\alpha^-),\phi(\alpha)).
\]
Now, as $D^+_{\rho,\alpha}=\emptyset$, we have  $\phi(0^-)=\phi(0)=0$ and $\phi(1^-)=\phi(0^-)+1=1$, by 4). It follows that $H_0=\emptyset$. On the other hand, as $\phi(\alpha^-)\leq a\leq\phi(\alpha)$, we have
\[
H_1^{l}=[\phi(\alpha^-),a)\cup[a, \phi(\alpha)).
\]
As $\alpha\neq 1-\rho$ and $\phi((1-\rho)^-)\leq\eta\leq\phi(1-\rho)$, we have $\eta\notin H_1^{l}$ since  $\phi$ is increasing. Therefore, the map $f$ is increasing and continuous in both intervals. It follows that
\[
f(H_1^{l})=[f(\phi(\alpha^-)), f(a^-))\cup[f(a),f(\phi(\alpha)))=[\phi(\{\rho+\alpha\}^-), f(a^-))\cup[f(a),\phi(\{\rho+\alpha\})=H_2^1\setminus G_2^0.
\]
and $H_2^1=f(H_1^{l})\cup G_2^0$. On the other hand,  item 2) ensures that $H_2^{m+1}=f(H_2^{m})$ for all $ m\geq 1$, since $D^+_{\rho,\alpha}=\emptyset$. So, we obtain   
\[
H_2^m=f^m(H_1^l)\cup G_2^{m-1}\quad\forall m\geq 1.
\]
Still because $D^+_{\rho,\alpha}=\emptyset$, item 4) gives 
\[
H_1^1=G_1^0.
\]
So, if $l=1$, then 
\[
H_1^1=G_1^0=H_1^{l},
\quand H_2^{m}=G_1^{m}\cup G_2^{m-1}\quad\forall m\geq 1,
\]
and we are done. Now, if $l\geq2$, then $k\notin D^-_{\alpha,\rho}$ and $H^{k+1}_1=f(H_1^{k})$ for all $k\in\{1,\dots,l-2\}$, by 1). This implies
\[
H_1^k=f^{k-1}(H_1^1)=G_1^{k-1}\quad \forall k\in\{2,\dots,l\}.
\]
In particular $H^l_1=G_1^{l-1}$ and we deduce that $H_2^{m}=G_1^{l-1+m}\cup G_2^{m-1}$ for all $m\geq 1$.

\medskip
\noindent
{\bf Case 3:} $D^+_{\rho,\alpha}\neq\emptyset$. Then, $D^-_{\rho,\alpha}=\emptyset$ and there exists a unique $l\geq 1$ such that $\{l\rho+\alpha\}=0$. This implies $\alpha\notin\{0,1\}$ and $\{k\rho+\alpha\}=\{(k-l)\rho\}$ for every $k\geq 1$. Therefore,  
\[
\bigcup_{k\geq 1}H_1^k\cup H_2^k\cup H_0=\left(\bigcup_{k= 1}^{l}H_2^k\cup H_0\right)
\cup\left(\bigcup_{m\geq 1}H_1^m\right)\quad\text{with}\quad H_2^l\cap[0,1)=[0,\phi(0))\subset H_0.
\]
As $\alpha$ and $1-\rho\notin\{0,1\}$, we have 
$\eta,a\in(\phi(0),\phi(1^-))$ and in both intervals defining $H_0$ the map $f$ is increasing and continuous. Therefore, 
\[
f(H_0)=[f(\phi(1^-)), f(1^-))\cup[f(0),f(\phi(0)))=[\phi(\rho^-), f(1^-))\cup[f(0),\phi(\rho))=H_1\setminus G_1^0.
\]
So we have $H_1^1=f(H_0)\cup G_1^0$. On the other hand,  by 1) we have $H_1^{m+1}=f(H_1^{m})$ for all $ m\geq 1$, since $D^-_{\rho,\alpha}=\emptyset$. So we obtain,   
\[
H_1^m=f^m(H_0)\cup G_1^{m-1}\quad\forall m\geq 1.
\]
Still because $D^-_{\rho,\alpha}=\emptyset$ and $\alpha\notin\{0,1\}$, item 3) ensures 
\[
H_2^1=G_2^0\quif \alpha\neq 1-\rho\quand H_0= G_2^0\quif \alpha=1-\rho.
\]
If $l=1$, that is $\alpha=1-\rho$, we have 
\[
H_2^l\cap[0,1)\subset H_0=G_2^0 \quand H_1^m=G_2^m\cup G_1^{m-1}\quad\forall m\geq 1,
\]
and we are done. Suppose $l\geq 2$, that is $\alpha\neq1-\rho$. Then, $\{\alpha-l\rho\}=0$ implies $\{(l-1)\rho+\alpha\}=\{1-\rho\}$ and 
\[
H_2^{l-1}=[\phi((1-\rho)^-), \phi(1-\rho))=[\phi((1-\rho)^-), \eta)\cup[\eta,\phi(1-\rho)),
\]
since $\phi((1-\rho)^-)\leq\eta\leq\phi((1-\rho))$.  Also $a=\phi(\alpha)$ and  we have $a\notin H_2^{l-1}$, since $\alpha\neq 1-\rho$ and $\phi$ is increasing. Therefore, the map $f$ is increasing and continuous in both intervals and
\[
f(H_2^{l-1})=[f(\phi((1-\rho)^-)), f(\eta^-))\cup[f(\eta),f(\phi(1-\rho)))=[\phi(1^-), 1)\cup[0,\phi(0))=H_0.
\]
So, if $l=2$, then $H_0=f(H_2^1)=G_2^1$ and we conclude as before. Finally, if $l>2$, then
$H^{k+1}_2=f(H_2^{k})$ for all $k\in\{1,\dots,l-2\}$, by 2). This implies
\[
H_2^k=f^{k-1}(H_2^1)=G_2^{k-1}\quad \forall k\in\{2,\dots,l-1\}.
\]
It follows that we have also
\[
H_2^l\cap[0,1)\subset H_0=f(H_2^{l-1})=G_2^{l-1} \quand H_1^m=G_2^{m+l-1}\cup G_1^{m-1}\quad\forall m\geq 1,
\]
which ends the proof.
\end{proof}

\subsubsection{Image of $\phi$ and attractor of $f$ for $\rho$ rational}

\begin{lemma}\label{IMAGEPHIRAT} Let $(\delta,a)\in M$ and denote $f=f_{\delta, a}$ the map \eqref{PROJ}. Let $\rho=p/q\in (0,1)$ with $0<p<q$ co-prime natural numbers, $\alpha\in[0,1]$ and denote $\phi=\phi_{\delta,\rho,\alpha}$. Then, $\phi([0,1))$ is the union of the two sets 
\[
P_1:=\{\sigma_0,\dots,\sigma_{q-1}\} \quand P_2:=\{\gamma_0,\dots,\gamma_{q-1}\},
\]
which have the following properties whenever $\delta$ and $a$ satisfy \eqref{CONDTHIR}.
\begin{enumerate}
\item If $\alpha\notin \{i/q: i=0,\dots,q\}$, the sets $P_1$ and $P_2$ are two $q$-periodic orbits of  $f$ such that
\begin{equation}\label{RATPART}
0\le \sigma_0<\gamma_0<\cdots<\sigma_{q-1}<\gamma_{q-1}<1,
\end{equation}
with $\eta\in (\gamma_l,\sigma_{l+1}]$ for some  $l\in\{0,\dots,q-2\}$  and $a\in (\sigma_s,\gamma_s]$ for some $s\in\{0,\dots,q-1\}$. Also, for each $i\in \{0,\dots,q-1\}$ there exists  $j_i\in \{0,\dots,q-1\}$ such that
\begin{equation}\label{RATINC}
f([\gamma_i,\sigma_{i+1}))=[\gamma_{j_i},\sigma_{j_i+1}) \quif i\neq l,\quad f([\sigma_i,\gamma_i))=[\sigma_{j_i},\gamma_{j_i}) \quif i\neq s 
\end{equation}
and
\begin{equation}\label{RATINC1}
f([\sigma_s,\gamma_s))\subset[\sigma_{j_s},\gamma_{j_s}).
\end{equation}
On the other hand, 
\begin{equation}\label{RATINC2}
f([\gamma_l,\sigma_{l+1}))=[0,\sigma_0)\cup [\gamma_{q-1},1)\quand f([0,\sigma_0)\cup [\gamma_{q-1},1))\subset [\gamma_{p-1},\sigma_p).
\end{equation}
\item If $\alpha\in \{i/q: i=0,\dots,q\}$, then $P_1=P_2$ and is a $q$-periodic orbit of $f$.
\end{enumerate}
\end{lemma}
\begin{proof} For every $i\in \{0,\dots,q-1\}$ let $\sigma_i:=\phi(i/q)$,  $\gamma_i:=\phi(\alpha_i)$ and $\alpha_i=\{\alpha+i/q\}$. Without loss of generality,  let us assume  that $\alpha_0\in [0,1/q)$ and $\alpha_i=\alpha_0+i/q$ for all $i\in \{1,...,q-1\}$. From Proposition \ref{PHISCR} we know that $\phi$ is non-decreasing and from Lemma \ref{RATPHI} we know that $\phi$ is piecewise constant in the interval $[0,1)$, with a jump at each of the points $0,1/q,\dots(q-1)/q$ and $\alpha_0,\alpha_1,\dots,\alpha_{q-1}$. It follows that $\phi([0,1))=P_1\cup P_2$, with $P_1=P_2$ if
$\alpha\in \{i/q: i=0,\dots,q\}$ and $P_1\cap P_2=\emptyset$ otherwise. Therefore, if  $\alpha\notin \{i/q: i=0,\dots,q\}$ then $\sigma_0<\gamma_0<\cdots<\sigma_{q-1}<\gamma_{q-1}$.

Since $\delta$ and $a$ satisfy \eqref{CONDTHIR}, we have $\phi([0,1))\subset[0,1)$ and the relation \eqref{CONJUGf} holds. On the one hand, we deduce that $0\leq\sigma_0$ and $\gamma_{q-1}<1$, which ends to prove \eqref{RATPART}. On the other hand, for any $i\in \{0,...,q-1\}$
\begin{equation}\label{CONJRAT}
f^k(\sigma_i)=\phi(\{i/q+kp/q\}) \quand f^k(\gamma_i)=\phi(\{\alpha_0+i/q+kp/q\})\quad\forall k\in\N.
\end{equation}
In particular,  $f^q(\sigma_i)=\sigma_i$ and $f^q(\gamma_i)=\gamma_i$. As  $k\mapsto \{\tfrac{i}{q}+k\frac{p}{q}\}$ is a bijection from $\{0,...,q-1\}$ onto $\{0,\frac{1}{q},...,\frac{q-1}{q}\}$ and  $k\mapsto \{\alpha_i+k\frac{p}{q}\}$ is a bijection from $\{0,...,q-1\}$ onto $\{\alpha_0,\alpha_1,...,\alpha_{q-1}\}$, we deduce that the sets $P_1=\{\sigma_0,\dots,\sigma_{q-1}\}$ and $P_2=\{\gamma_0,\dots,\gamma_{q-1}\}$ are $q$-periodic orbits of $f$. In the case where $\alpha\in \{i/q: i=0,\dots,q\}$, we have $P_1=P_2$, which proves item 2).

To prove item 1), let us suppose that $\alpha\notin \{i/q: i=0,\dots,q\}$. Then, $\alpha\neq 1-\rho$ and as \eqref{CONDTHIR} holds we deduce from \eqref{CODES} that
\[
\theta_a(\phi(y))=\theta_{\alpha}(y)
\quand
\theta_\eta(\phi(y))=\theta_{1-\rho}( y)
\quad\forall y\in[0,1).
\]
This implies that $a\in (\phi(\alpha^-),\phi(\alpha)]$ and $\eta\in (\phi((1-\rho)^-),\phi(1-\rho)]$. Therefore $a\in (\sigma_s,\gamma_s]$, for the $s\in \{0,\dots,q-1\}$ such that $\alpha=\alpha_s$. Let $l:=q-p-1\in\{0,\dots,q-2\}$, then $(l+1)/q=1-\rho$, so $\phi(1-\rho)=\phi((l+1)/q)=\sigma_{l+1}$ and $\eta\in (\gamma_l,\sigma_{l+1}]$. 

Let $i\in \{0,\dots,q-1\}$ and $j_i\in \{0,\dots,q-1\}$ be such that $j_i=i+p \pmod{q}$. 
Then, from \eqref{CONJRAT} it follows that
\begin{equation}\label{RATPERMUT0}
f(\sigma_i)=\phi(j_i/q)=\sigma_{j_i} \quand f(\gamma_i)=\phi(\{\alpha_0+j_i/q\})=\phi(\alpha_{j_i})=\gamma_{j_i} \quand f(\sigma_{i+1})=\sigma_{j_{i+1}},
\end{equation}
where $j_{i+1}=i+1+p \pmod{q}$, so $j_{i+1}=j_i+1$ if $i\neq l$ and $(j_0,j_l,j_{l+1},j_{q-1})=(p,q-1,0,p-1)$. 

Now, for any interval of the form $[b,c)\subset[0,1)$ such that $\{\eta,a\}\cap(b,c)=\emptyset$, we have $f([b,c))=[f(b),f(c^-))$ with $f(c^-)=f(c)$ if $c\notin\{\eta,a\}$. 
 Also, on one hand, either $\eta\in(\gamma_l,\sigma_{l+1})$ or $\eta=\sigma_{l+1}$ and on the other hand, either $a\in(\sigma_s,\gamma_s)$ or 
$a=\gamma_s$. Then, from \eqref{RATPART} and \eqref{RATPERMUT0}, we deduce that
\[
f([\gamma_i,\sigma_{i+1}))=[f(\gamma_i),f(\sigma_{i+1}))=[\gamma_{j_i},\sigma_{j_{i+1}})=[\gamma_{j_i},\sigma_{j_i+1}) \quad \text{for all}\quad  i\in  \{0,...,q-2\}\setminus\{l\},
\]
since $j_{i+1}=j_i+1$ if $i\neq l$ and
\[
f([\sigma_i,\gamma_i))= [f(\sigma_{i}),f(\gamma_i))=[\sigma_{j_i},\gamma_{j_i}) \quad \text{for all}\quad  i\in \{0,...,q-1\}\setminus\{s\}, 
\]
which proves \eqref{RATINC}.
Also, $[\sigma_s,\gamma_s)=[\sigma_s,a)\cup[a,\gamma_{s})$ and 
$ f([\sigma_s,\gamma_s))=[\sigma_{j_s},f(a^-))\cup[f(a),\gamma_{j_s})\subset [\sigma_{j_s},\gamma_{j_s})$, since $f(a^{-})<f(a)$. Analogously, 
\[
f([\gamma_l,\sigma_{l+1}))=[f(\gamma_l),f(\eta^-))\cup[f(\eta),f(\sigma_{l+1}))=[\gamma_{j_l},1)\cup [0, \sigma_{j_{l+1}})=[\gamma_{q-1},1)\cup[0,\sigma_0),
\]
as $j_l=q-1$ and $j_{l+1}=0$. Finally, as $a$ and $\eta\notin[0,\sigma_0]$, from \eqref{RATPERMUT0} we obtain
\[
f([0,\sigma_0)\cup [\gamma_{q-1},1))= [f(0),f(\sigma_0))\cup [f(\gamma_{q-1}),f(1^-))=[\gamma_{j_{q-1}},f(1^-))\cup [f(0),\sigma_{j_0})\subset [\gamma_{p-1},\sigma_p),
\] 
since $f(1^-)<f(0)$, $j_{q-1}=p-1$ and $j_0=p$.
\end{proof}

\begin{proposition}\label{ATTRAT} Let $(\delta,a)\in M$ and denote $f=f_{\delta, a}$ the map \eqref{PROJ}. Let $\rho=p/q\in (0,1)$ with $0<p<q$ co-prime natural numbers, $\alpha\in [0,1]\setminus \{\tfrac{i}{q}:i=0,...,q\}$ and denote 
$\phi=\phi_{\delta,\rho,\alpha}$. If $\delta$ and $a$ satisfy \eqref{CONDTHIR}, then 
\[
\bigcap\limits_{n\ge 1} f^n([0,1))=\phi([0,1)).
\]
\end{proposition}

\begin{proof} As $\alpha\notin \{\tfrac{i}{q}: i=0,\dots,q\}$, by Lemma \ref{IMAGEPHIRAT}, $\phi([0,1))$ is the union of two $q$-periodic orbits of $f$ given by $P_1=\{\sigma_0,\dots,\sigma_{q-1}\}$ and $P_2=\{\gamma_0,\dots,\gamma_{q-1}\}$. Let $A_i:=[\sigma_i,\gamma_i)$ for all $i\in \{0,...,q-1\}$, $E_i:=[\gamma_i,\sigma_{i+1})$ for all $i\in \{0,...,q-2\}$ and $E_{q-1}:=[0,\sigma_0)\cup [\gamma_{q-1},1)$. Then, by \eqref{RATPART}, we have the following partition,
\[
[0,1)= A_0\cup E_0\cup\cdots \cup A_{q-2}\cup E_{q-2}\cup A_{q-1}\cup E_{q-1}.
\]
and there exist $l\in\{0,\dots, q-2\}$ and $s\in\{0,\dots, q-1\}$ such that $\eta\in (\gamma_l,\sigma_{l+1}]$ and $a\in (\sigma_s,\gamma_s]$.

Let $i\in \{0,\dots,q-1\}$ with $i\neq s$ and $k_i\in\{1,\dots,q-1\}$ be such that $f^{k_i}(\sigma_i)=\sigma_s$. Then, $f^n(\sigma_i)\neq\sigma_s$ for all $0\leq n<k_i$ and we
deduce from \eqref{RATINC} that $f^{k_i}([\sigma_i,\gamma_i))=[\sigma_s,\gamma_s)=A_s$. It follows that $f^q([\sigma_i,\gamma_i))=f^{q-{k_i}}(A_s)$. Analogously, we deduce from \eqref{RATINC} and \eqref{RATINC2} that for every $i\in \{0,\dots,q-2\}$ with $i\neq l$, there exists $m_i\in\{1,\dots,q-1\}$ such that $f^q([\gamma_i,\sigma_{i+1}))=f^{q-{m_i}}(E_l)$. It follows then from \eqref{RATINC2} and \eqref{RATINC} that
\begin{equation}\label{qIMAGE}
f^q([0,1))\subset \bigcup_{j=1}^q f^j(A_s)\cup f^j(E_l)=\bigcup_{j=1}^q f^j([\sigma_s,a))\cup f^j([a,\gamma_s))\cup f^j([\gamma_{l},\eta))\cup f^j([\eta,\sigma_{l+1})).
\end{equation}
In the sequel of the proof $J:=[b,c)\neq\emptyset$ is one of the interval $[\sigma_s,a),[a,\gamma_s),[\gamma_{l},\eta),[\eta,\sigma_{l+1})$ and $(c_k)_{k\in\N}$ is the sequence defined by  $c_0=c$ and $c_{k+1}=f(c_k^-)$ for every $k\in\N$. 

Let $P:=P_1\cup P_2$. Note that $a\notin P_1$ and  
that  $\eta\in P$ only if $\eta=\sigma_{l+1}$. Therefore, if $\eta\neq\sigma_{l+1}$, then $f$ is continuous in $P_1$ and $c_k=f^k(\sigma_{l+1})$ if $c=\sigma_{l+1}$. Also $\eta\notin P_2$ and $a\in P$ only if $a=\gamma_s$. Therefore, if $a\neq\gamma_s$, then $f$ is continuous in $P_2$ and $c_k=f^k(\gamma_s)$ if $c=\gamma_s$. Let us show that this implies that
\begin{equation}\label{PBORDER}
P\cap\{f^k(b),c_k\}\neq\emptyset\quad\forall k\in\N.
\end{equation}
Indeed, if $f^{k}(b)\notin P$, then $b\notin\{\sigma_s,\gamma_l\}$ and $J=[a,\gamma_s)$ or $J=[\eta,\sigma_{l+1})$. If $J=[a,\gamma_s)\neq\emptyset$, then $a\neq\gamma_s$, and $c_{k}=f^{k}(\gamma_s)\in P$. If $J=[\eta,\sigma_{l+1})\neq\emptyset$, then $\eta\neq\sigma_{l+1}$, and $c_{k}=f^{k}(\sigma_{l+1})\in P$.

Now, let us show by induction that for any $k\in\N$ we have
\begin{equation}\label{fkJ}
f^k(J)=[f^k(b),c_k),\quad|f^k(b)-c_k|\leq\lambda^k|b-c|
\quand\{\eta,a\}\cap(f^k(b),c_k) =\emptyset.
\end{equation}
For $k=0$, the property is true. Let us suppose that it is true for some given $k\in\N$. As $\{\eta,a\}\cap(f^k(b),c_k) =\emptyset$ we deduce that   
$f^{k+1}(J)=[f^{k+1}(b),c_{k+1})$ and $|f^{k+1}(b)-c_{k+1}|\leq\lambda^{k+1}|b-c|$.   It remains to show that $\{\eta,a\}\cap(f^{k+1}(b),c_{k+1}) =\emptyset$.

\noindent{\bf Case 1:} Suppose that $J\subset [\sigma_s,\gamma_s)$. Then by Lemma \ref{IMAGEPHIRAT}, $f^{k+1}(J)=[f^{k+1}(b),c_{k+1})\subset A_j$ for some $j\in\{0,\dots,q-1\}$. Therefore, $\eta\notin(f^{k+1}(b),c_{k+1})$. Also, $a\notin(f^{k+1}(b),c_{k+1})$ if $j\neq s$. So, it only remains to study the case $j=s$, that is, when 
\begin{equation}\label{SIGSGAMS}
[f^{k+1}(b),c_{k+1})\subset[\sigma_s,\gamma_s).
\end{equation}
If $J=[\sigma_s,a)$, then $f^{k+1}(b)=f^{k+1}(\sigma_s)\in P_1$ and it follows from \eqref{SIGSGAMS} that 
$f^{k+1}(b)=\sigma_s$, since $P_1\cap[\sigma_s,\gamma_s)=\{\sigma_s\}$. We deduce that
\[
c_{k+1}-\sigma_s=|f^{k+1}(b)-c_{k+1}|\leq\lambda^{k+1}|b-c|=\lambda^{k+1}|\sigma_s-a|<a-\sigma_s,
\]
which implies that $a>c_{k+1}$ and $a\notin(f^{k+1}(b),c_{k+1})$. If $J=[a,\gamma_s)\neq\emptyset$, then $a\neq\gamma_s$ and $c_{k+1}=f^{k+1}(\gamma_s)\in P_2$. It follows from \eqref{SIGSGAMS} that 
$c_{k+1}=\gamma_s$, since $c_{k+1}\in[\sigma_s,\gamma_s]\cap P_2=\{\gamma_s\}$.
We deduce that
\[
\gamma_s-f^{k+1}(b)=|f^{k+1}(b)-c_{k+1}|\leq\lambda^{k+1}|b-c|=\lambda^{k+1}|a-\gamma_s|<\gamma_s-a,
\]
which implies that $a<f^{k+1}(b)$ and $a\notin(f^{k+1}(b),c_{k+1})$. It follows that \eqref{fkJ} holds for  $J\subset[\sigma_s,\gamma_s)$.

\noindent{\bf Case 2:} The proof for the remaining case $J\subset [\gamma_l,\sigma_{l+1})$ is the same, but replacing $a$, $\eta$, $s$, $\sigma_s$, $\gamma_s$, $A_j$, $P_2$, $P_1$ with $\eta$, $a$, $l$, $\gamma_l$, $\sigma_{l+1}$, $E_j$, $P_1$, $P_2$, respectively.

To end the proof, suppose that $x\in f^n([0,1))$ for all $n\in \N$ and let $k\ge q$. By \eqref{qIMAGE}, there exists an interval $J=[b,c)\neq \emptyset$ such that $x\in f^k(J)$. By \eqref{fkJ},  $f^k(J)=[f^k(b),c_k)$ and $|f^k(b)-c_k|\le \lambda^k |b-c|\le \lambda^k$. As by \eqref{PBORDER} we have $\{f^k(b),c_k\}\cap P\neq \emptyset$,  it follows that
\[
d(x,P):=\inf\{|x-y|: y\in P\}\le |f^k(b)-c_k|\le \lambda^k.
\] 
Since $k$ is arbitrary, we deduce that $x\in P$. Therefore, we have shown that 
\[
\bigcap_{n\ge 1}f^n([0,1))\subset P=\phi([0,1)),
\]
which together with Lemma \ref{PHIINCLA} allows us to conclude.
\end{proof}

\subsection{End of proof of Theorem \ref{CODDING2}}

Suppose that item 1), 2) or 3) holds. Then, by Lemma \ref{PHIINCLA} we have $\phi([0,1))\subset\Lambda$. Suppose $\rho=p/q\in\Q$, with $0<p<q$ co-prime. If $\alpha\in \{i/q: i=0,\dots,q\}$, then $\phi([0,1))$ is a $q$-periodic orbit of $f$, by item 2) of Lemma \ref{IMAGEPHIRAT}. If $\alpha\notin \{i/q: i=0,\dots,q\}$, then $\phi([0,1))$ is the union of two $q$-periodic orbits of $f$, by item 1) of Lemma \ref{IMAGEPHIRAT} and $\phi([0,1))=\Lambda$ by Proposition \ref{ATTRAT}.
Suppose $\rho\in(0,1)\setminus\Q$. Then, $\overline{\phi([0,1))}$ is a Cantor set by Lemma \ref{PHICANT} and $\phi([0,1))=\Lambda$ by Proposition \ref{ATTIR}.

\section{Proof of Theorem \ref{RECTHROT}}\label{PTH3}

\begin{lemma}\label{CONPSIALP}  Recall the function $\psi_{\alpha}:\R\to\R$ defined  in \eqref{PSI} for every $\alpha\in [0,1]$ by  
\[
\psi_\alpha(z)=(1-\lambda)\ent{z}+d\theta_{\alpha}(\{z\})\qquad\forall z\in\R.
\]
Let $z\in\R$. Then, 
\begin{enumerate}
\item  For every $0\leq\alpha_1<\alpha_2\leq 1$,
\begin{equation}\label{PSIALPMON}
\psi_{\alpha_1}(z)-\psi_{\alpha_2}(z)=
\left\{
\begin{array}{cl}
 d & \text{if}\quad \{z\}\in[\alpha_1,\alpha_2)   \\
 0 & \text{otherwise} 
\end{array}
\right..
\end{equation} 
\item $\psi_{\alpha^-}(z)=\psi_{\alpha}(z)$ for all $\alpha\in(0,1]$.
\item $\psi_{\alpha^+}(z)=\psi_{\alpha}(z)$ for every  $\alpha\in[0,1)$ such that $\{z\}\neq\alpha$. 
\end{enumerate}
\end{lemma}

\begin{proof} Let $0\leq\alpha_1<\alpha_2\leq 1$ and $x\in\R$. Then, $\theta_{\alpha_1}(x)-\theta_{\alpha_2}(x)=1$
if and only if $x\in[\alpha_1,\alpha_2)$, otherwise $\theta_{\alpha_1}(x)-\theta_{\alpha_2}(x)=0$.
Then,  \eqref{PSIALPMON} follows from 
\[
\psi_{\alpha_1}(z)-\psi_{\alpha_2}(z)=d(\theta_{\alpha_1}(\{z\})-\theta_{\alpha_2}(\{z\})).
\]

Let $\alpha\in[0,1]$, $z\in\R$ and $(\alpha_n)_{n\in\N}$ be a sequence of $[0,1]$ converging to $\alpha$. Suppose $\alpha\neq 0$ and $\alpha_n<\alpha$ for every $n\in\N$. Then, either $\{z\}\geq\alpha$ or there exists $n_0$ such that 
$\{z\}<\alpha_n$ for every $n\geq n_0$. We deduce from \eqref{PSIALPMON} that $\psi_{\alpha_n}(z)-\psi_{\alpha}(z)\to 0$, that is $\psi_{\alpha^{-}}(z)=\psi_\alpha(z)$. Now, suppose $\alpha\neq1$,  $\alpha<\alpha_n$ for every $n\in\N$ and 
$\{z\}\neq\alpha$. Then, either $\{z\}<\alpha$ or there exists $n_0$ such that 
$\{z\}>\alpha_n$ for every $n\geq n_0$. Once again, from \eqref{PSIALPMON} we obtain that $\psi_{\alpha_n}(z)-\psi_{\alpha}(z)\to 0$, that is $\psi_{\alpha^{+}}(z)=\psi_\alpha(z)$.
\end{proof}


\begin{lemma}\label{ALPHATODELTA} For any $\rho\in[0,1]$ the function  $\alpha\mapsto\delta(\rho,\alpha)$ defined in $[0,1]$ is non-decreasing and  increasing if $\rho\in(0,1)\setminus\Q$. Also, $\alpha\mapsto\delta(\rho,\alpha)$ is right continuous at every $\alpha\in[0,1)$ and continuous at every $\alpha\in[0,1]$ such that $\{k\rho\}\neq 1-\alpha$
for all $k\geq 1$.

\end{lemma}
\begin{proof} Let $\rho\in[0,1]$ and $0\leq\alpha_1<\alpha_2\leq 1$. Then, by \eqref{PSIALPMON} we have
\[
\delta(\rho,\alpha_2)-\delta(\rho,\alpha_1)=\frac{1-\lambda}{\lambda}\sum_{k=1}^{\infty}\lambda^k(\psi_{1-\alpha_2}(\{k\rho\})-\psi_{1-\alpha_1}(\{k\rho\}))\geq 0,
\]
since $0\leq1-\alpha_2<1-\alpha_1\leq 1$. If $\rho\in(0,1)\setminus\Q$, then there 
exists $k\geq 1$ such that $\{k\rho\}\in(1-\alpha_2,1-\alpha_1)$
and from \eqref{PSIALPMON} we deduce that $\delta(\rho,\alpha_2)-\delta(\rho,\alpha_1)>0$.

Now, as
\[
\delta(\rho,\alpha^-)=\frac{1-\lambda}{\lambda}\sum_{k=1}^{\infty}\lambda^k\psi_{(1-\alpha)^+}(\{k\rho\})
\quand 
\delta(\rho,\alpha^+)=\frac{1-\lambda}{\lambda}\sum_{k=1}^{\infty}\lambda^k\psi_{(1-\alpha)^-}(\{k\rho\}),
\]
the continuity properties of $\alpha\mapsto\delta(\rho,\alpha)$ follow from Lemma \ref{CONPSIALP}.
\end{proof}

\begin{proposition}\label{CONTRHO} Let $\delta\in(1-\lambda-d,1)$. The function $\rho_\delta:[0,1]\to[0,1]$ defined
 by 
\begin{equation}\label{RHODELALP}
\rho_\delta(\alpha):=\min\{\rho\in[0,1]:\delta(\rho,\alpha)\geq\delta\}
\qquad\forall \alpha\in[0,1]
\end{equation}
is non-increasing and continuous. Moreover, if $\alpha\in(0,1)$ or $\alpha\in\{0,1\}$ and $\delta\in(1-\lambda-d\Ind_{\{0\}}(\alpha),1-d\Ind_{\{0\}}(\alpha))$, then $\rho_\delta(\alpha)\in(0,1)$ and 
\begin{equation}\label{INTDELRHOALP}
\delta\in[\delta(\rho_\delta(\alpha)^-,\alpha),\delta(\rho_\delta(\alpha),\alpha)].
\end{equation}
\end{proposition}

\begin{proof} Let $\delta\in(1-\lambda-d,1)$. We first show that $\rho_\delta$ is well defined. Let $\alpha\in[0,1]$ and consider the set 
\[
R_{\delta,\alpha}:=\{\rho\in[0,1]:\delta(\rho,\alpha)\geq\delta\}.
\]
By Lemma \ref{RANGEDELTA} we have that 
\[
\lim_{\rho\to 1^-}\delta(\rho,\alpha)=1 \quad\forall \alpha\in(0,1]\quand \delta(1,0)=1+1-\lambda-d>1>\delta.
\]
It follows that $R_{\delta,\alpha}\neq\emptyset$ and there exists $\rho^*\in[0,1]$ such that 
$\rho^*=\inf R_{\delta,\alpha}$. If $\rho^*\notin R_{\delta,\alpha}$, then    $\rho^*\neq 1$ and
$\delta(\rho^*,\alpha)<\delta$. Since $\rho\mapsto\delta(\rho,\alpha)$ is right continuous, $\delta(\rho^*,\alpha)<\delta$ implies that there exists $\rho\in R_{\delta,\alpha}$ such 
that $\delta(\rho,\alpha)<\delta$, which is a contradiction. We deduce that $\rho^*=\min R_{\delta,\alpha}$ and that $\rho_{\delta}(\alpha)$ is well defined.

Now, for any $0\leq\alpha_1<\alpha_2\leq1$ we have $R_{\delta,\alpha_1}\subset R_{\delta,\alpha_2}$
because, by Lemma \ref{ALPHATODELTA}, $\alpha\mapsto\delta(\rho,\alpha)$ is non-decreasing. It follows 
$\rho_{\delta}$ is non-increasing.

Let $\alpha\in(0,1]$ and let us show that $\rho_{\delta}$ is left continuous at $\alpha$. As $\rho_\delta$
is monotonous and non-increasing, it admits a left limit $\rho_{\delta}(\alpha^-)$  such that
$0\leq\rho_{\delta}(\alpha)\leq\rho_{\delta}(\alpha^-)\leq 1$. Suppose that $\rho_{\delta}(\alpha)<\rho_{\delta}(\alpha^-)$.
Then, let $\rho'\in(\rho_{\delta}(\alpha),\rho_{\delta}(\alpha^-))$ be such that $\{k\rho'\}\neq 1-\alpha$
for all $k\geq 1$ (take $\rho'\in\Q$ if $\alpha\notin\Q$ and $\rho'\notin\Q$ if $\alpha\in\Q$). Recalling that $\rho\mapsto\delta(\rho,\alpha)$ is increasing, we obtain that $\delta(\rho',\alpha)>\delta(\rho_{\delta}(\alpha),\alpha)\geq\delta$. Let $(\alpha_n)_{n\in\N}$ be a  sequence of $(0,\alpha)$ converging to $\alpha$. Then, there exists $n_0$ such that $\delta(\rho',\alpha_n)>\delta$ for all $n\geq n_0$, since by Lemma \ref{ALPHATODELTA}  $\alpha'\mapsto\delta(\rho',\alpha')$ is continuous at $\alpha$ for our choice of $\rho'$.
It follows that $\rho'\in R_{\delta,\alpha_n}$ and $\rho_\delta(\alpha_n)\leq\rho'$ if $n\geq n_0$.
We obtain that  $\rho_\delta(\alpha^-)\leq\rho'$, which is a contradiction. So, we have proved that $\rho_{\delta}(\alpha^-)=\rho_{\delta}(\alpha)$. 

Let $\alpha\in[0,1)$ and let us show that $\rho_{\delta}$ is right continuous at $\alpha$. As $\rho_\delta$
is monotonous and non-increasing, it admits a right limit $\rho_{\delta}(\alpha^+)$  such that
$0\leq\rho_{\delta}(\alpha^+)\leq\rho_{\delta}(\alpha)\leq 1$. Suppose that $\rho_{\delta}(\alpha^+)<\rho_{\delta}(\alpha)$ and 
let $\rho'\in(\rho_{\delta}(\alpha^+),\rho_{\delta}(\alpha))$. Then, $0<\rho'<\rho_{\delta}(\alpha)\leq 1$ and $\delta(\rho',\alpha)<\delta$ by definition of $\rho_{\delta}(\alpha)$. Let $(\alpha_n)_{n\in\N}$ be a  sequence of $(\alpha,1)$ converging to $\alpha$. Then, there exists $n_0$ such that $\delta(\rho',\alpha_n)<\delta$ for all $n\geq n_0$, since by Lemma \ref{ALPHATODELTA}  $\alpha'\mapsto\delta(\rho',\alpha')$ is right continuous at $\alpha$.
As $\delta(\rho(\alpha_n),\alpha_n)\geq\delta$ for every $n$ and $\rho\mapsto\delta(\rho,\alpha)$ is increasing, we deduce that $\rho_\delta(\alpha_n)>\rho'$ if $n\geq n_0$.
We obtain that  $\rho_\delta(\alpha^+)\geq\rho'$, which is a contradiction. So, we have proved that $\rho_{\delta}(\alpha^+)=\rho_{\delta}(\alpha)$. 

Suppose that $\rho_\delta(\alpha)=0$ for some $\alpha\in[0,1]$. Then, $\delta(0,\alpha)\geq\delta$, that is
\[
1-\lambda-d+\frac{1-\lambda}{\lambda}d\sum_{k=1}^{\infty}\lambda^k\theta_{1-\alpha}(0)\geq\delta.
\]
As $\delta>1-\lambda-d$, necessarily $\alpha=1$, and $\delta \leq 1-\lambda$. Suppose that $\rho_\delta(\alpha)=1$ for some $\alpha\in[0,1]$. Then, 
$\delta(\rho,\alpha)<\delta$ for every $\rho\in[0, 1)$. As 
\[
\lim_{\rho\to 1^-}\delta(\rho,\alpha)=1 \quad\forall \alpha\in(0,1] \quand \lim_{\rho\to 1^-}\delta(\rho,0)=1-d, 
\]
necessarily $\alpha=0$ and $\delta\in[1-d,1)$. We have show that if $\rho_\delta(\alpha)\in\{0,1\}$ for some $\alpha\in[0,1]$, then $\alpha\in\{0,1\}$
and $\delta\in(1-\lambda-d,1)\setminus (1-\lambda-d\Ind_{\{0\}}(\alpha),1-d\Ind_{\{0\}}(\alpha))$. In other words, we have proved that if $\alpha\in(0,1)$ or $\alpha\in\{0,1\}$ and $\delta\in(1-\lambda-d\Ind_{\{0\}}(\alpha),1-d\Ind_{\{0\}}(\alpha))$, then $\rho_\delta(\alpha)\in(0,1)$.

To end the proof we show \eqref{INTDELRHOALP}. Let $\alpha\in[0,1]$ and $\delta$ be such that $\rho_{\delta}(\alpha)\neq 0$. By definition of $\rho_{\delta}(\alpha)$, we only have to prove that $\delta\geq\delta(\rho_\delta(\alpha)^-,\alpha)$.
Let $(\rho_n)_{n\in\N}$ be a  sequence  converging to $\rho_\delta(\alpha)$ and such that  $0\leq\rho_n<\rho_\delta(\alpha)$ for all $n\in\N$. Then $\rho_n\notin R_{\delta,\alpha}$ and therefore $\delta(\rho_n,\alpha)<\delta$ for all $n\in\N$. We conclude that $\delta(\rho(\alpha)^-,\alpha)\leq\delta$. 
\end{proof}

\begin{proposition}
For every $\delta\in (1-\lambda-d,1)$, let $\Phi_\delta:[0,1]\to \R$ be defined by  
\[
\Phi_\delta(\alpha):=\phi_{\delta,\rho_\delta(\alpha),\alpha}(\alpha)\qquad \forall\alpha\in [0,1],
\]
where 
\[
\phi_{\delta,\rho,\alpha}(y)=\frac{\delta}{1-\lambda}
+\frac{1}{\lambda}\sum_{k=1}^{\infty}\lambda^k\psi_\alpha(y-k\rho) \quad\forall y\in\R
\]
is \eqref{DEFPHI} and $\rho_\delta$ is given in  \eqref{RHODELALP}.
Then, $\Phi_\delta$ is non-decreasing and right continuous.
\end{proposition}

\begin{proof}
Let $\alpha_1<\alpha_2$ in $[0,1]$ and $k\in \N$ fixed. Denote by $z_1:=\alpha_1-k\rho_\delta(\alpha_1)$ and by $z_2:=\alpha_2-k\rho_\delta(\alpha_2)$. Then, 
\begin{align*}
\psi_{\alpha_2}(\alpha_2-k\rho_\delta(\alpha_2))-\psi_{\alpha_1}(\alpha_1-k\rho_\delta(\alpha_1))
&=\psi_{\alpha_2}(z_2)-\psi_{\alpha_1}(z_1)\\
&=(1-\lambda)(\ent{z_2}-\ent{z_1})+d(\theta_{\alpha_2}(\{z_2\})-\theta_{\alpha_1}(\{z_1\}))
\end{align*}
and
\[
z_2-z_1=\alpha_2-\alpha_1+k(\rho_\delta(\alpha_1)-\rho_\delta(\alpha_2))
\ge \alpha_2-\alpha_1>0
\]
because $\rho_\delta(\alpha_2)\le \rho_\delta(\alpha_1)$. If $\ent{z_2}>\ent{z_1}$, then
\[
\psi_{\alpha_2}(z_2)-\psi_{\alpha_1}(z_1)\geq 1-\lambda-d>0.
\]
If $\ent{z_2}=\ent{z_1}$, then
\[
\psi_{\alpha_2}(z_2)-\psi_{\alpha_1}(z_1)=d(\theta_{\alpha_2}(\{z_2\})-\theta_{\alpha_1}(\{z_1\}))
\quand 
\{z_2\}-\{z_1\}=z_2-z_1.
\]
But, $\{z_2\}-\{z_1\}\ge \alpha_2-\alpha_1>0$  implies that $\theta_{\alpha_1}(\{z_1\})\le \theta_{\alpha_2}(\{z_2\})$. So, we have proved that $\psi_{\alpha_2}(\alpha_2-k\rho_\delta(\alpha_2))\geq\psi_{\alpha_1}(\alpha_1-k\rho_\delta(\alpha_1))$ for very $k\geq 1$ and therefore that $\Phi_\delta$ is non-decreasing.

Let us prove that $\Phi_\delta$ is right continuous. Let $\alpha\in[0,1)$ and $(\alpha_n)_{n\in\N}$
be a sequence of $(\alpha,1)$ converging to $\alpha$. Let $k\geq 1$ and let us show that there exists
$n_0\in\N$ such that
\[
\psi_{\alpha}(\alpha-k\rho_\delta(\alpha))=\psi_{\alpha_n}(\alpha_n-k\rho_\delta(\alpha_n))
\qquad\forall n\geq n_0.
\]
Since $\rho_\delta$ is  non-increasing and  continuous we have 
\begin{equation}\label{LIMRNU}
\alpha-k\rho_\delta(\alpha)<\alpha_n-k\rho_\delta(\alpha_n)\quad\forall n\in\N
\quand
\lim_{n\to\infty}\alpha_n-k\rho_\delta(\alpha_n)=\alpha-k\rho_\delta(\alpha).
\end{equation}
Therefore, by right continuity of the integer part, there exists $n_1\in\N$ such that 
\[
\ent{\alpha-k\rho_\delta(\alpha)}=\ent{\alpha_n-k\rho_\delta(\alpha_n)}\qquad\forall n\geq n_1.
\]
If $\{\alpha-k\rho_\delta(\alpha)\}<\alpha$, then from \eqref{LIMRNU} and the right continuity of the fractional part, we deduce that there exists $n_2\in\N$ such that $\{\alpha_n-k\rho_\delta(\alpha_n)\}<\alpha<\alpha_n$ for all $n\geq n_2$. If $\alpha<\{\alpha-k\rho_\delta(\alpha)\}$, then there exists $n_2\geq n_1$ such that 
$\alpha<\alpha_n<\{\alpha-k\rho_\delta(\alpha)\}<\{\alpha_n-k\rho_\delta(\alpha_n)\}$
for all $n\geq n_2$. Therefore, in both cases, for $n_0\geq\max\{n_1,n_2\}$, we have
\begin{equation}\label{DIFTETAS}
\theta_{\alpha}(\{\alpha-k\rho_\delta(\alpha)\})=\theta_{\alpha_n}(\{\alpha_n-k\rho_\delta(\alpha_n)\})
\quand
\psi_{\alpha}(\alpha-k\rho_\delta(\alpha))=\psi_{\alpha_n}(\alpha_n-k\rho_\delta(\alpha_n))
\qquad\forall n\geq n_0.
\end{equation}
Now, if $\{\alpha-k\rho_\delta(\alpha)\}=\alpha$, then $\theta_\alpha(\{\alpha-k\rho_\delta(\alpha)\})=1$ and $k\rho_\delta(\alpha)\in\Z$. For every $n\in\N$, let $\epsilon_n:=\rho_\delta(\alpha)-\rho_\delta(\alpha_n)$. Then $\epsilon_n\geq 0$
and $\epsilon_n\to 0$ by continuity of $\rho_\delta$. Also,
\[
\{\alpha_n-k\rho_\delta(\alpha_n)\}=\{\alpha_n-k\rho_\delta(\alpha)+k\epsilon_n\}=\{\alpha_n+k\epsilon_n\}\qquad \forall n\in\N.
\]
As $\alpha<1$, $\alpha_n\to\alpha^+$ and $\epsilon_n\to 0^+$, there exists $n_2\in\N$ such that
$\alpha<\alpha_n\leq \alpha_n+k\epsilon_n<1$ for all $n\geq n_2$. Therefore, 
\eqref{DIFTETAS} also holds for $n_0\geq\max\{n_1,n_2\}$.

We have proved that $\psi_{\alpha_n}(\alpha_n-k\rho_\delta(\alpha_n))\to \psi_{\alpha}(\alpha-k\rho_\delta(\alpha))$ for every $k\geq 1$, and therefore that $\Phi_\delta$ is right continuous.  
\end{proof}

\begin{lemma}\label{Prop:PHIDPHI}
For every $\delta\in (1-\lambda-d,1)$ and $\alpha\in (0,1]$, we have $
\Phi_\delta(\alpha^-)\geq \phi_{\delta,\rho_\delta(\alpha),\alpha}(\alpha^-)$.
\end{lemma}

\begin{proof} Let $\alpha\in(0,1]$ and $(\alpha_n)_{n\in \N}$ be a sequence in $(0,\alpha)$ converging to $\alpha$. Since $\alpha'\mapsto \rho_\delta(\alpha')$ is  continuous we have that $\rho_\delta(\alpha_n)\to\rho_\delta(\alpha)$ as $n$ goes to infinity. Also observe that 
\[
\Phi_\delta(\alpha^-)=\lim_{n\to \infty} \phi_{\delta,\rho_\delta(\alpha_n),\alpha_n}(\alpha_n) \quand  \phi_{\delta,\rho_\delta(\alpha),\alpha}(\alpha^-)=\lim_{n\to \infty} \phi_{\delta,\rho_\delta(\alpha),\alpha}(\alpha_n),
\] 
since both left limits at $\alpha$ exist by monotony of $\Phi_\delta$ and
$\phi_{\delta,\rho_\delta(\alpha),\alpha}$.  To show that $\Phi_\delta(\alpha^-)\geq \phi_{\delta,\rho_\delta(\alpha),\alpha}(\alpha^-)$, it is enough to prove
that 
\[
\lim_{n\to\infty}\psi_{\alpha_n}(\alpha_n-k\rho_\delta(\alpha_n))
\geq\lim_{n\to\infty}\psi_\alpha(\alpha_n-k\rho_\delta(\alpha))\quad\forall k\geq 1.
\]

Let $k\geq 1$ fixed. First, note that
\[
\alpha_n-k\rho_\delta(\alpha_n)\leq\alpha_n-k\rho_\delta(\alpha)<\alpha-k\rho_\delta(\alpha)\qquad\forall n\in\N,
\]
since $\rho_\delta$ is non-increasing, and that
\[
\lim_{n\to\infty}\alpha_n-k\rho_\delta(\alpha_n)=\lim_{n\to\infty}\alpha_n-k\rho_\delta(\alpha)=\alpha-k\rho_\delta(\alpha).
\]
This implies that there exists $m_0\in\N$ such that
\begin{equation}\label{DESLALP}
l<\alpha_n-k\rho_\delta(\alpha_n)\leq\alpha_n-k\rho_\delta(\alpha)
<\alpha-k\rho_\delta(\alpha)\leq l+1\qquad\forall n\geq m_0,
\end{equation}
where $l=\ent{\alpha-k\rho_\delta(\alpha)}-1$ if $\alpha-k\rho_\delta(\alpha)\in \Z$
and $l=\ent{\alpha-k\rho_\delta(\alpha)}$ otherwise. So, we have that 
\begin{equation}\label{ENTN}
\ent{\alpha_n-k\rho_\delta(\alpha)}=\ent{\alpha_n-k\rho_\delta(\alpha_n)}
=\ent{\alpha-k\rho_\delta(\alpha)}-\Ind_\Z(\alpha-k\rho_\delta(\alpha))\qquad \forall n\ge m_0.
\end{equation}

Now, let us show that there exists $m_1\geq m_0$ such that 
\begin{equation}\label{FRACNU}
\theta_{\alpha_n}(\{\alpha_n-k\rho_\delta(\alpha_n)\})\geq\theta_{\alpha}(\{\alpha_n-k\rho_\delta(\alpha)\})\qquad\forall n\geq m_1.
\end{equation}
If $\alpha=1$, or if there exists $m_1\geq m_0$ such that 
$\theta_{\alpha_n}(\{\alpha_n-k\rho_\delta(\alpha_n)\})=1$ for all $n\geq m_1$, then  \eqref{FRACNU} holds immediately. So, suppose $\alpha\neq 1$ and let us assume that there exist infinitely many $n\geq m_0$ such that $\theta_{\alpha_n}(\{\alpha_n-k\rho_\delta(\alpha_n)\})=0$.
Let $(\alpha_{n_s})_{s\in\N}$ be any subsequence such that $n_s\geq m_0$ and $\theta_{\alpha_{n_s}}(\{\alpha_{n_s}-k\rho_\delta(\alpha_{n_s})\})=0$ for all $s\in\N$. Then,
\begin{equation}\label{ALNS}
\{\alpha_{n_s}-k\rho_\delta(\alpha_{n_s})\}<\alpha_{n_s}\qquad\forall s\in\N.
\end{equation}
If $\alpha-k\rho_\delta(\alpha)\in\Z$. Then by \eqref{ENTN} we have 
$\{\alpha_{n_s}-k\rho_\delta(\alpha_{n_s})\}=\alpha_{n_s}-k\rho_\delta(\alpha_{n_s})-(\alpha-k\rho_\delta(\alpha)-1)<\alpha_{n_s}$
 for all $s\in\N$. Taking the limit $s\to\infty$, we obtain that $\alpha\geq 1$, which is a contradiction. Therefore, $\alpha-k\rho_\delta(\alpha)\notin\Z$ and  taking the limit $s\to\infty$ in \eqref{ALNS}, we deduce that $\{\alpha-k\rho_\delta(\alpha)\}\leq\alpha$. Together with \eqref{DESLALP}, this implies that 
$\{\alpha_{n_s}-k\rho_\delta(\alpha)\}<\alpha$ and $\theta_{\alpha}(\{\alpha_{n_s}-k\rho_\delta(\alpha)\})=0$ for all $s\in\N$. Assuming, without loss of generality, that $\theta_{\alpha_{n}}(\{\alpha_{n}-k\rho_\delta(\alpha_{n})\})=1$
for every $n\geq m_0$ such that $n\notin\{n_s,s\in\N\}$, we obtain \eqref{FRACNU}  with $m_1=m_0$.

Now, \eqref{ENTN} and \eqref{FRACNU} imply that 
\begin{align*}
\psi_{\alpha_n}(\alpha_n-k\rho_\delta(\alpha_n))
&=(1-\lambda)\ent{\alpha_n-k\rho_\delta(\alpha_n)}+d\theta_{\alpha_n}(\{\alpha_n-k\rho_\delta(\alpha_n)\})\\
&\geq (1-\lambda)\ent{\alpha_n-k\rho_\delta(\alpha)}+d\theta_{\alpha}(\{\alpha_n-k\rho_\delta(\alpha)\})=\psi_{\alpha}(\alpha_n-k\rho_\delta(\alpha))
\end{align*}
for all $n\geq m_1$.
\end{proof}

\begin{proposition}\label{PROPTHREC} For every  $\delta\in(1-\lambda-d,1)$ and $a\in[0,1]$ such that
$(\delta,a)\notin \mathcal{F}_1\cup\mathcal{F}_2$
there exist $\rho\in(0,1)$ and $\alpha\in[0,1]$ such that
\begin{equation}\label{INTINV}
\delta\in[\delta(\rho^-,\alpha),\delta(\rho,\alpha)]\quand
a\in[a(\delta,\rho^+,\alpha),a(\delta,\rho,\alpha)].
\end{equation}
\end{proposition}

\begin{proof} Let $\delta\in(1-\lambda-d,1)$ and $a\in[0,1]$. By Proposition \ref{CONTRHO}, we have
\[ 
\rho_\delta(\alpha)\in(0,1)\quand\delta\in[\delta(\rho_\delta(\alpha)^-,\alpha),\delta(\rho_\delta(\alpha),\alpha)]\quad\forall\alpha\in(0,1).
\]
Suppose that $\delta\in[1-d,1)$. Then, Lemma \ref{INTERVALDELTA} implies that $\alpha\geq 1-\rho_\delta(\alpha)>0$ for every $\alpha\in(0,1)$. It follows that $\rho_\delta(0)=1$ and $\Phi_\delta(0)=\phi_{\delta,\rho_\delta(0),0}(0)=\phi_{\delta,1,0}(0)$. As
\[
\phi_{\delta,1,0}(0)=\frac{\delta+d}{1-\lambda}-(1-\lambda)\sum_{k=1}^{\infty}k\lambda^{k-1}=\frac{\delta+d-1}{1-\lambda}\geq 0,
\]
we deduce that $(\delta,a)\in\mathcal{F}_2$ if $a\in[0,\Phi_\delta(0)]$. Now suppose $\delta\in(1-\lambda-d,1-\lambda]$. Then, Lemma \ref{INTERVALDELTA} implies that $\alpha\leq 1-\rho_\delta(\alpha)<1$ for every $\alpha\in(0,1)$. It follows that $\rho_\delta(1^-)=\rho_\delta(1)=0$ and  that $\phi_{\delta,0,1}(1^-)\leq\Phi_\delta(1^-)$ using Lemma \ref{Prop:PHIDPHI}. As
\[
\phi_{\delta,0,1}(1^-)=\frac{\delta+d}{1-\lambda}-1-\frac{1}{\lambda}\sum_{k=1}^{\infty}\lambda^{k}((1-\lambda)\ent{-1}+d\theta_{0}(\{-1\}))=\frac{\delta}{1-\lambda}\leq 1,
\]
we deduce that $(\delta,a)\in\mathcal{F}_1$ if $a\in[\Phi_\delta(1^-),1]$.  We have proved that if $\delta\in(1-\lambda-d,1)$ and $a\in[0,1]$ are such that 
$(\delta,a)\notin \mathcal{F}_1\cup\mathcal{F}_2$, then only three cases are possible:

\medskip
\noindent {\bf Case 1:} $a\in[0,\Phi_\delta(0)]$ and $\delta\in(1-\lambda-d,1-d)$. Then, by \eqref{INTDELRHOALP} we have
\[ 
\rho_\delta(0)\in(0,1)\quand\delta\in[\delta(\rho_\delta(0)^-,0),\delta(\rho_\delta(0),0)]. 
\] 
It follows, from Lemma \ref{PROPPHI0}, that $\phi_{\delta,\rho_\delta(0),0}(0^-)\leq 0\leq a$ and, from definition of $\Phi_\delta$, that
$a\leq\Phi_\delta(0)=\phi_{\delta,\rho_\delta(0),0}(0)$. Then, by Lemma \ref{PROPPHIA}, it follows that \eqref{INTINV} holds for $\alpha=0$ and $\rho=\rho_\delta(0)$.

\medskip
\noindent {\bf Case 2:} $a\in[\Phi_\delta(1^-), 1]$ and  $\delta\in(1-\lambda,1)$. Then, by \eqref{INTDELRHOALP} we have
\[ 
\rho_\delta(1)\in(0,1)\quand\delta\in[\delta(\rho_\delta(1)^-,1),\delta(\rho_\delta(1),1)]. 
\]
It follows, from Lemma \ref{PROPPHI0}, that $\phi_{\delta,\rho_\delta(1),1}(1)=
\phi_{\delta,\rho_\delta(1),1}(0)+1\geq 1\geq a$ and, from Lemma  \ref{Prop:PHIDPHI}, that $\phi_{\delta,\rho_\delta(1),1}(1^-)\leq \Phi_\delta(1^-)\leq a$. Then, by Lemma \ref{PROPPHIA}, it follows that \eqref{INTINV} holds for $\alpha=1$ and $\rho=\rho_\delta(1)$.

\medskip
\noindent {\bf Case 3:} $a\in(\Phi_\delta(0),\Phi_\delta(1^-))$. Then, as $\Phi_\delta(1^-)>a$, the set
\[
A_{\delta,a}:=\{\alpha\in(0,1):\Phi_\delta(\alpha)\geq a\}
\] 
is non-empty and $\alpha^*:=\inf A_{\delta,a}$ is well defined. As  $\Phi_\delta(0)=\Phi_\delta(0^+)<a$ and $\Phi_\delta$ is non-decreasing, there exists $\alpha'\in (0,1)$ such that $\Phi_\delta(\alpha)<a$ for every $0<\alpha<\alpha'$ and $\alpha'$ is lower bound of $A_{\delta,a}$. This implies that $\alpha^*\in(0,1)$ and, by right continuity of $\Phi_\delta$,  we deduce that $\Phi_\delta(\alpha^*)\geq a$, that is, $\alpha^*\in A_{\delta,a}$. Therefore, we can define  
\[
\alpha_{\delta}(a):=\min\{\alpha\in(0,1):\Phi_\delta(\alpha)\geq a\}
\qquad\forall a\in(\Phi_\delta(0),\Phi_\delta(1^-)).
\] 
Thus, for $\alpha=\alpha_{\delta}(a)$ and $\rho=\rho_\delta(\alpha)$, we have $a\in[\Phi_\delta(\alpha^-),\Phi_\delta(\alpha)]$ and by Lemma \ref{Prop:PHIDPHI} and \eqref{INTDELRHOALP}, we have
\[
a\in[\phi_{\delta,\rho,\alpha}(\alpha^-),\phi_{\delta,\rho,\alpha}(\alpha)]
\quand
\delta\in[\delta(\rho^-,\alpha),\delta(\rho,\alpha)],
\]
where $\rho\in(0,1)$. Then, \eqref{INTINV} follows from Lemma \ref{PROPPHIA}. 
\end{proof}

\begin{proof}[Proof of Theorem \ref{RECTHROT}] Let $\delta\in(1-\lambda-d,1)$ and $a\in[0,1]$ be such that $(\delta,a)\notin \mathcal{F}_1\cup\mathcal{F}_2$ and let $f=f_{\delta,a}$ be the map \eqref{PROJ}. Then, by Proposition \ref{PROPTHREC}
there exist $\rho\in(0,1)$ and $\alpha\in[0,1]$ such that such that $\delta$ and $a$ satisfy \eqref{CONDTH}. By Theorem \ref{THROT}, it follows that $\rho=\rho_f$, where $\rho_f$ is the rotation number of $f$, and \eqref{INTREC} holds. If \eqref{INTREC} holds with $\alpha\in(1-\rho_f,1]$, then  $a\geq \eta_2$ or $a= 1$, by item 2) of Proposition \ref{RANGE}. Also $\delta\in(1-\lambda,1)$, which implies that $\eta_2\leq 1$, so in any case $a\geq \eta_2$. For $(\delta,a)\in M_1$, we have $a<\eta_1<\eta_2$ and necessarily  \eqref{INTREC} holds with $\alpha\in[0,1-\rho_f]$. For $(\delta,a)\in M_3$, we can show with an analogous proof using  item 1) of Proposition \ref{RANGE} that \eqref{INTREC} holds with $\alpha\in[1-\rho_f,1]$. If $(\delta,a)\in M_2$ and $\rho_f\notin\Q$, then \eqref{INTREC} holds with $\alpha=1-\rho_f$, by Proposition \ref{RANGE}. If $\rho_f\in\Q$ and $a\notin\{\eta_1,\eta_2\}$, then 
$a\in(\eta_1,\eta_2)$ and item 1) and 2) of Proposition \ref{RANGE} imply that \eqref{INTREC} cannot hold with $\alpha\neq 1-\rho_f$.
\end{proof}

\appendix
\section{Appendix}\label{APPA}

\subsection{Continuity of the rotation number}\label{ACONT}
For sake of completeness, here we recall and apply some results of  \cite{Rhodes1991} to show that the rotation number of the lift \eqref{LIFT} is a continuous function of $(\delta,a)$. We replace the notations of \cite{Rhodes1991} by those of this paper.
 
Let  $\mathcal M:=\{F:\R\to \R : \text{ $F$ is non-decreasing and  $F(t+1)=F(t)+1$ $\forall t\in\R$}\}$. For any $F\in \mathcal M$ let  $R(F)$ be defined by
\[
R(F)=\{(x,y)\in \R^2 :\ F^-(x)\le y\le F^+(x)\}\quad\text{where}\quad F^-(x)=\lim_{t\to x^-} F(t) \quand F^+(x)=\lim_{t\to x^+} F(t).
\] 
As mentioned in \cite{Rhodes1991}, a topology on the set $R(\mathcal M):=\{R(F):\ F\in \mathcal M \}$ is induced by the Hausdorff metric defined on the restrictions of the elements of $R(\mathcal M)$ to the points such that $0\le x\le 1$.

Let  $(F_{\delta,a})$ be a family of functions in $\mathcal{M}$, where $(\delta,a)$ belongs to a subset $M$ of some normed space. For $(\delta_0,a_0) \in M$, we say that $F_{\delta,a}$ is $H$-convergent to $F_{\delta_0,a_0}$ as $(\delta,a) \to (\delta_0,a_0)$, and we write $F_{\delta,a} \xrightarrow{H} F_{\delta_0,a_0}$, if $R(F_{\delta,a})$ converges to $R(F_{\delta_0,a_0})$ in the Hausdorff metric as $(\delta,a) \to (\delta_0,a_0)$. Then, we have the following:

\begin{theorem}{\cite[Theorem 5.8]{Rhodes1991}}\label{TA1} Suppose that $F_{\delta,a}$ is in $\mathcal M$ for all $(\delta,a)$ and that $F_{\delta,a} \xrightarrow{H} F_{\delta_0,a_0}$ as $(\delta,a) \to (\delta_0,a_0)$. If  $F_{\delta_0,a_0}$ is continuous or $F_{\delta_0,a_0}$ is strictly increasing then 
\[
\lim_{(\delta,a)\to (\delta_0,a_0)} \rho(F_{\delta,a})=\rho(F_{\delta_0,a_0}),
\] 
where $\rho(F)$ denotes the rotation number of $F$. 
\end{theorem}

Also, a criterion for the $H$-convergence is given by the following lemma:

\begin{lemma}{\cite[Lemma 3.3]{Rhodes1991}}\label{LA2} Let $(F_{\delta,a})$ be a family of functions in $\mathcal M$. Then $F_{\delta,a}\xrightarrow{H} F_{\delta_0,a_0}$ as $(\delta,a) \to (\delta_0,a_0)$ if and only if for each point $x_0$ of continuity of $F_{\delta_0,a_0}$ and for each $\epsilon>0$ there exists $\xi>0$ such that for every $(\delta,a)$ satisfying $\|(\delta,a)\|<\xi$ there exists $x_{\delta,a}$ such that 
\begin{equation}\label{Hconvergence}
|x_{\delta,a}-x_0|<\epsilon \quand |F_{\delta,a}(x_{\delta,a})-F_{\delta_0,a_0}(x_0)|<\epsilon.
\end{equation}
\end{lemma}

Now, we apply these results to the family of maps \eqref{LIFT}. So,
let $\lambda\in (0,1)$, $d\in (0,1-\lambda)$ and  $M=(1-\lambda-d,1)\times[0,1]$. Consider the family of functions $F_{\delta,a}:\R\to\R$ defined by
\[
F_{\delta,a}(x)=\lambda x+\delta+(1-\lambda)\ent{x}+d\theta_a(\{x\})\quad\forall x\in \R,	 
\]
for every $(\delta,a)\in M$. 
Since  $\lambda\in (0,1)$ and $d\in (0,1-\lambda)$,  any function $F_{\delta,a}$ belongs to the family $\mathcal{M}$ and is strictly increasing.

\begin{lemma}\label{ContRot}
For any $(\delta_0,a_0)\in M$, we have $F_{\delta,a}\xrightarrow{H} F_{\delta_0,a_0}$ as $(\delta,a) \to (\delta_0,a_0)$ and
\begin{equation}\label{CONTROTF}
\lim_{(\delta,a)\to (\delta_0,a_0)} \rho(F_{\delta,a})=\rho(F_{\delta_0,a_0}).
\end{equation}
\end{lemma}

\begin{proof} Let $(\delta_0,a_0)\in M$, $x_0\in\R$ be a continuity point of  $F_{\delta_0,a_0}$ and $\epsilon>0$. As $F_{\delta_0,a_0}$ is discontinuous at $a_0$, we have $\{x_0\}\neq a_0$ and $\xi:=\min \{|\{x_0\}-a_0|/2, \epsilon\}>0$. Let $(\delta,a)\in M$ be such that $\|(\delta,a)-(\delta_0,a_0)\|:=\max \{|\delta-\delta_0|, |a-a_0|\}<\xi$. Then $\{x_0\}<\min \{a,a_0\}$ or $\{x_0\}>\max \{a,a_0\}$, which implies $\theta_a(\{x_0\})=\theta_{a_0}(\{x_0\})$. Thus, $|F_{\delta,a}(x_0)-F_{\delta_0,a_0}(x_0)|=|\delta-\delta_0|<\epsilon$. Taking $x_{\delta,a}=x_0$ we obtain \eqref{Hconvergence} and it follows that $F_{\delta,a}\xrightarrow{H} F_{\delta_0,a_0}$ as $(\delta,a) \to (\delta_0,a_0)$. Now, as $F_{\delta_0,a_0}$ is strictly increasing, 
the limit \eqref{CONTROTF}  follows from Theorem \ref{TA1}.
\end{proof}

\subsection{Semi-conjugacy for irrational rotation numbers}\label{SEMICONJ}

\begin{lemma}\label{SEMICONJ0}
Let $\rho\in (0,1)\setminus \Q$, $\alpha\in [0,1]$ and suppose $\delta$ and $a$ satisfy \eqref{CONDTHIR}. Denote $f=f_{\delta,a}$ the map \eqref{PROJ} and $\phi=\phi_{\delta,\rho,\alpha}$. Let $\tilde\varphi:[\phi(0),\phi(1^-))\to [0,1)$ be defined by $\tilde\varphi(x)=\inf \{y\in [0,1): \phi(y)\ge x\}$ for all $x\in [\phi(0),\phi(1^-))$. Then $\tilde\varphi$ is continuous both with the euclidean and the circle topology, and $\varphi:=\tilde\varphi|_\Lambda:\Lambda \to [0,1)$ is a topological semi-conjugacy from $f|_\Lambda$ to $R_\rho$.
\end{lemma}

\begin{proof}
Since $\phi$ is increasing and right continuous by Proposition \ref{PHISCR}, the function $\tilde\varphi$ is the generalized inverse of $\phi|_{[0,1)}$. It follows that  $\tilde\varphi$ is continuous with the euclidean topology, is non-decreasing, satisfies $\tilde\varphi([\phi(0),\phi(1^-)))=[0,1)$  and  $\tilde\varphi(\phi(y))=y$ for all $y\in [0,1)$. The map $\tilde\varphi$ is also continuous with the circle topology, as
 for any $a<b$ in $(0,1)$ we have $\tilde\varphi^{-1}([0,a)\cup (b,1))=[\phi(0),\phi(a^-))\cup (\phi(b),\phi(1^-))$ and $\tilde\varphi^{-1}((a,b))=(\phi(a),\phi(b^-))$. 
As $\tilde\varphi(\phi([0,1)))=[0,1)$ and $\phi([0,1))\subset [\phi(0),\phi(1^-))$, it follows that
$\tilde \varphi|_{\phi([0,1))}$ is continuous and surjective. By Theorem \ref{CODDING2} we have $\Lambda=\phi([0,1))$, so  it only remains to prove that $\varphi=\tilde\varphi|_\Lambda$ satisfies $\varphi\circ f|_{\Lambda}=R_\rho\circ \varphi$. Let $x\in \Lambda$ and $y\in [0,1)$ be such that $x=\phi(y)$. Using \eqref{CONJUGf}, we obtain $(\varphi\circ f|_\Lambda)(x)=\varphi(f(\phi(y)))=\varphi(\phi(R_\rho(y)))=R_\rho(y)=R_\rho(\varphi(\phi(y)))=(R_\rho\circ \varphi)(x)$. It follows that $\varphi\circ f|_{\Lambda}=R_\rho\circ \varphi$.
\end{proof}

\begin{lemma}\label{SEMICONJ1}
Let $\rho\in (0,1)\setminus \Q$, $\alpha\in [0,1]$ and suppose $\delta$ and $a$ satisfy \eqref{CONDTHIR}. Denote $f=f_{\delta,a}$ the map \eqref{PROJ}. Let $D:=\{k\in \Z^*: k\rho\in \Z+\alpha\}$. Whenever $D\neq\emptyset$, suppose $a\neq a(\delta,\rho^+,\alpha)$ if $D\cap\Z^+\neq\emptyset$, and suppose  $\delta\neq \delta(\rho,\alpha)$ if $D\cap\Z^-\neq\emptyset$. Endow $[0,1)$ with the circle topology and let $\overline \Lambda$ be the closure of $\Lambda$ in the 
circle. Then there exists a topological semi-conjugacy $\bar\varphi$ from $f|_{\overline \Lambda}$ to $R_\rho$. 
\end{lemma}

\begin{proof}{\bf Case 1:} Suppose that $D=\emptyset$ or $D\cap \Z^-=\emptyset$. Then $\delta(\rho^-,\alpha)=\delta(\rho,\alpha)$ by Lemma \ref{LINTIR}. So $\phi(0)=0$ and $\phi(1^-)=\phi(0^-)+1=1$ by \eqref{PHI0PHI0M}. Then, the function $\tilde \varphi$ in Lemma \ref{SEMICONJ0} is defined in the whole circle $[0,1)$, in particular it is well defined in $\overline \Lambda$. As $\tilde\varphi$ is continuous and $\tilde\varphi|_{\Lambda}$ is a semi-conjugacy from $f|_\Lambda$ to $R_\rho$ by Lemma \ref{SEMICONJ0}, it follows that $\bar\varphi:=\tilde\varphi|_{\overline\Lambda}:\overline \Lambda\to [0,1)$ is continuous, surjective and  $\bar\varphi\circ f|_{\Lambda}=R_\rho\circ \bar\varphi|_{\Lambda}$. 

Let $x\in \overline \Lambda\setminus \Lambda$. Note that $0=\{\phi(1^-)\}$ is an accumulation point of $\Lambda$ such that $0=\phi(0)\in \Lambda$, since $\Lambda=\phi([0,1))$ by Theorem \ref{CODDING2}. So $x\neq 0$. As $\phi$ is increasing and right continuous, it follows that $x= \phi(y^-)$ for some $y\in [0,1)$ with $\phi(y^-)<\phi(y)$, that is, $y$ is a point of discontinuity of $\phi$. By definition, $f$ is discontinuous at $0$ and $a$ with the circle topology, so it is continuous at $x$ unless $x=a$. If $D=\emptyset$, then $a(\delta,\rho^+,\alpha)=a=a(\delta,\rho,\alpha)$ by Lemma \ref{LINTIR} and \eqref{CONDTHIR}. So $\phi(\alpha^-)=a=\phi(\alpha)$ by \eqref{PHIAPHIAM} and it follows that $x\neq a$. If $D\neq\emptyset$, then $D\cap \Z^+\neq \emptyset$ and $\phi(\alpha^-)=a(\delta,\rho^+,\alpha)<a(\delta,\rho,\alpha)=\phi(\alpha)$ by Lemma \ref{LINTIR} and \eqref{PHIAPHIAM}. Since $a\neq a(\delta,\rho^+,\alpha)$ by hypothesis, we deduce also that $x\neq a$. Therefore, $f$ is continuous at $x$. 

Let $(x_n)_{n\in \N}\subset \Lambda$ be a sequence converging to $x$. As  $(f(x_n))_{n\in \N}\subset \Lambda$ we have $f(x)=\lim_{n\to \infty} f(x_n)\in \overline \Lambda$. Using the continuity of $\bar\varphi$ and $R_\rho$, and the relation $\bar\varphi\circ f|_{\Lambda}=R_\rho\circ \bar\varphi|_{\Lambda}$, we obtain $(\bar\varphi\circ f)(x)=\lim_{n\to \infty} (\bar\varphi\circ f)(x_n)=\lim_{n\to \infty} (R_\rho\circ \bar\varphi)(x_n)=(R_\rho\circ \bar\varphi)(x)$. It follows that, $\bar\varphi$ is a semi-conjugacy from $f|_{\overline\Lambda}$ to $R_\rho$.

\medskip
\noindent{\bf Case 2:} Suppose that $D\cap \Z^-\neq \emptyset$. Then  $\delta(\rho^-,\alpha)<\delta(\rho,\alpha)$ by Lemma \ref{LINTIR}. So $0\le \phi(0)$ and $\phi(1^-)=\phi(0^-)+1< 1$ by \eqref{PHI0PHI0M}, \eqref{CONDTHIR} and the hypothesis $\delta\neq \delta(\rho,\alpha)$. Then $[\phi(0),\phi(1^-)]\subset [0,1)$ and the induced euclidean and circle topology in $[\phi(0),\phi(1^-)]$ are the same.
Consider the function $\varphi':[\phi(0),\phi(1^-)]\to [0,1]$ defined by $\varphi'(x)=\inf\{y\in [0,1]: \phi(y)\ge x \}$ for all $x\in [\phi(0),\phi(1^-)]$. Note that $\varphi'|_{[\phi(0),\phi(1^-))}=\tilde \varphi$ and $\varphi'(\phi(1^-))=1$. 
This function $\varphi'$ is also continuous in the euclidean topology.
Let us consider the projection $\pi:\R\to [0,1)$ defined by $\pi(x)=\{x\}$ for all $x\in \R$. Then $\pi\circ \varphi':[\phi(0),\phi(1^-)]\to [0,1)$ is continuous in the circle topology. In particular, $\bar\varphi:=\pi\circ \varphi'|_{\overline \Lambda}:\overline \Lambda\to [0,1)$ is continuous. As $(\pi\circ \varphi')|_{\Lambda}=\pi\circ \tilde \varphi=\tilde\varphi$, it follows that  $\bar\varphi$ is surjective and $\bar\varphi\circ f|_\Lambda=R_\rho\circ \bar\varphi|_\Lambda$ by Lemma \ref{SEMICONJ0}.

Let $x\in \overline \Lambda\setminus \Lambda$. Since $D\cap \Z^-\neq \emptyset$ implies $D\cap \Z^+=\emptyset$, it follows that $\phi(\alpha^-)=a=\phi(a)\in \Lambda$. 
If $\phi(0)$, then $0\in\Lambda$. If $\phi(0)>0$, then $\overline{\Lambda}\subset[\phi(0),\phi(1^-)]\subset(0,1)$ and we have $0\notin\overline{\Lambda}$. Therefore, $f$ is continuous at $x$. So we can conclude as in Case $1$.
\end{proof}

\bigskip
\noindent{\bf Acknowledgments:} This work was supported by the project MATH-AmSud TOMCAT 22-MATH-10 and by the Chilean ANID founds FONDECYT Reg. 1230569 and BECAS Doctorado Nacional 2024-21240193.


\begin{thebibliography}{50}

\bibitem{AB} P. Alessandri and V. Berth\'e, {\it Three distance theorems and combinatorics on words}, L'Enseignement Math\'ematique {\bf 44} (1998) 103-132. 

\bibitem{BS20} J. Bowman and S. Sanderson, {\it Angels' staircases, Sturmian sequences, and trajectories on homothety surfaces}, Journal of Modern Dynamics {\bf 16} (2020) 109-153. 

\bibitem{B06} J. Br\'emont, {\it Dynamics of injective quasi-contractions}, Ergodic. Theory Dynam. Systems {\bf 26} (2006) 19-44.

\bibitem{Brette} R. Brette, {\it Rotation Numbers of Discontinuous Orientation-Preserving Circle Maps}, Set-Valued Analysis {\bf 11} (2003) 359-371. 

\bibitem{B93} Y. Bugeaud, {\it Dynamique de certaines applications contractantes, lin\'eaires par morceaux, sur [0, 1)}, C. R. Acad. Sci. Paris S\'er I Math. {\bf 317} (1993) 575-578.

\bibitem{BC99} Y. Bugeaud and J.-P. Conze, {\it Calcul de la dynamique de transformations lin\'eaires contractantes mod 1 et arbre de Farey}, Acta Arith. {\bf 88} (1999) 201-218.

\bibitem{BuKLN} Y. Bugeaud, D. H. Kim, M. Laurent and  A. Nogueira, {\it On the diophantine nature of the elements of Cantor sets arising in the dynamics of contracted rotations}, The Annali della Scuola Normale di Pisa - Classe di Scienze (2020).

\bibitem{CCG20} A. Calder\'on, E. Catsigeras and P. Guiraud, {\it A spectral decomposition of the attractor of piecewise-contracting maps of the interval}, Ergodic Theory and Dynamical Systems, {\bf 41} (2021) 1940-1960.

\bibitem{CGM20} E. Catsigeras, P. Guiraud,  and A. Meyroneinc, {\it Complexity of injective piecewise contracting interval maps}, Ergodic Theory Dynamical Systems {\bf 40} (2020) 64-88.

\bibitem{Coutinho1999} R. Coutinho, Din\^amica simb\'olica linear, Ph.D Thesis, Instituto Superior T\'ecnico, Universidade T\'ecnica de Lisboa, 1999.

\bibitem{CFLM} R. Coutinho, B. Fernandez, R. Lima and A. Meyroneinc, {\it Discrete time piecewise affine models of genetic regulatory networks}, J. Math. Biol. {\bf 52} (2006) 524-570.

\bibitem{FC91} O. Feely and L. O. Chua, {\it The effect of integrator leak in $\Sigma-\Delta$ modulation}, IEEE Transactions on Circuits and Systems {\bf 38} (1991) 1293-1305.

\bibitem{FP20} F. Fernandes and B. Pires, {\it A switched server system semiconjugate to a minimal interval exchange}, European Journal of Applied Mathematics {\bf 31} (2020) 682-708. 

\bibitem{G25} J. P. Gaiv\~{a}o, {\it Hausdorff Dimension of the Exceptional Set of Interval Piecewise Affine Contractions}, Qual. Theory Dyn. Syst. {\bf 24} (2025)

\bibitem{GP} J. P. Gaiv\~{a}o and B. Pires, {\it Multi-dimensional piecewise contractions are asymptotically periodic}, preprint (2024), arXiv:2405.08444. 

\bibitem{GN22} J. P. Gaiv\~{a}o  and A. Nogueira, {\it Dynamics of piecewise increasing contractions}, Bulletin of the London Mathematical Society {\bf 54} (2022) 482-500.

\bibitem{GAK17} A. Granados, L. Alsed\`a and M. Krupa, {\it The Period Adding and
Incrementing Bifurcations: From Rotation Theory to Applications}, SIAM Review 
{\bf 59} (2017) 225-292.

\bibitem{JL25} S. Jain and C. Liverani, {\it Piecewise contractions}, Ergodic Theory and Dynamical Systems {\bf 45} (2025) 1503-1540.


\bibitem{JO19} S. Janson and A. \"{O}berg, {\it A piecewise contractive dynamical system and Phragm\'en's election method}, Bull. Soc. Math. France {\bf 147} (2019) 395-441.

\bibitem{LN18} M. Laurent and A. Nogueira, {\it Rotation number of contracted rotations}, J. Mod. Dyn., {\bf 12} (2018) 175-191.

\bibitem{LN21} M. Laurent and A. Nogueira, {\it  Dynamics of 2-interval piecewise affine maps and Hecke-Mahler series}, J. Mod. Dyn. {\bf 17} (2021) 33-63.


\bibitem{MSG} K. Matsuyama, I. Sushko and L. Gardini, {\it A piecewise linear model of credit traps and credit cycles: a complete characterization}, Decisions Econ. Finan. {\bf 41} (2018) 119-143.

\bibitem{NP15} A. Nogueira and B. Pires, {\it Dynamics of piecewise contractions of the interval}, Ergodic Theory Dynam. Systems, {\bf 35} (2015), 2198-2215.

\bibitem{NPR18} A. Nogueira, B. Pires and R. A. Rosales, {\it Topological dynamics of piecewise $\lambda$-affine maps}, Ergodic Theory Dynam. Systems, {\bf 38} (2018), 1876-1893.

\bibitem{NS72} J. Nagumo and S. Sato, {\it On a response characteristic of a mathematical neuron model}, Kybernetik, {\bf 10} (1972), 155-164.


\bibitem{P19} B. Pires, {\it Symbolic dynamics of piecewise contractions}, Nonlinearity {\bf 32} (2019) 4871-4889.


\bibitem{P20} B. Pires, {\it Piecewise Contractions and $b$-adic Expansions}
C. R. Math. Rep. Acad. Sci. {\bf 42} (2020) 1-9.

\bibitem{RT86} F. Rhodes and C. Thompson, {\it Rotation Numbers for Monotone Functions on the Circle}, J. London Math. Soc. (2) {\bf 34} (1986) 360-368. 

\bibitem{Rhodes1991} F. Rhodes and C. Thompson, {\it Topologies and rotation numbers for families of monotone functions on the circle}, J. London Math. Soc. (2) {\bf 43} (1991) 156-170.


\end{thebibliography}
\end{document}